\documentclass[10pt]{article}

\usepackage[utf8]{inputenc}
\usepackage{natbib}
\usepackage{amsfonts}
\usepackage{mathtools}
\usepackage{amssymb}
\usepackage{amsmath}
\usepackage{amsthm}
\usepackage{xcolor}
\definecolor{mydarkblue}{rgb}{0,0.08,0.85}
\usepackage[colorlinks=true,bookmarks=true,linkcolor=mydarkblue,urlcolor=mydarkblue,citecolor=mydarkblue,breaklinks=true]{hyperref}
\usepackage{breakcites}
\usepackage{enumitem}
\usepackage{cases}
\usepackage{commath}
\usepackage{algorithm}
\usepackage{algpseudocode}
\usepackage{interval}
\intervalconfig{soft open fences} %
\usepackage{tikz}
\usepackage{tikz-cd}
\usepackage{quiver}
\usetikzlibrary{matrix}

\newtheorem{theorem}{Theorem}[section]
\newtheorem{proposition}[theorem]{Proposition}
\newtheorem{definition}[theorem]{Definition}
\newtheorem{lemma}[theorem]{Lemma}
\newtheorem{corollary}[theorem]{Corollary}
\newtheorem{remark}[theorem]{Remark}

\newcommand{\zeros}{0}

\newcommand{\reals}{{\mathbb R}}

\newcommand{\naturals}{{\mathbb N}}

\DeclareMathOperator*{\argmin}{arg\,min}

\newcommand{\retr}{\mathrm{R}}
\newcommand{\grad}{\nabla}
\newcommand{\hess}{\nabla^2}

\newcommand{\tangent}{\mathrm{T}}
\newcommand{\normal}{\mathrm{N}}
\newcommand{\ptransport}[2]{\Gamma_{#1}^{#2}}

\newcommand{\Exp}{\mathrm{Exp}}
\newcommand{\Log}{\mathrm{Log}}

\DeclareMathOperator{\dist}{dist}

\DeclareMathOperator{\rank}{rank}

\DeclareMathOperator{\vecspan}{span}
\DeclareMathOperator{\proj}{proj}

\DeclarePairedDelimiterX{\inner}[2]{\langle}{\rangle}{#1, #2}

\newcommand{\D}{\mathrm{D}}
\newcommand{\deriv}{\mathrm{d}}

\newcommand{\smooth}[1]{{\mathrm{C}^{#1}}}

\newcommand{\sequence}[1]{\{#1\}}

\makeatletter
\newcommand{\TODOF}[1]{\@bsphack\@esphack}
\makeatother
\newcommand{\inj}{\mathrm{inj}}

\newcommand\restr[2]{{%
    \left.\kern-\nulldelimiterspace %
      #1 %
      \vphantom{\big|} %
    \right|_{#2} %
  }}

\newcommand{\aref}[1]{\hyperref[#1]{A\ref{#1}}}
\newtheorem{assumption}{A\ignorespaces} %

\newcommand{\retrdistboundconst}{c_{\mathrm{r}}}

\newcommand{\linesearchsigma}{\sigma}

\newcommand{\liplikegd}{L}

\newcommand{\cubicpenalty}{\varsigma}
\newcommand{\arcmodel}{m}
\newcommand{\arctheta}{\kappa}
\newcommand{\arcrho}{\varrho}
\newcommand{\arcstepdistconst}{c_1}
\newcommand{\arcquadconvconst}{c_q}
\newcommand{\absmineig}{\Lambda}

\newcommand{\linearmap}{H}
\newcommand{\ratiosteps}{r}
\newcommand{\qgradconst}{L_q}

\newcommand{\rtrmodel}{m}
\newcommand{\rtrrho}{\rho}
\newcommand{\cauchystep}{t^\mathrm{c}}
\newcommand{\rtrstepconstant}{c_s}
\newcommand{\rtrsufficientdecrease}{c_p}

\newcommand{\hesslipconstant}{L_H}
\newcommand{\lyapsufficientdecreaseconst}{\sigma}
\newcommand{\sufficientdecreaseconst}{\omega}

\newcommand{\hessapproxconst}{\beta_H}
\newcommand{\fliplikeconst}{L_H'}

\newcommand{\manifold}{\mathcal{M}}
\newcommand{\nanifold}{\mathcal{N}}

\newcommand{\chillmanifold}{\mathcal{Z}}

\newcommand{\mfc}{f}
\newcommand{\sfc}{g}
\newcommand{\hatmfc}{\hat \mfc}
\newcommand{\taylor}{\tau}

\newcommand{\optimalset}{\mathcal{S}}
\newcommand{\minsubset}{\mathfrak{X}}
\newcommand{\optpoint}{\bar x}

\newcommand{\mfcopt}{f_{\optimalset}}

\newcommand{\deterministicalgorithm}{F}

\newcommand{\rankop}{d}

\newcommand{\ball}{\mathrm{B}}

\newcommand{\leven}{Levenberg--Marquardt}

\newcommand{\causchwarz}{Cauchy--Schwarz}
\newcommand{\loja}{\L ojasiewicz}
\newcommand{\polyakloja}{Polyak--\loja}
\newcommand{\pl}{\ensuremath{\text{P\L}}}
\newcommand{\morsebott}{Morse--Bott}
\newcommand{\mb}{\ensuremath{\text{MB}}}
\newcommand{\kurdykaloja}{Kurdyka--\loja}
\newcommand{\kl}{\ensuremath{\text{K\L}}}
\newcommand{\eb}{\ensuremath{\text{EB}}}
\newcommand{\qg}{\ensuremath{\text{QG}}}

\newcommand{\plconstant}{\mu}

\newcommand{\boundpathlength}{\gamma}
\newcommand{\boundvanishingsteps}{\eta}

\newcommand{\lammax}{\lambda_{\max}}
\newcommand{\lamflat}{\mu^\flat}
\newcommand{\lamsharp}{\lambda^\sharp}
\newcommand{\plexp}{\theta}

\usepackage{sectsty}
\makeatletter\def\@seccntformat#1{\protect\makebox[0pt][r]{\csname the#1\endcsname\hspace{12pt}}}\makeatother

\usepackage[verbose=true,letterpaper]{geometry}
\AtBeginDocument{
  \newgeometry{
    textheight=9in,
    textwidth=6in,
    top=1in,
    headheight=12pt,
    headsep=25pt,
    footskip=30pt
  }
}

\title{
  Fast convergence to non-isolated minima:\\four equivalent conditions for
  $\smooth{2}$ functions
}
\author{
  Quentin Rebjock and Nicolas Boumal\thanks{Correspondence: quentin.rebjock@epfl.ch. Ecole Polytechnique F\'ed\'erale de Lausanne (EPFL), Insitute of Mathematics. This work was supported by the Swiss State Secretariat for Education, Research and Innovation (SERI) under contract number MB22.00027.}
}
\date{\today}

\begin{document}

\maketitle

\begin{abstract}
  Optimization algorithms can see their local convergence rates deteriorate when
  the Hessian at the optimum is singular.
  These singularities are inescapable when the optima are non-isolated.
  Yet, under the right circumstances, several algorithms preserve their
  favorable rates even when optima form a continuum (e.g., due to
  over-parameterization).
  This has been explained under various structural assumptions, including the
  Polyak--{\L}ojasiewicz condition, Quadratic Growth and the Error Bound.
  We show that, for cost functions which are twice continuously differentiable
  ($\mathrm{C}^2$), those three (local) properties are equivalent.
  Moreover, we show they are equivalent to the Morse--Bott property, that is,
  local minima form differentiable submanifolds, and the Hessian of the cost
  function is positive definite along its normal directions.
  We leverage this insight to improve local convergence guarantees for
  safe-guarded Newton-type methods under any (hence all) of the above
  assumptions.
  First, for adaptive cubic regularization, we secure quadratic convergence even
  with inexact subproblem solvers.
  Second, for trust-region methods, we argue capture can fail with an exact
  subproblem solver, then proceed to show linear convergence with an inexact one
  (Cauchy steps).
\end{abstract}

\section{Introduction}\label{sec:intro}

We consider local convergence of algorithms for unconstrained optimization
problems of the form
\begin{align*}
  \min_{x \in \manifold} \mfc(x),
\end{align*}
where $\manifold$ is a Riemannian manifold\footnote{The contributions are
  relevant for $\manifold = \mathbb{R}^n$ too. We treat the more general manifold case as it
  involves only mild overhead in notation, summarized in
  Table~\ref{table:euclidean-case}.} and $\mfc \colon \manifold \to \reals$ is
at least $\smooth{1}$ (continuously differentiable).

When $\mfc$ is $\smooth{2}$ (twice continuously differentiable), the most
classical local convergence results ensure favorable
rates for standard algorithms \emph{provided} they converge to a non-singular local minimum
$\optpoint$, that is, one such that the Hessian $\hess\mfc(\optpoint)$ is
positive definite.
And indeed, those rates can degrade if the Hessian is merely positive
semidefinite.
For example, with $f(x) = x^4$, gradient descent (with an appropriate step-size)
converges only sublinearly to the minimum, and Newton's method converges only
linearly.

This is problematic if the minimizers of $\mfc$ are not isolated, because in
that case the Hessian cannot be positive definite there.
This situation arises commonly in applications for structural reasons such as
over-parameterization, redundant parameterizations and symmetry---see
Section~\ref{subsec:applications}.

Notwithstanding, algorithms often exhibit good local behavior near non-isolated
minimizers.
As early as the 1960s, this has prompted investigations into properties that
such cost functions may satisfy and which lead to fast local rates despite
singular Hessians.
We study four such properties.

In all that follows, we are concerned with the behavior of algorithms in the
vicinity of its local minima.
Since we do not assume that they are isolated, rather than selecting one local
minimum $\optpoint$, we select all local minima of the same value.
Formally, given a local minimum $\optpoint$, let
\begin{align}
    \optimalset & = \{x \in \manifold : x \text{ is a local minimum of $\mfc$ and } \mfc(x) = \mfcopt\}
    \label{eq:def-s}
\end{align}
denote the set of all local minima with a given value $\mfcopt =
\mfc(\optpoint)$.

For $\mfc$ of class $\smooth{2}$, it is particularly favorable if $\optimalset$
is a differentiable submanifold of $\manifold$ around $\optpoint$.
In that case the set $\optimalset$ has a \emph{tangent space}
$\tangent_{\optpoint}\optimalset$ at $\optpoint$.
It is easy to see that each vector $v \in \tangent_{\optpoint}\optimalset$ must
be in the kernel of the Hessian $\hess\mfc$ at $\optpoint$ because the gradient
$\grad\mfc$ is constant (zero) on $\optimalset$.
Thus, $\tangent_{\optpoint}\optimalset \subseteq \ker\hess\mfc(\optpoint)$.
Since $\optpoint$ is a local minimum, we also know that $\hess\mfc(\optpoint)$
is positive semidefinite.
Then, in the spirit of asking the Hessian to be ``as positive definite as possible'', the best we can hope for is that the kernel of $\hess\mfc(\optpoint)$ is
exactly $\tangent_{\optpoint}\optimalset$, in which case the restriction of
$\hess\mfc(\optpoint)$ to the \emph{normal space} $\normal_{\optpoint}\optimalset$,
that is, the orthogonal complement of $\tangent_{\optpoint}\optimalset$ in
$\tangent_{\optpoint}\manifold$, is positive definite.

We call this the \emph{\morsebott{}} property (\mb{}), and we write
$\plconstant$-\mb{} to indicate that the positive eigenvalues are at least
$\plconstant > 0$.
The definition requires $\mfc$ to be twice differentiable.

\begin{definition}\label{def:mb}
  Let $\optpoint$ be a local minimum of $\mfc$ with associated set
  $\optimalset$~\eqref{eq:def-s}.
  We say $\mfc$ satisfies the \emph{\morsebott{}} property at $\optpoint$ if
  \begin{align}
    \optimalset \textrm{ is a $\smooth{1}$ submanifold around } \optpoint && \textrm{ and } &&
    \ker\hess \mfc(\optpoint) = \tangent_{\optpoint}\optimalset.
    \tag{\mb{}}
    \label{eq:morse-bott}
  \end{align}
  If also $\inner{v}{\hess \mfc(\optpoint)[v]} \geq \plconstant\|v\|^2$ for some
  $\plconstant > 0$ and all $v \in \normal_{\optpoint}\optimalset$ then we say $\mfc$
  satisfies \emph{$\plconstant$-\mb{}} at $\optpoint$.
\end{definition}

At first, a reasonable objection to the above is that one may not want to assume
that $\optimalset$ is a submanifold.
Perhaps for that reason, it is far more common to encounter other assumptions in
the optimization literature.
We focus on three: \emph{\polyakloja{}} (\pl{}), \emph{error bound} (\eb{}) and
\emph{quadratic growth} (\qg{}).
The first goes back to the 1960s~\citep{polyak1963gradient}.
The latter two go back at least to the
1990s~\citep{luo1993error,bonnans1995second}.

Below, the first two definitions (as stated) require $\mfc$ to be differentiable.
The distance to a set is defined as usual: $\dist(x, \optimalset) =
\inf_{y\in\optimalset} \dist(x, y)$ where $\dist(x, y)$ is the Riemannian
distance on $\manifold$.

\begin{definition}\label{def:pl-eb-qg}
  Let $\optpoint$ be a local minimum of $\mfc$ with associated set
  $\optimalset$~\eqref{eq:def-s}.
  We say $\mfc$ satisfies
  \begin{itemize}
  \item the \emph{\polyakloja{}} condition with constant $\plconstant > 0$
    ($\plconstant$-\pl{}) around $\optpoint$ if
    \begin{align}
      \mfc(x) - \mfcopt \leq \frac{1}{2\plconstant}\|\grad \mfc(x)\|^2;
      \tag{\pl{}}
      \label{eq:local-pl}
    \end{align}
  \item the \emph{error bound} with constant $\plconstant > 0$
    ($\plconstant$-\eb{}) around $\optpoint$ if
    \begin{align}
      \plconstant \dist(x, \optimalset) \leq \|\grad \mfc(x)\|;
      \tag{$\eb$}
      \label{eq:error-bound}
    \end{align}
  \item \emph{quadratic growth} with constant $\plconstant > 0$
    ($\plconstant$-\qg{}) around $\optpoint$ if
    \begin{align*}
      \mfc(x) - \mfcopt \geq \frac{\plconstant}{2}\dist(x, \optimalset)^2;
      \tag{QG}
      \label{eq:quadratic-growth}
    \end{align*}
  \end{itemize}
  all understood to hold for all $x$ in some neighborhood of $\optpoint$.
\end{definition}

Note that all the definitions are \emph{local} around a \emph{point}
$\optpoint$.
Two observations are immediate: \emph{(i)} \qg{} implies that $\optpoint$ is a
strict minimum relatively to $\optimalset$~\eqref{eq:def-s}, meaning that
$\mfc(x) > \mfcopt$ for all $x \notin \optimalset$ close enough to $\optpoint$,
and \emph{(ii)} both \eb{} and \pl{} imply that critical points and
$\optimalset$ coincide around $\optpoint$.
Thus, both \eb{} and \pl{} rule out existence of saddle points near $\optpoint$.
By extension, we say that $\mfc$ satisfies any of these four properties around a
\emph{set} of local minima if it holds around each point of that set.

\subsection{Contributions}

A number of relationships between \pl{}, \eb{} and \qg{} are well known already
for $\mfc$ of class $\smooth{1}$: see Table~\ref{table:results-summary} and
Section~\ref{sec:relatedwork}.
Our first main contribution in this paper is to show that:
\begin{quote}
    \centering
    \emph{If $\mfc$ is of class $\smooth{2}$, then \pl{}, \eb{}, \qg{} and \mb{}
      are essentially equivalent.}
\end{quote}
Here, ``essentially'' means that the constant $\plconstant$ may degrade (arbitrarily
little) and the neighborhoods where properties hold may shrink.
Notably, we show that if $f$ is $\smooth{p}$ with $p \geq 2$, then  \pl{}, \eb{}
and \qg{} all imply that the set of local minima $\optimalset$
is locally smooth (at least $\smooth{p - 1}$).
We also give counter-examples when $\mfc$ is only $\smooth{1}$.
Explicitly, in Section~\ref{sec:equiv-properties} we summarize known results for
$\mfc$ of class $\smooth{1}$ and we contribute the following:
\defcitealias{otto2000generalization}{Otto \& Villani}
\defcitealias{ioffe2000metric}{Ioffe}
\defcitealias{corvellec2008nonlinear}{Corvellec \& Motreanu}
\begin{table}
  \centering
  \begin{tabular}{|c|c|c|c|c|c|}
    \hline
                              & Statement                            & $\mfc$ is $\smooth{p}$  & Constants                                & Global? & Comments\\ \hline
    \pl{} $\Rightarrow$ \qg{} & Prop.~\ref{prop:pl-implies-qg}       & $p \geq 1$              & $\plconstant' = \plconstant$             & Yes     & \citetalias{otto2000generalization}, \citetalias{ioffe2000metric} [2000]\\ \hline
    \mb{} $\Rightarrow$ \qg{} & Prop.~\ref{prop:mb-implies-qg}       & $p \geq 2$              & $\plconstant' < \plconstant$             & n/a     &  Taylor expansion\\ \hline
    \eb{} $\Rightarrow$ \pl{} & Prop.~\ref{prop:eb-implies-pl}       & $p \geq 2$              & $\plconstant' < \plconstant$             & No      & \\
                              & Rmrk~\ref{remark:eb-implies-pl-c1}   & $p \geq 1$              & $\plconstant' = \frac{\plconstant^2}{L}$ & Yes     & History in Section~\ref{par:relationships-properties}\\ \hline
    \qg{} $\Rightarrow$ \eb{} & Prop.~\ref{prop:grad-dist-bounds}    & $p \geq 2$              & $\plconstant' < \plconstant$             & No      & Fails for $p = 1$, Rmrk~\ref{remark:qg-implies-eb-c1}\\ \hline
    \pl{} $\Rightarrow$ \eb{} & Rmrk~\ref{remark:pl-implies-eb-c1}   & $p \geq 1$              & $\plconstant'$ = $\plconstant$           & Yes     & History in Section~\ref{par:relationships-properties}\\ \hline
    \pl{} $\Rightarrow$ \mb{} & Cor.~\ref{cor:pl-implies-mb}         & $p \geq 2$              & $\plconstant' = \plconstant$             & n/a     &  Fails for $p = 1$, Rmrk~\ref{remark:structure-only-c1}\\ \hline
    \eb{} $\Rightarrow$ \qg{} &                                      & $p \geq 1$              & $\plconstant' = \plconstant$             & Yes     & \citetalias{corvellec2008nonlinear} [2008]\\ \hline
  \end{tabular}
  \caption{Summary of implications.
    For a row $\mathrm{A} \Rightarrow \mathrm{B}$, we mean that $\mathrm{A}$
    with constant $\plconstant$ implies $\mathrm{B}$ (in a possibly different
    neighborhood) with constant $\plconstant'$.
    The condition $\plconstant' < \plconstant$ means $\plconstant'$ can be taken
    arbitrarily close to $\plconstant$ by shrinking the neighborhood.
    For \eb{} $\Rightarrow$ \pl{} with $p = 1$, we let $L$ denote the Lipschitz
    constant of $\grad \mfc$.
    If $\mathrm{B}$ holds globally when $\mathrm{A}$ holds globally, we write ``Yes'' in the column ``Global?''
    Implications we found in the literature are marked with a reference.
}\label{table:results-summary}
\end{table}
\begin{itemize}
\item Theorem~\ref{th:pl-implies-submanifold} shows \pl{} around $\optpoint$ implies
  $\optimalset$ is a $\smooth{p - 1}$ submanifold around $\optpoint$ if $f$ is
  $\smooth{p}$ with $p \geq 2$.
  Remark~\ref{remark:structure-only-c1} provides counter-examples if $\mfc$ is
  only $\smooth{1}$.
  If $\mfc$ is analytic, so is $\optimalset$, as already shown
  by~\cite{feehan2020morse}.
\item Lemma~\ref{lemma:grad-image-hess} is instrumental to prove
  Theorem~\ref{th:pl-implies-submanifold} and to analyze algorithms.
  It states that under \pl{} the gradient locally aligns with the dominant
  eigenvectors of the Hessian.
\item Corollary~\ref{cor:pl-implies-mb} deduces that $\plconstant$-\pl{} implies
  $\plconstant$-\mb{} if $\mfc$ is $\smooth{2}$.
\item Proposition~\ref{prop:grad-dist-bounds} shows $\plconstant$-\qg{} implies
  $\plconstant'$-\eb{} with $\plconstant'<\plconstant$ arbitrarily close if
  $\mfc$ is $\smooth{2}$.
  Remark~\ref{remark:qg-implies-eb-c1} provides a counter-example if $\mfc$ is
  only $\smooth{1}$.
\item Proposition~\ref{prop:eb-implies-pl} shows $\plconstant$-\eb{} implies
  $\plconstant'$-\pl{} with $\plconstant'<\plconstant$ arbitrarily close if
  $\mfc$ is $\smooth{2}$.
  If $\mfc$ is only $\smooth{1}$ with $L$-Lipschitz continuous $\grad \mfc$,
  \citet{karimi2016linear} showed the same with $\plconstant' =
  \plconstant^2/L$.
\end{itemize}

In Section~\ref{sec:super-linear-conv},
we study the classical globalized versions of Newton's method.
We strengthen their local convergence guarantees
when minimizers are not isolated but the $\smooth{2}$ cost function
satisfies
any (hence all) of
the above.
The key observation that enables those improvements is
the fact that we may use all four conditions in the analysis
without loss of generality.
Specifically:
\begin{itemize}
\item Cubic regularization enjoys superlinear convergence under \pl{}, as shown
  by \citet{nesterov2006cubic}.
  \citet{yue2019quadratic} further showed quadratic convergence under \eb{}.
  Both references assume an exact subproblem solver.
  Leveraging our results above,
  we show that quadratic convergence still holds with inexact
  subproblem solvers (Theorem~\ref{th:arc-main-theorem}).
\item For the trust-region method with exact subproblem solver, we were
  surprised to find that even basic capture-type convergence properties can fail
  in the presence of non-isolated local minima
  (Section~\ref{par:shortcomings-exact-solver}).
  Notwithstanding, common implementations of the trust-region method use a
  truncated conjugate gradient (tCG) subproblem solver, and those do,
  empirically, exhibit superlinear convergence under the favorable conditions
  discussed above.
  We discuss this further in Remark~\ref{remark:more-than-cauchy-point}, and we
  show a partial result in Theorem~\ref{th:main-rtr}, namely, that using the
  Cauchy step (i.e., the first iterate of tCG) yields linear convergence.
\end{itemize}

As is classical, to prove the latter results, we rely on capture theorems and
Lyapunov stability.
Those hold under assumptions of vanishing step-sizes and bounded path length.
In Section~\ref{sec:reminders}, we state those building blocks succinctly,
adapted to accommodate non-isolated local minima.

\subsection{Non-isolated minima in applications}\label{subsec:applications}

We now illustrate how optimization problems with continuous sets of minima occur
in applications.
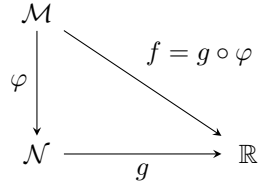
\begin{figure}[!h]
  \centering
  \begin{tikzpicture}
    \matrix (m) [matrix of math nodes,row sep=4em,column sep=6em,minimum width=2em]
    {
      \manifold & \phantom{a}\\
      \nanifold & \reals \\};
    \path[-stealth]
    (m-1-1) edge node [left] {$\varphi$} (m-2-1)
    (m-2-1) edge node [below] {$\sfc$} (m-2-2)
    (m-1-1) edge node [shift={(0.75,0.4)}] {$\mfc = \sfc \circ \varphi$} (m-2-2)
    ;
  \end{tikzpicture}
  \caption{Optimization through the map $\varphi$.}
  \label{fig:lift-diagram}
\end{figure}
In all three scenarios below, we can cast the cost function $\mfc \colon
\manifold \to \reals$ as a composition of some other function $\sfc \colon
\nanifold \to \reals$ through a map $\varphi \colon \manifold \to
\nanifold$, where $\nanifold$ is a smooth manifold (see
Figure~\ref{fig:lift-diagram} and \citep{levin2024lifts}).
If $\sfc$ and $\varphi$ are twice differentiable, then the \morsebott{}
property~\eqref{eq:morse-bott} for $\mfc = \sfc \circ \varphi$ can come about as follows.
Consider a local minimum $\bar{y}$ for $\sfc$.
The set $\minsubset = \varphi^{-1}(\bar{y})$ consists of local minima for $\mfc$.
Pick a point $\optpoint \in \minsubset$.
Assume $x \mapsto \rank\D\varphi(x)$ is constant in a neighborhood of
$\optpoint$.
Then, the set $\minsubset$ is an embedded submanifold of $\manifold$ around
$\optpoint$ with tangent space $\ker\D\varphi(\optpoint)$.
Moreover, the Hessians of $\mfc$ and $\sfc$ at $\optpoint$ are related by
\begin{align}\label{eq:hess-upstairs}
  \hess \mfc(\optpoint) = \D \varphi(\optpoint)^* \circ \hess \sfc(\varphi(\optpoint)) \circ \D \varphi(\optpoint).
\end{align}
Therefore, if $\hess \sfc(\varphi(\optpoint))$ is positive definite, then $\ker
\hess \mfc(\optpoint) = \tangent_{\optpoint}\minsubset$ and $\hess
\mfc(\optpoint)$ is positive definite along the orthogonal complement.
In other words: $\mfc$ satisfies the \morsebott{} property~\eqref{eq:morse-bott}
at $\optpoint$.
We present below a few concrete examples of optimization problems where this can
happen.

\paragraph{Over-parameterization and nonlinear regression.}

Consider minimizing $\mfc(x) = \frac{1}{2} \|F(x) - b\|^2$ with $F \colon
\reals^m \to \reals^n$ a $\smooth{2}$ function.
We cast this as above with $g(y) = \frac{1}{2} \|y - b\|^2$ and $\varphi = F$.
Suppose $\minsubset = \varphi^{-1}(b) = \{x : F(x) = b\}$ is non-empty
(interpolation regime), which is typical in deep learning.
This is the set of global minimizers of $f$.
If $\rank \D F(x)$ is equal to a constant $r$ in a neighborhood of $\minsubset$,
then $\minsubset$ is a smooth submanifold of $\reals^m$ of dimension $m - r$.
If additionally the problem is over-parameterized, that is, $m > n \geq r$, then
$\minsubset$ has positive dimension (see Figure~\ref{fig:plots} for an
illustration).
The discussion above immediately implies that $\mfc$ satisfies \mb{} on
$\minsubset$.
See also~\citet[\S4.2]{nesterov2006cubic} who argue that \pl{} holds in this
setting.
\begin{figure}
  \centering
  \includegraphics[clip, trim=3.5cm 4.55cm 3.5cm 5.4cm, width=1.00\textwidth]{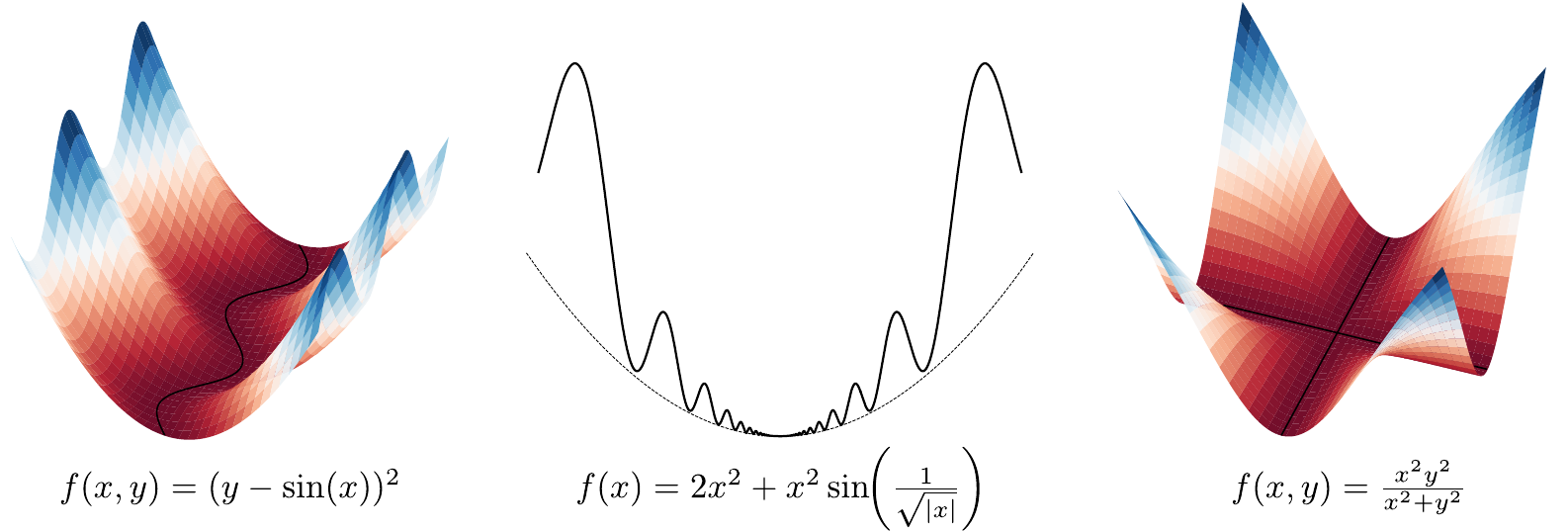}
  \caption{\emph{(Left)} Submanifold of minima where \mb{} holds.
    \emph{(Middle)} A $\smooth{1}$ function that satisfies \qg{} but not \pl{}
    nor \eb{}.
    \emph{(Right)} A $\smooth{1}$ function that satisfies \pl{}, yet whose
    set of minima is a cross.}\label{fig:plots}
\end{figure}

\paragraph{Redundant parameterizations and submersions.}\label{par:submersions}

Say we want to minimize $\sfc \colon \nanifold \to \reals$ constrained to
$\mathcal{C} \subseteq \nanifold$.
If $\mathcal{C}$ is complicated, and if we have access to a
parameterization $\varphi$ for that set (so that $\varphi(\manifold) =
\mathcal{C}$), it may be advantageous to minimize $\mfc = \sfc \circ
\varphi$ instead.
If the parameterization is redundant, this can cause $\mfc$ to have non-isolated
minima, even if the minima of $\sfc$ are isolated.

As an example, consider minimizing $\sfc \colon \reals^{m \times n} \to \reals$
over the bounded-rank matrices $\mathcal{C} = \{Y \in \reals^{m \times
  n} : \rank Y \leq r\}$.
A popular approach consists in lifting the search space to $\manifold =
\reals^{m \times r} \times \reals^{n \times r}$ and minimizing $\mfc = \sfc
\circ \varphi$, where $\varphi \colon \manifold \to \nanifold$ is defined as
$\varphi(L, R) = LR^\top$.
The parameterization is redundant because $\varphi(LJ^{-1}, RJ^{\top}) = \varphi(L,
R)$ for all invertible $J$.
In particular, given a local minimum $Y \in \mathcal{C}$ of $\sfc$, the
fiber $\varphi^{-1}(Y)$ is unbounded, which hinders convergence analyses
(see~\citep{levin2021finding}).
However, if $Y$ is of \emph{maximal rank} $r$ then $\D\varphi$ has constant rank
in a neighborhood of $\varphi^{-1}(Y)$.
From the discussion above, it follows that $\mfc$ satisfies \mb{} on
$\varphi^{-1}(Y)$ if the (Riemannian) Hessian of $g$ (restricted to the manifold
of rank-$r$ matrices) is positive definite.

Similarly, \cite{burer2003nonlinear,burer2005local} introduced a popular approach
to minimize a function $\sfc$ over the set of positive semidefinite matrices of
bounded rank through the map $\varphi \colon Y \mapsto
YY^\top$.
The resulting function $\mfc = \sfc \circ \varphi$ can have non-isolated minima.
However, the same arguments as above ensure that \mb{} holds at minimizers of
maximal rank when $\sfc$ is strongly convex (this setting is for
example considered in~\citep{zhang2022improved}).
This further extends to tensors~\citep{li2022local}.

\paragraph{Symmetries and quotients.}

Some optimization problems have intrinsic symmetries.
For example, in estimation problems, if the measurements are invariant under
particular transformations of the signal, then the signal can only be retrieved
up to those transformations.
The likelihood function then has symmetries, and possibly a continuous set of
optima as a result.
Sometimes, factoring these symmetries out (that is, passing to the quotient)
yields a quotient manifold, and we can investigate optimization on that
manifold~\citep{absil2009optimization}.
In the notation of our general framework above, $\varphi$ is then the quotient
map.
In particular, $\varphi$ is a submersion, so that if $\bar{y} \in \nanifold$ is
a non-singular minimum of $\sfc$ then $\mfc$ satisfies \mb{} on
$\varphi^{-1}(\bar{y})$ (which is a submanifold of dimension $\dim \manifold -
\dim \nanifold$).
See also~\citep[\S9.9]{boumal2020introduction} for the case where $\nanifold$ is
a \emph{Riemannian} quotient of $\manifold$.

\subsection{Related work}
\label{sec:relatedwork}

\paragraph{Historical note.}
Discussions about convergence to singular minima appear in the literature at
least as early as~\cite[\S6.1]{polyak1987introduction}.
\cite{luo1993error} introduced the \eb{} condition explicitly to study gradient
methods around singular minima.
The QG property is arguably as old as optimization, though the earliest work we
could locate is by~\cite{bonnans1995second}.
They employed \qg{} to understand complicated landscapes with non-isolated
minima.
\cite{lojasiewicz1963propriete,lojasiewicz1982trajectoires} introduced his
inequalities and used them subsequently to analyze gradient flow trajectories.
Specifically, he proved that for analytic
functions the trajectories either converge to a point or diverge.
Concurrently,~\cite{polyak1963gradient} introduced what became known as the
\polyakloja{} (\pl{}) variant (also ``gradient dominance'') to study both
gradient flows and discrete gradient methods.
Later,~\cite{kurdyka1998gradients} developed generalizations now known as
\kurdykaloja{} (\kl{}) inequalities.
They are satisfied by most functions encountered in practice, as discussed
in~\citep[\S4]{attouch2010proximal}.
In contrast, the \morsebott{} property has received little attention in the
optimization literature.
Early work by~\citet{shapiro1988perturbation} analyzes perturbations of
optimization problems assuming a property similar to \mb{}.
There is also a mention of gradient flow under MB
in~\citep[Prop.~12.3]{helmke1996optimization}.

\paragraph{Relationships between properties.}\label{par:relationships-properties}

Several articles have explored the interplay between \mb{}, \pl{}, \eb{} and
\qg{} in the last decades.
The implication \pl{} $\Rightarrow$ \qg{} has a rich history.
It can be obtained as a corollary from~\citep[Basic lemma]{ioffe2000metric}
(based on Ekeland's variational principle; see
also~\citep[Lem.~2.5]{drusvyatskiy2015curves}).
It also follows from \loja{}-type arguments that consist in
bounding the length of gradient flow
trajectories~\cite[Prop.~1]{otto2000generalization}.
Likewise, \citet{bolte2010characterizations} study growth under \kl{}
inequalities, with \pl{} $\Rightarrow$ \qg{} as a special case.
For lower-semicontinuous functions,~\citet[Thm.~4.2]{corvellec2008nonlinear}
show that $\plconstant$-\eb{} $\Rightarrow$ $\plconstant$-\qg{}.
They also find that $\plconstant$-\qg{} $\Rightarrow$
$\frac{\plconstant}{2}$-\eb{} for convex functions~\cite[Prop.~5.3]{corvellec2008nonlinear}.
In a somewhat different setting, \citet[Cor.~3.2]{drusvyatskiy2013second} prove
that $\plconstant$-\eb{} $\Rightarrow$ $\plconstant'$-\qg{} with arbitrary
$\plconstant' < \plconstant$.
With extra assumptions, they also show $\plconstant$-\qg{} $\Rightarrow$
$\plconstant'$-\eb{} but without control of $\plconstant'$.
Later,~\cite{bolte2017fromerror} proved an equivalence between \kl{}
inequalities and function growth for convex and potentially non-smooth
functions.
Their results seem to generalize to semi-convex functions.
See also~\citep{zhang2017therestricted} and~\citep{drusvyatskiy2018error} for
equivalences between \eb{} and \qg{}.
\citet{karimi2016linear} established implications between several properties
encountered in the optimization literature.
In particular, they also show that \pl{} and \eb{} are equivalent, and that they
both imply \qg{}.
\citet{liao2024error} extended this to a non-smooth and weakly convex setting.
\cite{li2018calculus} showed that \eb{} implies \pl{} for non-smooth functions
under a level set separation assumption, though with no control on the \pl{}
constant.
Implications between \pl{} and \eb{} are also reported in~\cite[Thm.~3.7,
Prop.~3.8]{drusvyatskiy2021nonsmooth} for non-smooth functions under
broad conditions.
In the context of functional analysis,~\cite{feehan2020morse} proved that \mb{}
and \pl{} are equivalent for \emph{analytic} functions defined on Banach spaces.
The work of~\cite[Ex.~2.9]{wojtowytsch2023stochastic} also mentions that
\mb{} implies \pl{} for $\smooth{2}$ functions.
A more general implication is given by \citet[Prop.~1]{arbel2022non} for
parameterized optimization problems.
Previously,~\cite{bonnans1995second} had exhibited sufficient conditions
(similar to \mb{}) for \qg{} to hold.
As a side note,~\cite{marteau2024second} proved that \mb{} is a sufficient
condition to ensure that a non-negative function is globally decomposable as a sum of
squares of smooth functions.

\paragraph{Convergence guarantees.}

The error bound approach of~\cite{luo1993error} has proven to be fruitful as
multiple analyses based on this condition followed.
Notably,~\cite{tseng2000error} proved local superlinear convergence rates for
some Newton-type methods applied to systems of nonlinear equations.
They relied specifically on \eb{} and did not assume isolated minima.
Later,~\cite{yamashita2001rate} employed \eb{} to establish capture theorems and
quadratic convergence rates for the \leven{} method.
\cite{fan2005quadratic},~\cite{behling2019local} and~\cite{boos2024levenberg}
generalized their results (in particular to solutions with a non-zero residual).
More recently,~\cite{bellavia2015strong} found that two adaptive regularized
methods converge quadratically for nonlinear least-squares problems (assuming
that \eb{} holds).

An early work of~\cite{anitescu2000degenerate} combines the \qg{} property with
other conditions to ensure isolated minima in constrained optimization, then
deducing convergence results.
Later, \qg{} has found applications mainly in the context of convex
optimization: see~\citep{liu2015asynchronous} for coordinate
descent,~\citep{necoara2019linear} for various gradient methods,
and~\citep{drusvyatskiy2018error} for the proximal gradient method.
It is also worth mentioning that the definition of \qg{} does not require
differentiability of the function.
For this reason, \qg{} is valuable to study algorithms in
non-smooth optimization too~\citep{davis2024local,lewis2024identifiability}.

The literature about convergence results based on \loja{} inequalities is vast,
and we touch here on some particularly relevant references.
\cite{absil2005convergence} discretized the arguments from
\citep{lojasiewicz1982trajectoires} and obtained capture results for a broad
class of optimization algorithms.
\cite{lageman2007convergence,lageman2007pointwise} provided generalizations to
broader classes of functions.
Later, such arguments have been used in many contexts to prove algorithmic
convergence guarantees, among
which~\citep{attouch2010proximal,bolte2014proximal} are particularly influential
works.
Moreover,~\cite{attouch2013convergence} proposed a general abstract framework
based on \kl{} to derive capture results and convergence rates,
and~\cite{frankel2015splitting} extended their statements.
See also~\citep{necoara2020general} for a framework that encompasses
higher-order methods.
\cite{li2018calculus} studied the preservation of \loja{}
inequalities under function transformations (such as sums and compositions).
See also~\citep{bassily2018exponential,terjek2022framework} for the preservation
of \pl{} through function compositions.
Interestingly, the \pl{} condition is known to be a necessary condition for
gradient descent to converge
linearly~\cite[Thm.~5]{abbaszadehpeivasti2023conditions}.
Recently,~\cite{yue2023lower} proved that acceleration is impossible to minimize
globally \pl{} functions: gradient descent is optimal absent further structure.
Assuming \pl{},~\cite{stonyakin2023stopping} formulated stopping criteria for
gradient methods when the gradient is corrupted with noise.
Using \kl{} inequalities,~\cite{noll2013convergence} and~\cite{khanh2022inexact}
analyzed the convergence of line-search gradient descent and trust-region
methods.
\loja{} inequalities have also proved relevant for the study of
second-order algorithms and superlinear convergence rates.
A prominent example is the regularized Newton algorithm that converges
superlinearly when \pl{} holds, as shown in~\citep{nesterov2006cubic}.
More recently,~\cite{zhou2018convergence,yue2019quadratic} provided finer
analyses of this algorithm, respectively assuming \loja{} inequalities and EB.
\citet{qian2022superlinear} extended the abstract framework
of~\cite{attouch2013convergence} to establish superlinear convergence rates.

Stochastic algorithms have also been extensively studied through \loja{}
inequalities, and we briefly mention a few references here.
\cite{dereich2023central,dereich2021convergence} analyzed stochastic gradient
descent (SGD) in the presence of non-isolated minima, using \loja{} inequalities
among other things.
Local analyses of SGD using \pl{} inequalities are given by~\cite{li2022what}
and~\cite{wojtowytsch2023stochastic}.
\citet{ko2023local} studied the local stability and convergence of
SGD in the presence of a compact set of minima with a condition that is weaker
than \pl{}.
As for second-order algorithms,~\cite{masiha2022stochastic} proved that a
stochastic version of regularized Newton has fast convergence under \pl{}.

\loja{} inequalities are also particularly suited to analyze the convergence of
flows.
Notably,~\cite{lojasiewicz1982trajectoires} bounded the path length of gradient
flow trajectories and~\cite{polyak1963gradient} derived linear convergence
of flows assuming \pl{}.
Related results for flows but under \mb{} are claimed
in~\citep[Prop.~12.3]{helmke1996optimization}.
More recently,~\cite{apidopoulos2021convergence} considered the Heavy-Ball
differential equation and deduced convergence guarantees from \pl{}.
\cite{wojtowytsch2024stochastic} studied a continuous model for SGD and the
impact of the noise on the trajectory.

Finally, we found only few convergence results based on \mb{} in the optimization
literature.
\cite{fehrman2020convergence} derive capture theorems and asymptotic
sublinear convergence rates for gradient descent assuming that \mb{} holds on
the set of minima.
They also provide probabilistic bounds for stochastic variants.
\cite{usevich2020approximate} consider optimization problems over unitary
matrices.
They propose sufficient conditions for \mb{} to hold at a local optimum and then
exploit the induced \pl{} condition to obtain convergence rates.
In order to solve systems of nonlinear equations,~\cite{zeng2023newton}
proposes a Newton-type method that is robust to non-isolated solutions.
It enjoys local quadratic convergence assuming a \mb{}-type property.
The algorithm requires the knowledge of the dimension of the set of solutions.

\paragraph{Applications.}
Non-isolated minima arise in all sorts of optimization problems.
It is common for non-convex inverse problems to have continuous symmetries,
hence non-isolated minima (see~\citep{zhang2020symmetry}).
In the context of deep learning,~\cite{cooper2021global} proved that the set of
global minima of a sufficiently over-parameterized neural network is a smooth
manifold.
In the last decade, there has been a renewed interest in \loja{} inequalities
because they are compatible with these complicated non-convex landscapes.
In particular, a whole line of research exploits them to understand deep
learning problems specifically.
As an example,~\cite{oymak2019overparameterized} employed \pl{} to analyze the
path taken by (stochastic) gradient descent in the vicinity of minimizers.
Several other works suggested that non-convex machine learning loss landscapes
can be understood in over-parameterized regimes through the lens of \loja{}
inequalities~\citep{bassily2018exponential,belkin2021fit,liu2022loss,terjek2022framework}.
Specifically, they argue that \pl{} holds on a significant part of the search
space and analyze (stochastic) gradient methods.
\cite{chatterjee2022convergence} also establishes local convergence
results for a large class of neural networks with \pl{} inequalities.

\subsection{Notation and geometric preliminaries}

\begin{table}[]
  \centering
  \begin{tabular}{|c|c|c|c|c|c|c|}
    \hline
     & $\Exp_x(s)$ & $\Log_x(y)$ & $\ptransport{x}{y}$, $\ptransport{v}{}$ & $\dist(x, y)$ & $\inj(x)$\\\hline
    $\manifold = \reals^n$ & $x + s$ & $y - x$ & identity & $\|x - y\|$ & $+\infty$\\\hline
  \end{tabular}
  \caption{Simplifications in the case where $\manifold$ is a Euclidean
    space.
    Here, $\ptransport{x}{y}$ denotes parallel transport along the minimizing geodesic
    between $x$ and $y$ (assuming the points are close enough).}\label{table:euclidean-case}
\end{table}
This section anchors notation and some basic geometric facts.
In the important case where $\manifold = \reals^n$, several objects reduce as
summarized in Table~\ref{table:euclidean-case}.

We let $\inner{\cdot}{\cdot}$ denote the inner product on
$\tangent_x\manifold$---it may depend on $x \in \manifold$, but the base point
is always clear from context.
The associated norm is $\|v\| = \sqrt{\inner{v}{v}}$.
The map $\dist \colon \manifold \times \manifold \to \reals_+$ is the Riemannian
distance on $\manifold$.
We let $\ball(x, \delta)$ denote the open ball of radius $\delta$ around $x \in
\manifold$.
The tangent bundle is $\tangent\manifold = \{ (x, v) : x \in \manifold \textrm{
  and } v \in \tangent_x\manifold \}$.

Moving away from $x \in \manifold$ along the geodesic with (sufficiently small)
initial velocity $v \in \tangent_x\manifold$ for unit time produces the point
$\Exp_x(v) \in \manifold$ (Riemannian exponential).
The injectivity radius at $x$ is $\inj(x) > 0$.
It is defined such that, given $y \in \ball(x, \inj(x))$, there exists a unique
smallest vector $v \in \tangent_x\manifold$ for which $\Exp_x(v) = y$.
We denote this $v$ by $\Log_x(y)$ (Riemannian logarithm).
Additionally, given $x \in \manifold$ and $y \in \ball(x, \inj(x))$, we let
$\ptransport{x}{y} \colon \tangent_x\manifold \to \tangent_y\manifold$ denote
parallel transport along the unique minimizing geodesic between $x$ and $y$.
If $v = \Log_x(y)$, we also let $\ptransport{v}{} = \ptransport{x}{y}$.

Let $\minsubset$ be a subset of $\manifold$.
We need the notions of tangent and normal cones to $\minsubset$, defined below.
\begin{definition}\label{def:tangent-normal-cones}
  The \emph{tangent cone} to a set $\minsubset$ at $x \in \minsubset$ is the closed set
  \begin{align*}
    \tangent_x\minsubset = \left\{\lim_{k \to +\infty}\frac{1}{t_k}\Log_x(x_k) \,\Big|\, x_k \in \minsubset, t_k > 0 \text{ for all $k$}, x_k \to x, t_k \to 0 \right\}.
  \end{align*}
  We also let $\normal_x\minsubset = \big\{w \in \tangent_x\manifold :
  \inner{w}{v} \leq 0 \text{ for all $v \in \tangent_x\minsubset$}\big\}$ denote
  the normal cone to $\minsubset$ at $x$.
\end{definition}
When $\minsubset$ is a submanifold of $\manifold$ around $x$, the cones
$\tangent_x\minsubset$ and $\normal_x\minsubset$ reduce to the tangent and
normal spaces of $\minsubset$ at $x$.
(By ``submanifold'', we always mean \emph{embedded} submanifold.)

Given $x \in \manifold$, we let $\dist(x, \minsubset) = \inf_{y \in \minsubset}
\dist(x, y)$ denote the distance of $x$ to $\minsubset$.
We further let $\proj_\minsubset(x)$ denote the set of minima of the
optimization problem $\min_{y \in \minsubset} \,\dist(x, y)$.
If this set is non-empty (which is the case in particular if $\minsubset$ is
closed), then we have:
\begin{align}
    \forall \optpoint \in \minsubset, y \in \proj_\minsubset(x), &&
    \dist(y, \optpoint) \leq 2 \dist(x, \optpoint).
    \label{eq:proj-dist-bound}
\end{align}
Indeed, the triangle inequality yields $\dist(y, \optpoint) \leq \dist(x, y) +
\dist(x, \optpoint)$, and $\dist(x, y) = \dist(x, \minsubset) \leq \dist(x,
\optpoint)$.
Moreover, if $y \in \proj_\minsubset(x)$ with $\dist(x, y) < \inj(y)$ then
$\Log_y(x) \in \normal_y\minsubset$.

The set of local minima $\optimalset$ defined in~\eqref{eq:def-s} may not
be closed: consider for example the function $\mfc(x) =
\mathrm{sgn}(x)\exp(-\frac{1}{x^2})(1 + \sin(\frac{1}{x^2}))$ with $\mfcopt =
0$.
It follows that the projection onto $\optimalset$ may be empty.
Notwithstanding, the following holds:
\begin{lemma}\label{lemma:proj-non-empty}
  Around each $\optpoint \in \optimalset$
  there exists a neighborhood in which $\proj_\optimalset$ is
  non-empty.
\end{lemma}
\begin{proof}
  Let $\mathcal{U}$ be an open neighborhood of $\optpoint$ such that $\mfc(x)
  \geq \mfcopt$ for all $x \in \mathcal{U}$.
  Let $\mathcal{V}_1, \mathcal{V}_2 \subset \mathcal{U}$ be two closed balls
  around $\optpoint$ of radii $\delta > 0$ and $\frac{1}{4}\delta$ respectively.
  Then $\optimalset \cap \mathcal{V}_1 = \mfc^{-1}(\mfcopt) \cap \mathcal{V}_1$,
  showing that $\optimalset \cap \mathcal{V}_1$ is closed and the projection
  onto this set is non-empty.
  Let $x \in \mathcal{V}_2$ and $y \in \proj_{\optimalset \cap
    \mathcal{V}_1}(x)$.
  Then $\dist(x, y) \leq \frac{1}{4}\delta$.  %
  Moreover, for all $y' \in \optimalset \setminus \mathcal{V}_1$ we have
  $\dist(x, y') \geq \frac{3}{4}\delta$.
  It follows that $\proj_{\optimalset}(x) = \proj_{\optimalset \cap
    \mathcal{V}_1}(x)$, and this is non-empty.
\end{proof}

From these considerations we deduce that the projection onto $\optimalset$ is
always locally well behaved.

\begin{lemma}\label{lemma:log-well-defined}
  Let $\optpoint \in \optimalset$.
  There exists a neighborhood $\mathcal{U}$ of $\optpoint$ such that for all $x
  \in \mathcal{U}$ the set $\proj_{\optimalset}(x)$ is non-empty, and for all $y
  \in \proj_{\optimalset}(x)$ we have $\dist(x, y) < \inj(y)$.
  In particular, $v = \Log_y(x)$ is well defined and $v \in
  \normal_y\optimalset$.
\end{lemma}
\begin{proof}
  Let $\mathcal{U}$ be a neighborhood of $\optpoint$ where $\proj_{\optimalset}$
  is non-empty (given by Lemma~\ref{lemma:proj-non-empty}).
  Given $\bar \delta < \inj(\optpoint)$, the ball $\bar \ball(\optpoint,
  \delta)$ is compact for all $\delta < \bar \delta$.
  Define $h(\delta) = \inf_{x \in \bar \ball(\optpoint, 2\delta)} \inj(x)$ on
  the interval $\interval{0}{\frac{\bar \delta}{2}}$.
  The function $h$ is continuous with $h(0) = \inj(\optpoint) > 0$ so we can
  pick $\delta > 0$ such that $\delta \leq h(\delta)$.
  Let $x \in \ball(\optpoint, \delta) \cap \mathcal{U}$ and $y \in
  \proj_\optimalset(x)$.
  By definition of the projection we have $\dist(x, y) \leq \dist(x, \optpoint)
  < \delta$.
  Moreover, inequality~\eqref{eq:proj-dist-bound} yields $\dist(y, \optpoint)
  \leq 2 \dist(x, \optpoint) \leq 2\delta$ so $h(\delta) \leq \inj(y)$.
  It follows that $\dist(x, y) < \delta \leq h(\delta) \leq \inj(y)$ and $v =
  \Log_{y}(x)$ is well defined.
  The fact that $v$ is in the normal cone follows from optimality conditions of
  projections.
\end{proof}

Given a self-adjoint linear map $H$, we let $\lambda_i(H)$ denote the $i$th
largest eigenvalue of $H$, and $\lambda_{\min}(H)$ and $\lambda_{\max}(H)$
denote the minimum and maximum eigenvalues respectively.

\section{Four equivalent properties}\label{sec:equiv-properties}

In this section, we establish that \mb{}, \pl{}, \eb{} and \qg{} (see
Definitions~\ref{def:mb} and~\ref{def:pl-eb-qg}) are equivalent
around a local minimum $\optpoint$ when $\mfc$ is $\smooth{2}$.
Specifically, we show the implication graph in
Figure~\ref{fig:implication-graph}.
\begin{figure}
  \centering
  \begin{tikzcd}
    \mb{} &&& \pl{} \\
    \\
    \\
    \qg{} &&& \eb{}
    \arrow["\text{Prop.~\ref{prop:mb-implies-qg}}"', from=1-1, to=4-1]
    \arrow["\text{Prop.~\ref{prop:grad-dist-bounds}}"', from=4-1, to=4-4]
    \arrow["\text{Prop.~\ref{prop:eb-implies-pl}}"', from=4-4, to=1-4]
    \arrow["\text{Prop.~\ref{prop:pl-implies-qg}}", from=1-4, to=4-1]
    \arrow["\text{Cor.~\ref{cor:pl-implies-mb}}"', from=1-4, to=1-1]
  \end{tikzcd}
  \caption{Implication graph when $\mfc$ is $\smooth{2}$.
    The main missing pieces were \pl{} $\Rightarrow$ \mb{} and \qg{}
    $\Rightarrow$ \eb{}, both secured with the right constants and under the
    right regularity assumptions.}\label{fig:implication-graph}
\end{figure}
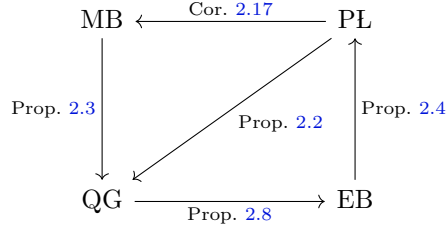

It is well known that \pl{} implies \qg{} around
minima: see references in Section~\ref{sec:relatedwork}.
Perhaps the most popular argument relies on the bounded length of gradient flow
trajectories under the more general \loja{}
inequality~\citep{lojasiewicz1963propriete,lojasiewicz1982trajectoires,otto2000generalization,absil2005convergence,bolte2010characterizations}.
\begin{definition}\label{def:loja}
  Let $\optpoint$ be a local minimum of $\mfc$ with associated set
  $\optimalset$~\eqref{eq:def-s}.
  We say $\mfc$ satisfies the \loja{} inequality with constants $\plexp \in
  \interval[open right]{0}{1}$ and $\plconstant > 0$ around $\optpoint$ if
  \begin{align}\label{eq:local-loja}\tag{\L{}}
    \lvert \mfc(x) - \mfcopt \rvert^{2\plexp} \leq \frac{1}{2\plconstant}\|\grad \mfc(x)\|^2
  \end{align}
  for all $x$ in some neighborhood of $\optpoint$.
\end{definition}
Notice that if $\mfc$ is \loja{} with exponent $\plexp$ then it is \loja{} with
exponent $\theta'$ for all $\plexp \leq \theta' < 1$ (though possibly in a
different neighborhood).
The case $\plexp = \frac{1}{2}$ is exactly the~\eqref{eq:local-pl} condition.

\begin{proposition}[$\text{\pl{}} \Rightarrow \text{\qg{}}$]\label{prop:pl-implies-qg}
  Suppose that $\mfc$ satisfies~\eqref{eq:local-loja} around $\optpoint \in
  \optimalset$.
  Then $\mfc$ satisfies
  \begin{align*}
    \mfc(x) - \mfcopt \geq \big((1 - \plexp)\sqrt{2\plconstant}\big)^{\frac{1}{1 - \plexp}}\dist(x, \optimalset)^{\frac{1}{1 - \plexp}}
  \end{align*}
  for all $x$ sufficiently close to $\optpoint$.
  In particular, if $\plexp = \frac{1}{2}$, this shows $\plconstant$-\eqref{eq:local-pl} $\Rightarrow$ $\plconstant$-\eqref{eq:quadratic-growth}.
\end{proposition}
We include a classical proof in Appendix~\ref{sec:loja-proofs} for completeness,
with care regarding neighborhoods.

\subsection{Two straightforward implications}

In this section we show that $\mb{} \Rightarrow \qg{}$ and $\eb \Rightarrow
\pl{}$.
These implications are known and direct.
We give succinct proofs for completeness.
The first one follows immediately from a Taylor expansion.

\begin{proposition}[$\text{\mb{}} \Rightarrow \text{\qg{}}$]\label{prop:mb-implies-qg}
  Suppose that $\mfc$ is $\smooth{2}$ and
  satisfies~$\plconstant$-\eqref{eq:morse-bott} at $\optpoint \in \optimalset$.
  Then $\mfc$ satisfies $\plconstant'$-\eqref{eq:quadratic-growth} around
  $\optpoint$ for all $\plconstant' < \plconstant$.
\end{proposition}
\begin{proof}
  Let $\mathcal{U}$ be a neighborhood of
  $\optpoint$ as in Lemma~\ref{lemma:log-well-defined}.
  Let $\rankop$ be the codimension of $\optimalset$ (around $\optpoint$).
  Given $\plconstant' < \plconstant$,
  pick $\varepsilon \in \interval[open]{0}{\plconstant - \plconstant'}$ and shrink
  $\mathcal{U}$ so that for all $x \in \mathcal{U}$ and $y \in
  \proj_\optimalset(x)$ we have $\lambda_{\rankop}(\hess \mfc(y)) \geq
  \plconstant' + \varepsilon$.
  Given $x \in \mathcal{U}$ and $y \in \proj_\optimalset(x)$, a Taylor expansion
  around $y$ gives
  \begin{align*}
    \mfc(x) - \mfcopt = \frac{1}{2}\inner{v}{\hess \mfc(y)[v]} + o(\|v\|^2) \geq \frac{\plconstant' + \varepsilon}{2} \dist(x, \optimalset)^2 + o(\dist(x, \optimalset)^2),
  \end{align*}
  where $v = \Log_y(x)$ is normal to $\optimalset$.
  We get the inequality $\plconstant'$-\eqref{eq:quadratic-growth} for all $x$ sufficiently close to $\optpoint$.
\end{proof}

\begin{proposition}[$\text{\eb{}} \Rightarrow \text{\pl{}}$]\label{prop:eb-implies-pl}
  Suppose that $\mfc$ is $\smooth{2}$ and satisfies
  $\plconstant$-\eqref{eq:error-bound} around $\optpoint \in \optimalset$.
  Then $\mfc$ satisfies $\plconstant'$-\eqref{eq:local-pl} around $\optpoint$
  for all $\plconstant' < \plconstant$.
\end{proposition}
\begin{proof}
  Let $\mathcal{U}$ be the intersection of two neighborhoods of $\optpoint$:
  one where $\plconstant$-\eqref{eq:error-bound} holds,
  and the other provided by Lemma~\ref{lemma:log-well-defined}.
  Given $x \in \mathcal{U}$ and $y \in \proj_{\optimalset}(x)$, a Taylor
  expansion around $y$ yields
  \begin{align*}
    \mfc(x) - \mfcopt = \frac{1}{2}\inner{v}{\hess \mfc(y)[v]} + o(\|v\|^2) && \text{and} && \grad \mfc(x) = \ptransport{v}{}\hess \mfc(y)[v] + o(\|v\|),
  \end{align*}
  where $v = \Log_y(x)$.
  Using the \causchwarz{} inequality and the triangle inequality, it follows
  that
  \begin{align*}
    \mfc(x) - \mfcopt \leq \frac{1}{2}\|v\|\|\hess \mfc(y)[v]\| + o(\|v\|^2) \leq \frac{1}{2}\|v\|\|\grad \mfc(x)\| + o(\|v\|^2).
  \end{align*}
  Finally, \eb{} gives that $\|v\| \leq \frac{1}{\plconstant} \|\grad \mfc(x)\|$
  so $\mfc(x) - \mfcopt \leq \frac{1}{2\plconstant}\|\grad \mfc(x)\|^2 +
  o(\|\grad \mfc(x)\|^2)$.
  We get the inequality $\plconstant'$-\eqref{eq:local-pl} for all $x$ sufficiently close to $\optpoint$.
\end{proof}

\begin{remark}\label{remark:eb-implies-pl-c1}
  Suppose that $\mfc$ is only $\smooth{1}$ and $\grad \mfc$ is locally
  $L$-Lipschitz continuous around $\optpoint$.
  If $\plconstant$-\eb{} holds around $\optpoint$ then $\mfc$ satisfies \pl{}
  with constant $\frac{\plconstant^2}{L}$ around
  $\optpoint$~\citep[Thm.~2]{karimi2016linear}.
  (This still holds locally on manifolds with the same proof.)
  The constant worsens but that is inevitable:
  see the example in Remark~\ref{remark:structure-only-c1}.
\end{remark}

\subsection{Quadratic growth implies error bound}\label{subsec:qg-implies-eb}

In this section, we show that \qg{} implies \eb{} for $\smooth{2}$ functions.
Other works proving this implication either assume that $\mfc$ is convex
(see~\cite[Prop.~5.3]{corvellec2008nonlinear},~\citep{karimi2016linear}, and~\cite[Cor.~3.6]{drusvyatskiy2018error})
or do not provide control on the constants
(see~\cite[Cor.~3.2]{drusvyatskiy2013second}).
For this, we first characterize a distance growth rate when we move from
$\optimalset$ in a normal direction (see
Definition~\ref{def:tangent-normal-cones}).
Recall that for now $\optimalset$ is not necessarily smooth, and therefore
$\normal_{\optpoint} \optimalset$ is a priori only a cone.
\begin{lemma}\label{lemma:dist-equiv}
  Let $\optpoint \in \optimalset$ and $v \in \normal_{\optpoint}\optimalset$
  unitary.
  Then $\dist(\Exp_{\optpoint}(tv), \optimalset) = t + o(t)$ as $t \to 0$, $t
  \geq 0$.
\end{lemma}
\begin{proof}
  Let $\mathcal{U}$ be a neighborhood of $\optpoint$ as in
  Lemma~\ref{lemma:log-well-defined}.
  Shrink $\mathcal{U}$ so that for all $x \in \mathcal{U}$ and $y \in
  \proj_{\optimalset}(x)$ we have $\dist(y, \optpoint) < \inj(\optpoint)$.
  Given a small parameter $t > 0$, define $x(t) = \Exp_{\optpoint}(tv)$ and let
  $y(t) \in \proj_{\optimalset}(x(t))$.
  From~\eqref{eq:proj-dist-bound} we have $\dist(y(t), \optpoint) \leq 2t$, and
  it follows that $y(t) \to \optpoint$ as $t \to 0$.
  Define $u(t) = \Log_{y(t)}(x(t))$ and $w(t) = \Log_{\optpoint}(y(t))$.
  Then
  \begin{align*}
    \dist(x(t), \optimalset)^2 = \|u(t)\|^2 = \|tv - w(t) + \ptransport{y(t)}{\optpoint}u(t) - tv + w(t)\|^2 = \|tv - w(t)\|^2 + o(t^2)
  \end{align*}
  as $t \to 0$ because $\|\ptransport{y(t)}{\optpoint}u(t) - tv + w(t)\| = o(t)$
  as $t \to 0$~\citep[Eq.~(6)]{sun2019escaping}. %
  In particular, for all $t$ sufficiently small we have
  \begin{align*}
    \dist(x(t), \optimalset)^2 = t^2 - 2t\inner{w(t)}{v} + \|w(t)\|^2 + o(t^2) \geq t^2 - 2t\inner{w(t)}{v} + o(t^2).
  \end{align*}
  Let $I \subseteq \reals_{> 0}$ be the times when $y(t) \neq \optpoint$.
  If $\inf I > 0$ then the final claim holds because $y(t) = \optpoint$ for
  small enough $t$.
  Suppose now that $\inf I = 0$.
  Define $r(t) = \frac{w(t)}{\|w(t)\|}$ on $I$, and let $\mathcal{A}$ be the set
  of accumulation points of $r(t)$ as $t \to 0$, $t \in I$.
  Then $\mathcal{A}$ is included in the unit sphere and $\mathcal{A} \subseteq
  \tangent_{\optpoint}\optimalset$ by definition.
  Given $a \in \mathcal{A}$ and $t \in I$, we can use $\inner{a}{v} \leq 0$ to find that
  \begin{align*}
    \dist(x(t), \optimalset)^2 \geq t^2 - 2 t \|w(t)\| \inner{r(t) - a + a}{v} + o(t^2) \geq t^2 - 4 t^2 \|r(t) - a\| + o(t^2).
  \end{align*}
  It follows that $\dist(x(t), \optimalset)^2 \geq t^2 - 4 t^2
  \dist(r(t), \mathcal{A}) + o(t^2) = t^2 + o(t^2)$ because $\dist(r(t),
  \mathcal{A}) \to 0$ as $t \to 0$.
\end{proof}

Using this rate, we now find that \qg{} implies the following bounds for $\hess
\mfc$ in the normal cones of $\optimalset$.
Note that the following proposition does not yet show \qg{} $\Rightarrow$ \mb{},
because to establish \mb{} we still need to argue that $\optimalset$ is a
smooth set.

\begin{proposition}\label{prop:c2-and-qg-hessian}
  Suppose $\mfc$ is $\smooth{2}$ and satisfies~\eqref{eq:quadratic-growth} with
  constant $\plconstant$ around $\optpoint \in \optimalset$.
  Then for all $y \in \optimalset$ sufficiently close to $\optpoint$ and all $v
  \in \normal_{y}\optimalset$ we have
  \begin{align*}
    \inner{v}{\hess \mfc(y)[v]} \geq \plconstant\|v\|^2 && \text{and} && \|\hess \mfc(y)[v]\| \geq \plconstant\|v\|.
  \end{align*}
\end{proposition}
\begin{proof}
  Let $\mathcal{U}$ be an open neighborhood of $\optpoint$ where
  $\plconstant$-\qg{} holds.
  Let $y \in \optimalset \cap \mathcal{U}$ and $v \in \normal_{y}\optimalset$.
  For all small enough $t > 0$, we obtain from a Taylor expansion that
  \begin{align*}
    \mfc(\Exp_y(tv)) - \mfcopt = \frac{t^2}{2}\inner{v}{\hess \mfc(y)[v]} + o(t^2) \geq \frac{\plconstant}{2} \dist(\Exp_y(tv), \optimalset)^2 = \frac{\plconstant}{2}t^2\|v\|^2 + o(t^2),
  \end{align*}
  where the inequality comes from \qg{} and the following equality comes from
  Lemma~\ref{lemma:dist-equiv}.
  Take $t \to 0$ to get the first inequality.
  The other inequality follows by \causchwarz{}.
\end{proof}

We deduce from this that, under \qg{}, the gradient norm is locally bounded from below and
from above by the distance to $\optimalset$, up to some constant factors.
This notably secures \eb{}.
\begin{proposition}[$\text{\qg{}} \Rightarrow \text{\eb{}}$]\label{prop:grad-dist-bounds}
  Suppose $\mfc$ is $\smooth{2}$ and satisfies
  $\plconstant$-\eqref{eq:quadratic-growth} around $\optpoint \in \optimalset$.
  For all $\lamflat < \plconstant$ and $\lamsharp > \lammax(\hess
  \mfc(\optpoint))$ there exists a neighborhood $\mathcal{U}$ of $\optpoint$
  such that for all $x \in \mathcal{U}$ we have
  \begin{align*}
    \lamflat \dist(x, \optimalset) \leq \|\grad \mfc(x)\| \leq \lamsharp \dist(x, \optimalset).
  \end{align*}
\end{proposition}
\begin{proof}
  Let $\mathcal{U}$ be a neighborhood of $\optpoint$ as in
  Lemma~\ref{lemma:log-well-defined}.
  Shrink $\mathcal{U}$ so that for all $x \in \mathcal{U}$ and $y \in
  \proj_\optimalset(x)$ the inequalities of
  Proposition~\ref{prop:c2-and-qg-hessian} hold.
  Now let $x \in \mathcal{U}$ and $y \in \proj_\optimalset(x)$.
  Define $v = \Log_y(x)$ and $\gamma(t) = \Exp_y(tv)$, so that $y = \gamma(0)$ and $x = \gamma(1)$.
  Then a Taylor expansion of $\grad\mfc$ around $y$ yields
  \begin{align*}
    \ptransport{v}{-1} \grad \mfc(x) = \hess \mfc(y)[v] + r(x) && \text{where} && r(x) = \int_0^1 \Big( \ptransport{\tau v}{-1} \circ \hess \mfc(\gamma(\tau)) \circ \ptransport{\tau v}{} - \hess \mfc(y) \Big)[v] \deriv \tau.
  \end{align*}
  The Hessian is continuous so $r(x) = o(\|v\|)$ as $x \to \optpoint$.
  Moreover, Proposition~\ref{prop:c2-and-qg-hessian} provides that $\|\hess
  \mfc(y)[v]\| \geq \plconstant \|v\|$ so using the triangle inequality and the
  reverse triangle inequality we get
  \begin{align*}
    \|v\|\big(\plconstant - o(1)\big) \leq \|\grad \mfc(x)\| \leq \|v\|\big(\lambda_{\max}(\hess \mfc(y)) + o(1)\big)
  \end{align*}
  as $x \to \optpoint$.
  We get the result if we choose $x$ close enough to $\optpoint$.
\end{proof}

The first inequality in the previous proposition is exactly \eb{}.
In Proposition~\ref{prop:eb-implies-pl} we showed that \eb{} implies \pl{} when $\mfc$ is $\smooth{2}$.
Thus, it also holds that \qg{} implies \pl{}.

\begin{remark}\label{remark:qg-implies-eb-c1}
  When $\mfc$ is only $\smooth{1}$ it is not true that \qg{} implies \eb{}
  or \pl{}.
  To see this, consider the function $\mfc(x) = 2x^2 + x^2\sin(1/\sqrt{|x|})$
  (see Figure~\ref{fig:plots}).
  It is $\smooth{1}$ and satisfies \qg{} around the minimum
  $\optpoint = 0$ because $\mfc(x) \geq x^2$.
  However, there are other local minima arbitrarily close to $\bar x$.
  Those are critical points with function value strictly larger than $\mfc(\optpoint)$, which disqualifies \eb{} and \pl{} around $\optpoint$.
\end{remark}

\begin{remark}\label{remark:pl-implies-eb-c1}
  Combining Propositions~\ref{prop:pl-implies-qg} and~\ref{prop:grad-dist-bounds} ($\plconstant$-\pl{} $\Rightarrow$ $\plconstant$-\qg{} and $\plconstant$-\qg{} $\Rightarrow$ $\plconstant'$-\eb{}), one finds that $\plconstant$-\pl{} implies $\plconstant'$-\eb{} for $\smooth{2}$ functions.
  In fact, \citet{karimi2016linear} show that $\plconstant$-\pl{} implies $\plconstant$-\eb{} for $\smooth{1}$ functions, globally so if \pl{} holds globally.
  Indeed, for $x$ sufficiently close to $\optpoint$, we have
  \begin{align*}
    \frac{\plconstant}{2}\dist(x, \optimalset)^2 \leq \mfc(x) - \mfcopt \leq \frac{1}{2\plconstant}\|\grad \mfc(x)\|^2,
  \end{align*}
  where the second inequality comes from \pl{}, and the first inequality is
  \qg{} (implied by \pl{}).
\end{remark}

Combining all previous implications, we obtain that \mb{} implies \pl{}.
This is also straightforward from Taylor expansion arguments.
\begin{corollary}[$\text{\mb{}} \Rightarrow \text{\pl{}}$]\label{cor:mb-implies-pl}
  Suppose that $\mfc$ is $\smooth{2}$ and
  satisfies $\plconstant$-\eqref{eq:morse-bott} at $\optpoint \in \optimalset$.
  Let $0 < \plconstant' < \plconstant$.
  Then $\mfc$ satisfies $\plconstant'$-\eqref{eq:local-pl} around $\optpoint$.
\end{corollary}

As a consequence, if $\mfc$ is $\plconstant$-\mb{} on any subset of $\optimalset$ then for all
$\plconstant' < \plconstant$ there exists a neighborhood of that subset where
$\mfc$ is $\plconstant'$-\pl{}.
When $\mfc$ is $\smooth{3}$ the size of the neighborhood where \pl{} holds can
be controlled (to some extent) with the third derivative.
A version of this corollary appears in~\cite[Ex.~2.9]{wojtowytsch2023stochastic}
with a different trade-off between control of the neighborhood and the constant
$\plconstant'$.
\citet[Thm.~6]{feehan2020morse} shows a similar result in Banach spaces assuming
that the function is $\smooth{3}$.

\subsection{\pl{} implies a smooth set of minima and \mb{}}

The \mb{} property is explicitly strong because it presupposes a smooth set of
minima, and it clearly implies \pl{}, \eb{}, and \qg{}.
It raises a natural question: do the latter also enforce some structure
on the set of minima?
In this section we show that the answer is yes for $\smooth{2}$ functions: if
\pl{} (or \eb{} or \qg{}) holds around $\optpoint \in \optimalset$ then
$\optimalset$ must be a submanifold around $\optpoint$.

To get a sense of why $\optimalset$ cannot have singularities, suppose that $\manifold = \reals^2$ and that
around $\optpoint$ the set of minima is the union of two orthogonal
lines (a cross) that intersect at $\optpoint$.
Then it must be that $\hess \mfc(\optpoint) = \zeros$ because the gradient is
zero along both lines.
However, if we assume \pl{} then the spectrum of $\hess \mfc(x)$ must contain
at least one positive eigenvalue bigger than $\plconstant$ for all points $x \in
\optimalset \setminus \{\optpoint\}$  close to $\optpoint$, owing to the \qg{}
property.
We obtain a contradiction because the eigenvalues of $\hess \mfc$ are
continuous.

To generalize this intuition, we first show that \pl{} induces
a lower-bound on the positive eigenvalues of $\hess \mfc$.

\begin{proposition}\label{prop:eigenvalue-bigger-pl}
  Suppose $\mfc$ is $\smooth{2}$ and $\plconstant$-\eqref{eq:local-pl} around
  $\optpoint \in \optimalset$.
  If $\lambda$ is a \emph{non-zero} eigenvalue of $\hess \mfc(\optpoint)$ then
  $\lambda \geq \plconstant$.
\end{proposition}
\begin{proof}
  Let $\lambda > 0$ be an eigenvalue of $\hess \mfc(\optpoint)$ with associated
  unit eigenvector $v$.
  Then
  \begin{align*}
    \mfc(\Exp_{\optpoint}(tv)) - \mfcopt = \frac{\lambda}{2}t^2 + o(t^2) && \text{and} && \grad \mfc(\Exp_{\optpoint}(tv)) = \lambda t \ptransport{tv}{} v + o(t).
  \end{align*}
  The \pl{} condition implies $\frac{\lambda}{2}t^2 + o(t^2) \leq
  \frac{1}{2\plconstant}{\lambda^2 t^2 + o(t^2)}$, which gives the result as $t
  \to 0$.
\end{proof}
The latter argument is inconclusive when $\lambda = 0$.
Still, we do get control of the Hessian's rank.
\begin{corollary} \label{cor:hessconstantrank}
    Suppose $\mfc$ is $\smooth{2}$ and $\plconstant$-\eqref{eq:local-pl} around
    $\optpoint \in \optimalset$.
    Then $\rank(\hess\mfc(x)) = \rank(\hess\mfc(\optpoint))$ for all $x \in
    \optimalset$ close enough to $\optpoint$.
\end{corollary}
\begin{proof}
    Since $\mfc$ is $\smooth{2}$ the eigenvalues of $\hess \mfc$ are continuous and
    the map $x \mapsto \rank(\hess \mfc(x))$ is lower semi-continuous, that is, if $x \in
    \manifold$ is close enough to $\optpoint$ then $\rank \hess \mfc(x) \geq
    \rankop$, where $\rankop = \rank \hess \mfc(\optpoint)$.
    Furthermore, if $y \in \optimalset$ is sufficiently close to $\optpoint$ then
    $\lambda_{\rankop + 1}(\hess \mfc(y)) < \plconstant$ by continuity of
    eigenvalues, and Proposition~\ref{prop:eigenvalue-bigger-pl} then implies
    $\lambda_{\rankop + 1}(\hess \mfc(y)) = 0$.
\end{proof}
This allows us to show that $\grad \mfc$ aligns locally in a special way
with the eigenspaces of $\hess \mfc$.
This alignment will be particularly valuable to analyze second-order algorithms
in Section~\ref{sec:super-linear-conv}.
\begin{lemma}\label{lemma:grad-image-hess}
  Suppose $\mfc$ is $\smooth{2}$ and satisfies~\eqref{eq:local-pl} around
  $\optpoint \in \optimalset$.
  Let $\rankop = \rank(\hess \mfc(\optpoint))$.
  Then the orthogonal projector $P(x)$ onto the top $\rankop$ eigenspace of
  $\hess \mfc(x)$ is well defined when $x$ is sufficiently close to $\optpoint$,
  and (with $I$ denoting identity)
  \begin{align*}
    \|(I - P(x)) \grad \mfc(x)\| = o(\dist(x, \optimalset)) =  o(\|\grad \mfc(x)\|)
  \end{align*}
  as $x \to \optpoint$.
  Additionally, if $\hess \mfc$ is locally Lipschitz continuous around
  $\optpoint$ then $\|(I - P(x)) \grad \mfc(x)\| = O(\dist(x, \optimalset)^2)
  = O(\|\grad \mfc(x)\|^2)$ as $x \to \optpoint$.
\end{lemma}
\begin{proof}
  Given a point $x \in \manifold$, let $P(x) \colon \tangent_x\manifold \to
  \tangent_x\manifold$ denote the orthogonal projector onto the top $\rankop$
  eigenspace of $\hess \mfc(x)$.
  This is well defined provided $\lambda_{\rankop}(\hess\mfc(x)) > \lambda_{\rankop+1}(\hess\mfc(x))$.
  Let $\mathcal{U}$ be a neighborhood of $\optpoint$ as in
  Lemma~\ref{lemma:log-well-defined}.
  By continuity of the eigenvalues of $\hess \mfc$, we can shrink $\mathcal{U}$
  so that for all $x \in \mathcal{U}$ and $y \in \proj_\optimalset(x)$ the
  projectors $P(x)$ and $P(y)$ are well defined.
  Given $x \in \mathcal{U}$, we let $y \in \proj_\optimalset(x)$ and $v =
  \Log_{y}(x)$.
  Now define $\gamma(t) = \Exp_y(tv)$.
  A Taylor expansion of the gradient around $y$ gives that
  \begin{align*}
    \grad \mfc(x) = \ptransport{v}{}\big(\hess \mfc(y)[v] + r(x)\big) && \text{where} && r(x) = \int_0^1 \big( \ptransport{\tau v}{-1} \circ \hess \mfc(\gamma(\tau)) \circ \ptransport{\tau v}{} - \hess \mfc(y) \big)[v] \deriv \tau.
  \end{align*}
  By Corollary~\ref{cor:hessconstantrank}, the rank of $\hess \mfc$ is locally
  constant on $\optimalset$ (equal to $\rankop$) so $\hess\mfc(y) = P(y) \hess
  \mfc(y)$ whenever $x$ is sufficiently close to $\optpoint$ (using the bound
  $\dist(y, \optpoint) \leq 2 \dist(x, \optpoint)$
  from~\eqref{eq:proj-dist-bound}).
  It follows that
  \begin{align*}
    (I - P(x)) \grad \mfc(x) &= (I - P(x)) \ptransport{v}{} \big(P(y) \hess \mfc(y)[v] + r(x)\big).
  \end{align*}
  Notice that $\ptransport{v}{} \to I$ as $x \to \optpoint$ and $P(y) \to
  P(\optpoint)$ as $x \to \optpoint$ so $(I - P(x))\ptransport{v}{}P(y) \to
  \zeros$ as $x \to \optpoint$.
  It follows that $(I - P(x)) \grad \mfc(x) = o(\|v\|)$ as $x \to \optpoint$.
  The claim follows by noting that $\|v\| = \dist(x, \optimalset)$ and the fact
  that $\dist(x, \optimalset)$ is commensurate $\|\grad \mfc(x)\|$ (as shown in
  Proposition~\ref{prop:grad-dist-bounds}).

  Suppose now that $\hess \mfc$ is locally Lipschitz continuous around
  $\optpoint$.
  Then $P$ is also locally Lipschitz continuous around
  $\optpoint$~\citep[Thm.~1]{vannieuwenhoven2024condition} and $(I - P(x))
  \ptransport{v}{} P(y) = (I - P(x)) \ptransport{v}{} (P(y) - \ptransport{v}{-1}
  P(x)) = O(\|v\|)$.
  Moreover, we have $r(x) = O(\|v\|^2)$ so it follows that $(I - P(x)) \grad
  \mfc(x) = O(\|v\|^2)$ as $x \to \optpoint$.
  We conclude again with Proposition~\ref{prop:grad-dist-bounds}.
\end{proof}

The lemma below exhibits a submanifold $\chillmanifold$ that contains
$\optpoint$.
It need \emph{not} coincide with the set of minima $\optimalset$.
However, they \emph{do} coincide if we assume that \pl{} holds.
This is the main argument to show that $\optimalset$ is locally a submanifold
whenever $\mfc$ is \pl{} and sufficiently regular.

\begin{lemma}\label{lemma:chill-manifold}
  Suppose $\mfc$ is $\smooth{p}$ with $p \geq 2$.
  Let $\optpoint \in \optimalset$ and define $\mathcal{U} = \{x \in \manifold :
  \dist(x, \optpoint) < \inj(\optpoint)\}$.
  Let $P(\optpoint)$ denote the orthogonal projector onto the image of $\hess
  \mfc(\optpoint)$.
  Then the set
  \begin{align*}
    \chillmanifold = \{x \in \mathcal{U} : P(\optpoint) \ptransport{x}{\optpoint} \grad \mfc(x) = \zeros\}
  \end{align*}
  is a $\smooth{p - 1}$ embedded submanifold locally around $\optpoint$.
  If $\mfc$ is analytic then $\chillmanifold$ is also analytic.
\end{lemma}
\begin{proof}
  We build a local defining function for $\chillmanifold$.
  Let $\rankop$ be the rank of $\hess \mfc(\optpoint)$ and $u_1, \dots,
  u_\rankop$ be a set of orthonormal eigenvectors of $\hess \mfc(\optpoint)$
  with associated eigenvalues $\lambda_1, \dots, \lambda_\rankop > 0$.
  We define $h \colon \mathcal{U} \to \reals^\rankop$ as $h_i(x) =
  \inner{u_i}{\ptransport{x}{\optpoint} \grad \mfc(x)}$.
  Clearly, $h(x) = \zeros$ if and only if $x \in \chillmanifold$.
  The function $h$ is $\smooth{p - 1}$ if $\mfc$ is $\smooth{p}$, and analytic
  if $\mfc$ is analytic.
  For all $\dot x \in \tangent_{\optpoint}\manifold$ we have
  \begin{align*}
    \D h_i(\optpoint)[\dot x] = \inner{u_i}{\hess \mfc(\optpoint)[\dot x]} = \inner{\hess \mfc(\optpoint)[u_i]}{\dot x} = \lambda_i\inner{u_i}{\dot x}.
  \end{align*}
  It follows that $\D h(\optpoint)$ has full rank.
  Thus, $\chillmanifold$ is a
  submanifold around $\optpoint$ with the stated regularity.
\end{proof}

A result similar to Lemma~\ref{lemma:chill-manifold} is presented
in~\cite[Lem.~1]{chill2006lojasiewicz} for Banach spaces.
We are now ready for one of our main theorems, regarding the regularity of the set of local minimizers $\optimalset$~\eqref{eq:def-s}.
\begin{theorem}\label{th:pl-implies-submanifold}
  Suppose $\mfc$ is $\smooth{p}$ with $p \geq 2$ and
  satisfies~\eqref{eq:local-pl} around $\optpoint \in \optimalset$.
  Then $\optimalset$ is a $\smooth{p - 1}$ submanifold of $\manifold$ locally
  around $\optpoint$.
  If $\mfc$ is analytic then $\optimalset$ is also analytic.
\end{theorem}
\begin{proof}
  We let $\rankop$ denote the rank of $\hess \mfc(\optpoint)$.
  By Corollary~\ref{cor:hessconstantrank}, $\rank \hess \mfc(y) = \rankop$ for
  all $y \in \optimalset$ sufficiently close to $\optpoint$.
  We let $\mathcal{U} \subseteq \ball(\optpoint, \inj(\optpoint))$ be a
  neighborhood of $\optpoint$ such that for all $x \in \mathcal{U}$ the
  orthogonal projector $P(x) \colon \tangent_x\manifold \to
  \tangent_x\manifold$ onto the top $\rankop$ eigenspace of $\hess \mfc(x)$ is
  well defined (it exists because eigenvalues of $\hess \mfc$ are continuous).
  Lemma~\ref{lemma:chill-manifold} ensures that
  \begin{align*}
    \chillmanifold = \{x \in \mathcal{U} : P(\optpoint) \ptransport{x}{\optpoint} \grad \mfc(x) = \zeros\}
  \end{align*}
  is a submanifold around $\optpoint$.
  Clearly, $\optimalset \cap \mathcal{U}
  \subseteq \chillmanifold$ holds.
  We now show the other inclusion to obtain that $\optimalset$ and
  $\chillmanifold$ coincide around $\optpoint$.
  From Lemma~\ref{lemma:grad-image-hess} we have $\|\grad \mfc(x)\| \leq \|P(x)
  \grad \mfc(x)\| + o(\|\grad \mfc(x)\|)$ as $x \to \optpoint$.
  Moreover, the triangle inequality gives $\|P(x) \grad \mfc(x)\| \leq \|(P(x) -
  \ptransport{\optpoint}{x} P(\optpoint) \ptransport{x}{\optpoint}) \grad
  \mfc(x)\| + \|P(\optpoint) \ptransport{x}{\optpoint} \grad \mfc(x)\|$.
  By continuity of $P$ we have $P(x) - \ptransport{\optpoint}{x} P(\optpoint)
  \ptransport{x}{\optpoint} = o(1)$ as $x \to \optpoint$ so it follows that
  \begin{align*}
    \|\grad \mfc(x)\| \leq \|P(\optpoint) \ptransport{x}{\optpoint} \grad \mfc(x)\| + o(\|\grad \mfc(x)\|)
  \end{align*}
  as $x \to \optpoint$.
  We conclude that $P(\optpoint) \ptransport{x}{\optpoint} \grad \mfc(x) =
  \zeros$ implies $\grad \mfc(x) = \zeros$ for all $x$ sufficiently close to $\optpoint$.
  This confirms that $\chillmanifold \subseteq \optimalset \cap \mathcal{U}$ around $\optpoint$
  because all critical points near $\optpoint$ are in $\optimalset$ by \pl{}.
\end{proof}

The codimension of $\optimalset$ is equal to the rank of $\hess \mfc$ on
$\optimalset$, as expected.
A similar result holds for Banach spaces when the function is assumed
analytic~\cite[Thm.~1]{feehan2020morse}.
Around $\optpoint$, the set of all minima of $\mfc$ and $\optimalset$ coincide
when \pl{} holds.
Hence, Theorem~\ref{th:pl-implies-submanifold} implies that the set of minima of
$\mfc$ is a submanifold around $\optpoint$.
Using the \qg{} property we now deduce that \pl{} implies \mb{}.

\begin{corollary}\label{cor:pl-implies-mb}
  If $\mfc$ is $\smooth{2}$ and $\plconstant$-\eqref{eq:local-pl} around
  $\optpoint \in \optimalset$ then it
  satisfies~$\plconstant$-\eqref{eq:morse-bott} at $\optpoint$.
  The same holds if $\mfc$ is $\plconstant$-\eqref{eq:error-bound} or
  $\plconstant$-\eqref{eq:quadratic-growth} rather than
  $\plconstant$-\eqref{eq:local-pl}.
\end{corollary}
\begin{proof}
  Apply Theorem~\ref{th:pl-implies-submanifold} to get that $\optimalset$ is
  locally a $\smooth{1}$ submanifold around $\optpoint$.
  Proposition~\ref{prop:pl-implies-qg} gives~\eqref{eq:quadratic-growth} around
  $\optpoint$.
  Finally apply Proposition~\ref{prop:c2-and-qg-hessian} to normal eigenvectors
  of $\hess \mfc(\optpoint)$.
  This yields that the normal eigenvalues are at least $\plconstant$.
  We obtain the same result if we suppose $\plconstant$-\eb{} or
  $\plconstant$-\qg{} instead of $\plconstant$-\pl{}.
  This is because they both imply $\plconstant'$-\pl{} for $\plconstant' <
  \plconstant$ arbitrarily close to $\plconstant$.
  Taking the limit $\plconstant' \to \plconstant$ gives $\plconstant$-\mb{}.
\end{proof}

\begin{remark}[Connections to the distance function]\label{remark:connection-distance}
  Given a closed set $\minsubset \subseteq \manifold$, the function $\mfc(x) =
  \frac{1}{2}\dist_{\minsubset}(x)^2$ clearly satisfies \qg{}.
  If $\mfc$ is $\smooth{p}$ with $p \geq 2$ in a neighborhood of
  $\minsubset$ then Theorem~\ref{th:pl-implies-submanifold} applies, revealing that
  $\minsubset$ is a $\smooth{p-1}$ submanifold.
  This question is of independent interest, see for
  example~\citep{bellettini2007conjecture} for a proof when assuming $p \geq 3$.
\end{remark}

\begin{remark}[Structure when $\mfc$ is only $\smooth{1}$]\label{remark:structure-only-c1}
  Theorem~\ref{th:pl-implies-submanifold} requires $\mfc$ to be $\smooth{2}$.
  And indeed, if $\mfc$ is only $\smooth{1}$ the set of minima may not be a
  submanifold.
  We provide two examples.

  The function $\mfc(x, y) = \frac{x^2y^2}{x^2 + y^2}$ is $\smooth{1}$ and \pl{} around the origin,
  yet its minimizers form a cross (see
  Figure~\ref{fig:plots}).\footnote{We thank Christopher Criscitiello
    who found this function.}
  Incidentally, $f$ is $\frac{1}{\sqrt{2}}$-\eb{} but only $\frac{1}{2}$-\pl{}
  around the origin, confirming that the constant worsens for the implication
  $\eb{} \Rightarrow \pl{}$ when $\mfc$ is only $\smooth{1}$.

  Additionally, let $\minsubset \subseteq \manifold$ be a closed set and
  suppose that the distance function $\dist_{\minsubset}$ is $\smooth{1}$ around
  $\minsubset$ (such a set is called proximally
  smooth~\citep{clarke1995proximal}).
  For $\mfc(x) = \frac{1}{2}\dist_\minsubset(x)^2$, we find that $\grad
  \mfc(x) = x - \proj_\minsubset(x)$, meaning that $\mfc$ is \pl{} around
  $\minsubset$ with constant $\plconstant = 1$.
  This holds in particular for all closed convex sets, yet many such sets fail
  to be $\smooth{0}$ submanifolds (e.g., consider a closed square in $\manifold
  = \reals^2$).
  This provides further examples of $\smooth{1}$ functions satisfying \pl{} yet
  for which the set of minima is not $\smooth{0}$.
\end{remark}

\begin{remark}[Restricted secant inequality] \label{rem:RSI}
  It is possible to show equivalences with even more properties.
  For example, as in~\citep{zhang2013gradient}, we say $\mfc$ satisfies the
  \emph{restricted secant inequality} (RSI) with constant $\plconstant$ around
  $\optpoint \in \optimalset$ if $\inner{\grad \mfc(x)}{v} \geq \plconstant
  \dist(x, \optimalset)^2$ for all $x$ in a neighborhood of $\optpoint$, where
  $v = -\Log_x(y)$ and $y \in \proj_\optimalset(x)$.
  From simple Taylor expansion arguments, we find that $\plconstant$-\mb{}
  implies $\plconstant'$-RSI for all $\plconstant' < \plconstant$.
  By \causchwarz{}, we also find that $\plconstant$-RSI implies
  $\plconstant$-\eb{} for $\smooth{1}$ functions (see~\citep{karimi2016linear}).
  It follows that for $\smooth{2}$ functions RSI is also equivalent to the four
  properties that we consider.
\end{remark}

\begin{remark}[Other \loja{} exponents]\label{remark:other-loja-exponents}
  The \pl{} condition is exactly~\eqref{eq:local-loja} with exponent $\plexp =
  \frac{1}{2}$.
  We comment here about other values of $\plexp$.
  First, suppose $\mfc$ is $\smooth{1}$ and $\grad \mfc$ locally Lipschitz
  continuous.
  If $\mfc$ is non-constant around $\optpoint \in \optimalset$ then it cannot
  satisfy~\eqref{eq:local-loja} with an exponent $\plexp < \frac{1}{2}$ around
  $\optpoint$.
  This is because these assumptions are incompatible with the growth property
  from Proposition~\ref{prop:pl-implies-qg}.
  See also~\cite[Thm.~4]{abbaszadehpeivasti2023conditions} for an algorithmic
  perspective on this.
  Now suppose $\mfc$ is $\smooth{2}$ and satisfies~\eqref{eq:local-loja} with
  exponent $\plexp$ around $\optpoint \in \optimalset$.
  If $\plexp \in \interval[open right]{\frac{1}{2}}{\frac{2}{3}}$ and $\hess
  \mfc$ is Lipschitz continuous around $\optpoint$ then \pl{} holds around
  $\optpoint$.
  Furthermore, if $\plexp = \frac{2}{3}$ and $\mfc$ is additionally $\smooth{3}$ then \pl{}
  also holds around $\optpoint$.
  We include a proof of these observations in
  Appendix~\ref{sec:other-loja-exponents}.
\end{remark}

\section{Stability of minima and linear rates}\label{sec:reminders}

In this section, we consider two types of algorithmic questions:
the stability of minima and standard local convergence rates.
We review some necessary classical arguments, taking this opportunity to generalize
some of them to accommodate non-isolated minima.
This will serve us well to analyze algorithms in
Section~\ref{sec:super-linear-conv}.

\subsection{Capture for sets of non-isolated minima}\label{subsec:capture}

Typically, global convergence analyses of optimization algorithms merely
guarantee that iterates accumulate only at critical points.
The set of accumulation points may be empty (when the iterates diverge).
Worse, it may even be infinite when minima are not isolated.
See~\cite[\S3.2.1]{absil2005convergence} and \cite[\S5.3]{bolte2021curiosities}
for examples of pathological functions for which reasonable algorithms (such
as gradient descent) produce iterates with continuous sets of accumulation
points.
The latter issue cannot occur when minima are isolated.

What kind of stability results still hold when minima are not isolated?
Consider an algorithm generating iterates as $x_{k + 1} =
\deterministicalgorithm_k(x_k, \dots, x_0)$, where $\deterministicalgorithm_k$
is a \emph{descent mapping}: it satisfies $\mfc(\deterministicalgorithm_k(x_k, \dots,
x_0)) \leq \mfc(x_k)$ for all $x_k, \dots, x_0 \in \manifold$.
Many deterministic algorithms fall in this category, including gradient descent
and trust-region methods under suitable hypotheses.
The standard capture theorem asserts that if the iterates generated by such
descent mappings get sufficiently close to an \emph{isolated} local minimum then
the sequence eventually converges to it (under a few weak
assumptions)~\citep[Prop.~1.2.5]{bertsekas1995nonlinear}, \citep[Thm.~4.4.2]{absil2009optimization}.
The result can be easily extended to a compact set of non-isolated local minima
that satisfies several properties that we define now.

\begin{definition}\label{def:isolated-from-critical-points}
  We say $\minsubset \subseteq \manifold$ is \emph{isolated from critical
    points} if there exists a neighborhood $\mathcal{U}$ of $\minsubset$ such
  that $x \in \mathcal{U}$ and $\grad \mfc(x) = \zeros$ imply $x \in
  \minsubset$.
\end{definition}

Note that the points in $\minsubset$ do \emph{not} need to be isolated: the set
$\minsubset$ may be a continuum of critical points.
It is clear that if $\minsubset \subseteq \optimalset$~\eqref{eq:def-s} is
isolated from critical points then there exists a neighborhood $\mathcal{U}$ of
$\minsubset$ such that $\mfc(x) > \mfcopt$ for all $x \in \mathcal{U} \setminus
\minsubset$, where $\mfcopt$ is the value of $\mfc$ on $\minsubset$.
The capture result below, based on~\citep[Thm.~4.4.2]{absil2009optimization}, states that if
the set of minima $\minsubset \subseteq \optimalset$ is both compact and
isolated from critical points then it traps the iterates generated by all
reasonable descent algorithms.
A key hypothesis is that the steps have to be small around local minima: that is
typically the case.

\begin{definition}
  An algorithm which generates sequences on $\manifold$ has the \emph{vanishing
    steps} property on a set $\minsubset \subseteq \manifold$ if there exists a neighborhood
  $\mathcal{U}$ of $\minsubset$ and a continuous function $\boundvanishingsteps
  \colon \manifold \to \reals_+$ with $\boundvanishingsteps(\minsubset) = 0$
  such that, if $x_k$ is an iterate in $\mathcal{U}$, then the next iterate
  $x_{k + 1}$ satisfies
  \begin{align}\label{eq:vanishing-steps}\tag{VS}
    \dist(x_k, x_{k + 1}) \leq \boundvanishingsteps(x_k).
  \end{align}
  We say that the algorithm has the~\eqref{eq:vanishing-steps} property at a
  point $\optpoint \in \optimalset$ if it holds on the set $\{\optpoint\}$.
\end{definition}

\begin{proposition}[Capture of iterates]\label{prop:capture-theorem}
  Let $\minsubset$ be a compact subset of $\optimalset$ isolated from critical
  points.
  Consider an algorithm that produces iterates as $x_{k + 1} =
  \deterministicalgorithm_k(x_k, \dots, x_0)$, where $\deterministicalgorithm_k$
  is a descent mapping.
  Assume that it satisfies the~\eqref{eq:vanishing-steps} property on
  $\minsubset$.
  Also suppose that the sequences generated by this algorithm accumulate only at
  critical points of $\mfc$.
  Then there exists a neighborhood $\mathcal{U}$ of $\minsubset$ such that if a
  sequence enters $\mathcal{U}$ then all subsequent iterates are in
  $\mathcal{U}$ and $\dist(x_k, \minsubset) \to 0$.
\end{proposition}
\begin{proof}
  There exists a compact neighborhood $\mathcal{V}$ of $\minsubset$ such that
  $\mathcal{V} \setminus \minsubset$ does not contain any critical point and
  $\mfc(x) > \mfcopt$ for all $x \in \mathcal{V} \setminus \minsubset$.
  The~\eqref{eq:vanishing-steps} property implies that there exists an open
  neighborhood $\mathcal{W}$ of $\minsubset$ included in $\mathcal{V}$ such
  that for all $k \in \naturals$, $x_k \in \mathcal{W}$ and $x_{k - 1}, \dots,
  x_0 \in \manifold$ we have $\deterministicalgorithm_k(x_k, \dots, x_0) \in
  \mathcal{V}$.
  The set $\mathcal{V} \setminus \mathcal{W}$ is compact, and we let $\alpha >
  \mfcopt$ denote the minimum of $\mfc$ on this set.
  We define $\mathcal{U} = \{x \in \mathcal{V} : \mfc(x) < \alpha\}$, which is
  included in $\mathcal{W}$ by minimality of $\alpha$.
  Now let $K \in \naturals$, $x_K \in \mathcal{U}$ and $x_{K - 1}, \dots, x_0
  \in \manifold$.
  Then we have $\deterministicalgorithm_K(x_K, \dots, x_0) \in \mathcal{V}$ by
  definition of $\mathcal{W}$.
  Moreover, $\deterministicalgorithm_K$ is a descent mapping so
  $\mfc(\deterministicalgorithm_K(x_K, \dots, x_0)) < \alpha$, and it implies
  that $\deterministicalgorithm_K(x_K, \dots, x_0) \in \mathcal{U}$.
  It follows that $x_{K + 1}$ is in $\mathcal{U}$ and all subsequent iterates
  are also in $\mathcal{U}$.
  Now we show that $\sequence{x_k}$ converges to $\minsubset$.
  The sequence $\sequence{x_k}$ eventually stays in a compact set (because $x_k \in \mathcal{U} \subseteq \mathcal{V}$ for
  all $k \geq K$) so it has a non-empty and compact set of accumulation points
  that we denote by $\mathcal{A}$.
  Then $\mathcal{A} \subseteq \minsubset$ because $\mathcal{A} \subseteq
  \mathcal{V}$ and the only critical points in $\mathcal{V}$ are in
  $\minsubset$.
  The set of accumulation points of $\sequence{\dist(x_k, \mathcal{A})}$ is
  $\{0\}$ (because $\sequence{x_k}$ is bounded).
  So we deduce that $\lim_{k \to +\infty} \dist(x_k, \minsubset) \leq \lim_{k
    \to +\infty} \dist(x_k, \mathcal{A}) = 0$.
\end{proof}

This statement does \emph{not} guarantee that the iterates converge to a specific
point, but merely that $\dist(x_k, \minsubset) \to 0$.
Notice that we do not require any particular structure for $\minsubset$ nor any
form of function growth.
The assumptions on the mappings $F_k$ are also mild.

However, we need $\minsubset$ to be compact and this cannot be relaxed.
Indeed, consider the function $\mfc \colon \reals^2 \to \reals$ defined as
$\mfc(x, y) = \exp(x)y^2$.
The set of global minima is $\minsubset = \{(x, y) \in \reals^2 : y = 0\}$, and
it contains all the critical points of $\mfc$.
Consider the update rule $(x_{k + 1}, y_{k + 1}) = (x_k - y_k^2, y_k)$.
It satisfies the descent condition because $\mfc(x_{k + 1}, y_{k + 1}) =
\exp(-y_k^2)\mfc(x_k, y_k)$ for all $k$.
The distance assumption~\eqref{eq:vanishing-steps} also holds because
$\dist((x_{k + 1}, y_{k + 1}), (x_k, y_k)) = \dist((x_k - y_k^2, y_k), (x_k,
y_k)) = y_k^2$.
However, the sequence $\sequence{\dist((x_k, y_k), \minsubset)}$ is constant
(and not converging to zero) even if we initialize the algorithm arbitrarily
close to $\minsubset$ (but not exactly on $\minsubset$).

The vanishing steps property is a reasonable assumption.
When $\manifold = \reals^n$, many optimization algorithms have an
update rule of the form $x_{k + 1} = x_k + s_k$ and the vector $s_k$ is small
if $x_k$ is close to a local minimum (e.g., $s_k = -\alpha_k \grad
\mfc(x_k)$ with bounded $\alpha_k$).
For a general manifold $\manifold$, algorithms produce iterates as
$x_{k + 1} = \retr_{x_k}(s_k)$, where $\retr \colon \tangent\manifold \to
\manifold \colon (x, s) \mapsto \retr_x(s)$ is a
retraction~\citep[Def.~3.47]{boumal2020introduction}.
Specifically, for all $(x, s) \in \tangent\manifold$ the curve $c(t) =
\retr_x(ts)$ satisfies $c(0) = x$ and $c'(0) = s$.

To ensure vanishing steps, we must control the distance traveled by retractions.
We let $\retrdistboundconst \geq 1$ be such that
\begin{align}\label{eq:retr-distance-condition}
  \dist(\retr_x(s), x) \leq \retrdistboundconst \|s\|
  \tag{RD}
\end{align}
for all $(x, s) \in \tangent\manifold$ where $\|s\|$ is smaller than some fixed positive radius.
For $\retr = \Exp$ (including the Euclidean case) the choice $\retrdistboundconst = 1$ is valid for all $s$.
More generally, $\retrdistboundconst$ can be set arbitrarily close to 1 for all retractions as long as the radius is sufficiently small~\citep[Lem.~6]{ring2012optimization}.

\subsection{Lyapunov stability and convergence to a single
  point}\label{subsec:lyap}

Pathological behavior such as a continuous set of accumulation points can be
ruled out assuming \loja{} inequalities.
In this case, local minima are stable for a variety of
algorithms, even without compactness hypothesis (in contrast to
Proposition~\ref{prop:capture-theorem}).
This in turn ensures that the iterates converge to a single point.
In this section, we review arguments
from~\cite[Thm.~4]{polyak1963gradient},~\citep{absil2005convergence},~\cite[Lem.~2.6]{attouch2013convergence},
and~\cite[Thm.~14]{bolte2017fromerror}.

A central property to obtain convergence to a single point is a bound on the
path length of the iterates.
We make this precise in the following definition.

\begin{definition}
  An algorithm which generates sequences on $\manifold$ has the \emph{bounded
    path length} property on a set $\minsubset \subseteq \manifold$ if the following
  is true.
  There exist a neighborhood $\mathcal{U}$ of $\minsubset$ and a continuous
  function $\boundpathlength \colon \manifold \to \reals_+$ with
  $\boundpathlength(\minsubset) = 0$ such that, if $x_L, \dots, x_K \in
  \manifold$ are consecutive points generated by the algorithm and which are all
  in $\mathcal{U}$, then
  \begin{align}\label{eq:bounded-path-length-prop}\tag{BPL}
    \sum_{k = L}^{K - 1} \dist(x_k, x_{k + 1}) \leq \boundpathlength(x_L).
  \end{align}
  We say that the algorithm has the~\eqref{eq:bounded-path-length-prop} property
  at a point $\optpoint \in \manifold$ if it does so on the set $\{\optpoint\}$.
\end{definition}

Here we think of the algorithm as an optimization method with some fixed
hyper-parameters.
The definition is given for a generic set $\minsubset$ but the bounded path
length property is usually only satisfied around local minima.
Combined with the vanishing steps property, bounded path length ensures
stability of local minima.
For comparison,~\cite{absil2005convergence}
and~\cite{attouch2013convergence} deduce~\eqref{eq:bounded-path-length-prop} from a function decrease condition.
Here, we factor out~\eqref{eq:bounded-path-length-prop} to enable analysis of
algorithms that do not satisfy that decrease condition.

\begin{proposition}[Lyapunov stability]\label{prop:lyapunov-stability}
  Suppose that an algorithm satisfies the~\eqref{eq:vanishing-steps}
  and~\eqref{eq:bounded-path-length-prop} properties at $\optpoint \in
  \manifold$.
  Given a neighborhood $\mathcal{U}$ of $\optpoint$, there exists a neighborhood
  $\mathcal{V}$ of $\optpoint$ such that if a sequence generated by this
  algorithm enters $\mathcal{V}$ then all subsequent iterates stay in
  $\mathcal{U}$.
\end{proposition}
\begin{proof}
  The set $\mathcal{U}$ contains an open ball of radius $\delta_u > 0$ around
  $\optpoint$ in which the~\eqref{eq:bounded-path-length-prop}
  and~\eqref{eq:vanishing-steps} properties are satisfied with some functions
  $\boundpathlength$ and $\boundvanishingsteps$.
  By continuity of $\boundvanishingsteps$ there exists an open ball
  $\mathcal{W}$ centered on $\optpoint$ of radius $\delta_w$ that satisfies
  $\delta_w + \boundvanishingsteps(x) < \delta_u$ for all $x \in
  \mathcal{W}$.
  Likewise, by continuity of $\boundpathlength$, there exists an open ball
  $\mathcal{V} \subset \mathcal{W}$ of radius $\delta_v > 0$ around $\optpoint$
  such that for all $x \in \mathcal{V}$ we have $\delta_v + \boundpathlength(x)
  < \delta_w$.
  Suppose that an iterate $x_L$ is in $\mathcal{V}$.
  For contradiction, let $K \geq L$ be the first index such
  that $x_{K + 1} \notin \ball(\optpoint, \delta_u)$.
  We deduce from the triangle inequality and
  the~\eqref{eq:bounded-path-length-prop} property that
  \begin{align*}
    \dist(x_K, \optpoint) \leq \dist(x_L, \optpoint) + \sum_{k = L}^{K - 1} \dist(x_k, x_{k + 1}) \leq \delta_v + \boundpathlength(x_L) < \delta_w.
  \end{align*}
  It follows that $x_K \in \mathcal{W}$.
  Using again the triangle inequality we find $\dist(x_{K + 1}, \optpoint) \leq
  \dist(x_K, \optpoint) + \dist(x_K, x_{K + 1}) \leq \delta_w +
  \boundvanishingsteps(x_K) < \delta_u$.
  This implies that $x_{K + 1}$ is in $\ball(\optpoint, \delta_u)$: a
  contradiction.
\end{proof}

In particular, this excludes that the iterates diverge.
We can also guarantee that accumulation points are actually limit points.

\begin{corollary}\label{cor:pl-accumulation-unique-limit}
  Suppose that an algorithm satisfies the~\eqref{eq:vanishing-steps}
  and~\eqref{eq:bounded-path-length-prop} properties at $\optpoint \in
  \manifold$.
  If it generates a sequence that accumulates at $\optpoint$ then the sequence
  converges to $\optpoint$.
\end{corollary}
\begin{proof}
  Let $\mathcal{U}$ be a neighborhood of $\optpoint$.
  From Proposition~\ref{prop:lyapunov-stability} there is a neighborhood
  $\mathcal{V}$ such that if an iterate is in $\mathcal{V}$ then all
  subsequent iterates are in $\mathcal{U}$.
  Since $\optpoint$ is an accumulation point we know that such an iterate
  exists.
  Repeat with a sequence of smaller and smaller neighborhoods of $\optpoint$.
\end{proof}

Many optimization algorithms generate sequences that accumulate only at critical
points.
In that scenario, we can deduce that the sequence converges to a point, provided
that it gets close enough to a set where~\eqref{eq:vanishing-steps}
and~\eqref{eq:bounded-path-length-prop} hold.

\begin{corollary}\label{cor:capture-to-single-point}
  Consider an algorithm that satisfies the~\eqref{eq:vanishing-steps}
  and~\eqref{eq:bounded-path-length-prop} properties on a set $\minsubset
  \subseteq \manifold$ and let $\optpoint \in \minsubset$.
  Let $\mathcal{U}$ be a neighborhood of $\optpoint$ such that if a sequence
  generated by this algorithm accumulates at a point $x_\infty \in \mathcal{U}$
  then $x_\infty$ is in $\minsubset$.
  There exists a neighborhood $\mathcal{V}$ of $\optpoint$ such that if a
  sequence enters $\mathcal{V}$ then all subsequent iterates stay in
  $\mathcal{U}$ and converge to some $x_\infty \in \mathcal{U} \cap \minsubset$.
\end{corollary}
\begin{proof}
  The set $\mathcal{U}$ contains a compact neighborhood $\mathcal{B}$ of
  $\optpoint$ such that the~\eqref{eq:vanishing-steps}
  and~\eqref{eq:bounded-path-length-prop} properties hold on $\mathcal{B} \cap
  \minsubset$.
  From Proposition~\ref{prop:lyapunov-stability} there exists a neighborhood
  $\mathcal{V}$ of $\optpoint$ such that if $\sequence{x_k}$ enters
  $\mathcal{V}$ then it stays in $\mathcal{B}$.
  The set $\mathcal{B}$ is compact so $\sequence{x_k}$ has at least one
  accumulation point $x_\infty \in \mathcal{B}$.
  Our hypotheses ensure that $x_\infty$ must be in $\minsubset$ since
  $\mathcal{B} \subseteq \mathcal{U}$.
  We conclude with Corollary~\ref{cor:pl-accumulation-unique-limit}.
\end{proof}

The conclusions of Corollary~\ref{cor:capture-to-single-point} are similar to
the ones in~\cite[Prop.~3.3]{absil2005convergence}
and~\cite[Thm.~2.10]{attouch2013convergence}.\footnote{Notice
  that~\cite{absil2005convergence} do \emph{not} assume vanishing steps, but in
  that case the property can fail when $\optpoint$ is not a \emph{global}
  minimum. The statement of Corollary~\ref{cor:capture-to-single-point} accounts
  for this technical point.
}

We describe below the argument that~\cite{absil2005convergence} used to show
that many gradient descent algorithms
satisfy~\eqref{eq:bounded-path-length-prop} around points where $\mfc$ is
\loja{}.
We say that the sequence $\sequence{x_k}$ satisfies the \emph{strong decrease
property} around $\optpoint \in \manifold$ if there exists
$\lyapsufficientdecreaseconst > 0$ such that
\begin{align}\label{eq:sufficient-decrease-lyapunov}
  \mfc(x_k) - \mfc(x_{k + 1}) \geq \lyapsufficientdecreaseconst \|\grad \mfc(x_k)\|\dist(x_k, x_{k + 1}) && \text{and} && x_k \in \optimalset \;\;\Rightarrow\;\; x_{k + 1} = x_k
\end{align}
whenever $x_k$ is sufficiently close to $\optpoint$, as introduced
by~\cite{absil2005convergence}.

\begin{lemma}\label{lemma:bound-discrete-gd-path-length}
  Suppose that $\mfc$ satisfies~\eqref{eq:local-loja} around $\optpoint \in
  \optimalset$ with constants $\plexp$ and $\plconstant$.
  If an algorithm generates sequences $\sequence{x_k}$ that
  satisfy~\eqref{eq:sufficient-decrease-lyapunov} around $\optpoint$ then it satisfies
  the~\eqref{eq:bounded-path-length-prop} property at $\optpoint$ with
  \begin{align*}
    \boundpathlength(x) = \frac{1}{\lyapsufficientdecreaseconst (1 - \plexp) \sqrt{2\plconstant}} |\mfc(x) - \mfcopt|^{1 - \plexp}.
  \end{align*}
\end{lemma}

We include a proof of this statement in Appendix~\ref{sec:loja-discrete-proofs}
for completeness.
In fact, the algorithm would still satisfy
the~\eqref{eq:bounded-path-length-prop} property under the more general
\kurdykaloja{} assumption (see~\cite[\S3.2.3]{absil2005convergence}).
In practice, many first-order algorithms (including gradient descent with
constant step-sizes or with line-search) generate sequences with the strong
decrease condition~\eqref{eq:sufficient-decrease-lyapunov}, as shown
in~\cite[\S4]{absil2005convergence}.

\subsection{Asymptotic convergence rate}
\label{subsec:asymptoticrate}

To conclude this section, we briefly review classical linear convergence results
for gradient methods under the \pl{} assumption, as needed for
Section~\ref{sec:super-linear-conv}.
Proofs are in Appendix~\ref{sec:loja-conv-rate} for completeness.
It is well known that gradient descent with appropriate step-sizes converges
linearly to a minimum when $\mfc$ satisfies \pl{} globally and has a
Lipschitz continuous gradient~\citep{polyak1963gradient}.
The same arguments lead to an asymptotic linear convergence rate when \pl{}
holds only locally.
We say that the sequence $\sequence{x_k}$ satisfies the \emph{sufficient
  decrease property} with constant $\sufficientdecreaseconst > 0$ if
\begin{align}\label{eq:sufficient-decrease}
  \mfc(x_k) - \mfc(x_{k + 1}) \geq \sufficientdecreaseconst \|\grad \mfc(x_k)\|^2.
\end{align}
whenever $x_k$ is sufficiently close to a point $\optpoint$.
The classical result below follows from that
inequality~\citep{polyak1963gradient}.

\begin{proposition}\label{prop:linear-convergence-sufficient-decrease}
  Let $\sequence{x_k}$ be a sequence of iterates converging to some $\optpoint
  \in \optimalset$ and satisfying~\eqref{eq:sufficient-decrease}.
  Suppose $\mfc$ satisfies~\eqref{eq:local-pl} around $\optpoint$ with constant
  $\plconstant > 0$.
  Then the sequence $\sequence{\mfc(x_k)}$ converges linearly to $\mfcopt$ with
  rate $1 - 2 \sufficientdecreaseconst \plconstant$.
  Moreover, $\sequence{\|\grad \mfc(x_k)\|}$ and $\sequence{\dist(x_k,
    \optimalset)}$ converge linearly to zero with rate $\sqrt{1 - 2
    \sufficientdecreaseconst \plconstant} \leq 1 - \sufficientdecreaseconst
  \plconstant$.
\end{proposition}

In the case where $\manifold$ is a Euclidean space and $\retr_x(s) = x + s$, it
is well known that the sufficient decrease
condition~\eqref{eq:sufficient-decrease} holds for many first-order algorithms
when $\grad \mfc$ is Lipschitz continuous.
This is also true for a general manifold $\manifold$ and retraction $\retr$ as
we briefly describe now.
We say that $\mfc$ and the retraction $\retr$ locally satisfy a
\emph{Lipschitz-type} property around $\optpoint \in \optimalset$ if there
exists $\liplikegd > 0$ such that
\begin{align}\label{eq:lipschitz-like}
  \mfc(\retr_x(s)) \leq \mfc(x) + \inner{\grad \mfc(x)}{s} + \frac{\liplikegd}{2}\|s\|^2
\end{align}
for all $x$ close enough to $\optpoint$ and $s$ small enough.
Note that if $\mfc \circ \retr$ is $\smooth{2}$ then the
inequality~\eqref{eq:lipschitz-like} is always (locally) satisfied.
It is a classical result that~\eqref{eq:lipschitz-like} implies sufficient
decrease for gradient descent with constant step-sizes.
This yields the following statement.

\begin{proposition}\label{prop:sufficient-decrease-gd}
  Suppose $\mfc$ satisfies $\plconstant$-\eqref{eq:local-pl} around $\optpoint
  \in \optimalset$.
  Also assume that~\eqref{eq:lipschitz-like} holds around $\optpoint$.
  Let $\sequence{x_k}$ be a sequence of iterates generated by gradient descent
  with constant step-size $\gamma \in \interval[open]{0}{\frac{2}{\liplikegd}}$,
  that is, $x_{k + 1} = \retr_{x_k}(-\gamma \grad \mfc(x_k))$.
  Given a neighborhood $\mathcal{U}$ of $\optpoint$, there exists a neighborhood
  $\mathcal{V}$ of $\optpoint$ such that if an iterate enters $\mathcal{V}$ then
  the sequence converges linearly to some $x_\infty \in \mathcal{U} \cap
  \optimalset$ with rate $\sqrt{1 - 2\plconstant(\gamma -
    \frac{\liplikegd}{2}\gamma^2)}$.
\end{proposition}

\section{Aiming for superlinear convergence}\label{sec:super-linear-conv}

Under fairly general assumptions, the \pl{} condition (which is compatible with
non-isolated minima) ensures stability of minima and linear convergence for
first-order methods, as recalled in the previous section.
We now assume $\mfc$ is $\smooth{2}$ and investigate superlinear convergence to
non-isolated minima.

A natural starting point is Newton's method which,
in spite of terrible global behavior~\citep{jarre2016simple},
enjoys quadratic convergence to a
\emph{non-singular} minimum, provided
the method is initialized sufficiently close.
Unfortunately, this does not extend to non-isolated minima.

We exhibit here an example showing that the \mb{} property is in general not
sufficient to ensure such a strong convergence behavior.
The update rule is $x_{k + 1} = x_k -\hess \mfc(x_k)^{-1}[\grad \mfc(x_k)]$ (we
may use the pseudo-inverse instead).
Consider the cost function $\mfc(x, y) = \frac{1}{2}(x^2 + 1)y^2$, whose set of
minima is the line $\optimalset = \{(x, y) \in \reals^2 : y = 0\}$.
The gradient and Hessian of $\mfc$ are
\begin{align*}
  \grad \mfc(x, y) =
  \begin{bmatrix}
    xy^2\\ (x^2 + 1)y
  \end{bmatrix}
  \qquad\text{and}\qquad
  \hess \mfc(x, y) =
  \begin{bmatrix}
    y^2 & 2xy\\ 2xy & x^2 + 1
  \end{bmatrix}.
\end{align*}
One can check that $\mfc$ satisfies~\eqref{eq:morse-bott}.
To see how Newton's method behaves on $\mfc$, notice that
\begin{align*}
  \hess \mfc(x, y)^{-1} \grad \mfc(x, y) =
  \frac{1}{3x^2 - 1}
  \begin{bmatrix}
    x^3 + x\\ (x^2 - 1)y
  \end{bmatrix}
\end{align*}
whenever $3x^2 \neq 1$.
Let $x(t) = \sqrt{\frac{1 - t}{3}}$ and $y(t) = \sqrt{t}$.
We can choose $t \in \interval[open]{0}{1}$ as small as desired to make the
point $(x(t), y(t))$ arbitrarily close to $\optimalset$.
Yet computing the Newton step at $(x(t), y(t))$ results in a new point at a
distance $\frac{2}{3}\frac{1 - t}{\sqrt{t}}$ from the optimal set $\optimalset$:
that is arbitrarily far away.
The failure of Newton's method stems from a misalignment between the gradient
and some eigenspaces of the Hessian.

The usual fix for Newton's method is to regularize it.
This yields two classes of algorithms in particular: regularized Newton with cubics and
trust-region methods.
We will show that cubic regularization enjoys satisfying local convergence properties, even in
the presence of non-isolated minima.
In contrast, the picture is less clear for trust-region methods.

Throughout, we make several local assumptions around a point $\optpoint \in
\optimalset$ as stated below.
The first two are Lipschitz-type properties.
\begin{assumption}\label{assu:hess-lip}
  The Hessian $\hess \mfc$ is locally $\hesslipconstant$-Lipschitz continuous
  around $\optpoint$ for some $\hesslipconstant \geq 0$.
\end{assumption}
\begin{assumption}\label{assu:hess-lipschitz-like}
  There exists a constant $\fliplikeconst \geq 0$ such that the Lipschitz-type
  inequality
  \begin{align*}
    \mfc(\retr_{x}(s)) - \mfc(x) - \inner{s}{\grad \mfc(x)} - \frac{1}{2}\inner{s}{\hess \mfc(x)[s]} \leq \frac{\fliplikeconst}{6}\|s\|^3
  \end{align*}
  holds for all $x$ close enough to $\optpoint$ and all $s$ small enough.
\end{assumption}
These assumptions typically hold (locally).
For the retraction $\retr = \Exp$, the first implies the other (with the same constant).
This is true in particular when $\manifold$ is a Euclidean space and $\retr_x(s)
= x + s$.
The third assumption concerns the two classes of algorithms we consider.
At every iterate $x_k$, they build a local model of the cost function around
$x_k$ based on a linear map $\linearmap_k$.
We require it to be close to the Hessian $\hess \mfc(x_k)$.
That holds in particular if $\linearmap_k = \hess \mfc(x_k)$.
\begin{assumption}\label{assu:hess-approx}
  For all $k$ the map $\linearmap_k$ is linear, symmetric, and there is a
  constant $\hessapproxconst \geq 0$ such that
  \begin{align*}
    \|\linearmap_k - \hess \mfc(x_k)\| \leq \hessapproxconst \|\grad \mfc(x_k)\|
  \end{align*}
  whenever the iterate $x_k$ is close enough to $\optpoint$.
\end{assumption}
Finally, we let $\retrdistboundconst \geq 1$ be such that the retraction
satisfies~\eqref{eq:retr-distance-condition} for sufficiently small tangent vectors (which is enough for the local analyses below).

\subsection{Adaptive regularized Newton}\label{subsec:arc}

The regularized Newton method using cubics was introduced
by~\cite{griewank1981modification} and later revisited
by~\cite{nesterov2006cubic}.
An adaptive version of this algorithm was proposed
by~\cite{cartis2011adaptive,cartis2011adaptive2}, with extensions to manifolds
by~\cite{qi2011numerical,zhang2018cubic}
and~\cite{agarwal2021adaptive}.
The adaptive variants update the penalty weight automatically: they are called
ARC.
We consider those variants, and more specifically an algorithm that generates
sequences $\sequence{(x_k, \cubicpenalty_k)}$, where $x_k$ is the
current iterate and $\cubicpenalty_k$ is
the cubic penalty weight.
The update rule is $x_{k + 1} = \retr_{x_k}(s_k)$ for some step $s_k$.
At each iteration $k$, we define a linear operator $\linearmap_k \colon
\tangent_{x_k}\manifold \to \tangent_{x_k}\manifold$ and the step $s_k$ is
chosen to approximately minimize the regularized second-order model
\begin{align}
  \arcmodel_k(s) = \mfc(x_k) + \inner{s}{\grad \mfc(x_k)} + \frac{1}{2}\inner{s}{\linearmap_k[s]} + \frac{\cubicpenalty_k}{3} \|s\|^3
  \label{eq:modelmk}
\end{align}
in a way that we make precise below.
We require $\linearmap_k$ to be close to $\hess \mfc(x_k)$ as prescribed
in~\aref{assu:hess-approx}.

In the literature, there are a number of superlinear (but non-quadratic)
convergence results for such algorithms.
Assuming the \pl{} condition,~\cite{nesterov2006cubic} showed that regularized
Newton generates sequences that converge superlinearly, with exponent $4/3$.
Later, assuming a \loja{} inequality with exponent $\plexp \in \interval[open
right]{0}{1}$,~\cite{zhou2018convergence} characterized the convergence speed of
regularized Newton depending on $\plexp$.
In particular, they also show that the \pl{} condition implies superlinear
convergence.\footnote{
  In fact,~\cite{zhou2018convergence} show a superlinear convergence rate only
  assuming that $\plexp \in \interval[open]{0}{\frac{2}{3}}$ and that $\hess
  \mfc$ is Lipschitz continuous.
  However, these assumptions imply \pl{} as noted in
  Remark~\ref{remark:other-loja-exponents}.
}
More recently,~\cite{qian2022superlinear} developed an abstract framework that
encompasses these superlinear convergence results,
and~\citet[\S5.3]{cartis2022evaluation} reviewed superlinear convergence rates
of ARC under \loja{} inequalities.

There is also a quadratic convergence result:~\cite{yue2019quadratic} employed a
local error bound assumption to show quadratic convergence for the regularized
Newton method.
As discussed in Section~\ref{sec:equiv-properties}, this assumption is
equivalent to local \pl{}, making their result an improvement over the superlinear
rates from the aforementioned references.
This underlines one of the benefits of recognizing the equivalence of the four
conditions \mb{}, \pl{}, \eb{} and \qg{}, as some may more readily lead to a
sharp analysis than others.

Note that the results in~\citep{yue2019quadratic} assume that the subproblem is
solved exactly, meaning that $s_k \in \argmin_s \arcmodel_k(s)$.
Several authors proposed weaker conditions on $s_k$ (only requiring an
approximate solution to the subproblem) to ensure convergence guarantees: this
is important because we cannot find an exact solution in practice.
\cite{agarwal2021adaptive} for example,
following \cite{birgin2017worst},
establish global convergence
guarantees assuming only that $\arcmodel_k(s_k) \leq \arcmodel_k(\zeros)$ and
$\|\grad \arcmodel_k(s_k)\| \leq \arctheta \|s_k\|^2$ for some $\arctheta \geq
0$.
For their results to hold, they require $\linearmap_k = \hess
\hatmfc_k(\zeros)$, where $\hatmfc_k = \mfc \circ \retr_x$ is the pullback of
$\mfc$ at $x$.
This choice of $\linearmap_k$ is compatible with~\aref{assu:hess-approx} for
retractions with bounded initial acceleration (which is typical).

We revisit the results of~\cite{yue2019quadratic} to obtain
an asymptotic quadratic convergence rate for ARC under the \pl{} assumption,
even with approximate solutions to the subproblem.
Specifically, throughout this section we suppose that the steps $s_k$ satisfy
\begin{align}\label{eq:arc-dynamics}
  \arcmodel_k(s_k) \leq \arcmodel_k(\zeros) && \text{and} && \|\grad \arcmodel_k(s_k)\| \leq \arctheta \|s_k\| \|\grad \mfc(x_k)\|
\end{align}
for some $\arctheta \geq 0$.
At each iteration $k$, we define the ratio
\begin{align}
  \arcrho_k = \frac{\mfc(x_k) - \mfc(\retr_{x_k}(s_k))}{\arcmodel_k(\zeros) - \arcmodel_k(s_k) + \frac{\cubicpenalty_k}{3}\|s_k\|^3}
  \label{eq:arcrhok}
\end{align}
(as do~\cite{birgin2017worst}), which measures the adequacy of the local model.
Iteration $k$ is said to be \emph{successful} when $\arcrho_k$ is larger than some fixed
parameter $\arcrho_c \in \interval[open]{0}{1}$.
In this case, we set $x_{k + 1} = \retr_{x_k}(s_k)$ and decrease the penalty
weight so that $\cubicpenalty_{k + 1} \leq \cubicpenalty_k$.
The update mechanism ensures that $\cubicpenalty_k \geq \cubicpenalty_{\min}$
for all $k$, where $\cubicpenalty_{\min} > 0$ is a fixed parameter.
Conversely, the step is \emph{unsuccessful} when $\arcrho_k < \arcrho_c$: we set $x_{k
  + 1} = x_k$ and increase the penalty so that $\cubicpenalty_{k + 1} >
\cubicpenalty_k$.
The explicit updates for $\cubicpenalty_{k + 1}$ are stated
in~\citep{agarwal2021adaptive}.
We prove the following result for this algorithm.

\begin{theorem}\label{th:arc-main-theorem}
  Suppose~\aref{assu:hess-lip}, \aref{assu:hess-lipschitz-like},
  \aref{assu:hess-approx} and~\eqref{eq:local-pl} hold around $\optpoint \in
  \optimalset$.
  We run ARC with inexact subproblem solver satisfying~\eqref{eq:arc-dynamics}.
  Given any neighborhood $\mathcal{U}$ of $\optpoint$, there exists a
  neighborhood $\mathcal{V}$ of $\optpoint$ such that if an iterate enters
  $\mathcal{V}$ then the sequence converges quadratically to some $x_\infty \in
  \mathcal{U} \cap \optimalset$.
\end{theorem}

We first adapt an argument from~\cite[Lem.~6]{agarwal2021adaptive} to show that
ARC satisfies the vanishing steps property~\eqref{eq:vanishing-steps} defined in
Section~\ref{subsec:capture}.

\begin{lemma}\label{lemma:arc-vanishing-condition}
  Suppose~\aref{assu:hess-approx} holds around a point $\optpoint$.
  There exists a neighborhood $\mathcal{U}$ of $\optpoint$ such that if an
  iterate $x_k$ is in $\mathcal{U}$ then the step-size has norm bounded as
  $\|s_k\| \leq \tilde \boundvanishingsteps(x_k, \cubicpenalty_k) \leq \tilde
  \boundvanishingsteps(x_k, \cubicpenalty_{\min})$, where
  \begin{align*}
    \tilde \boundvanishingsteps(x, \cubicpenalty) = \sqrt{\frac{3\|\grad \mfc(x)\|}{\cubicpenalty}} + \frac{3}{2\cubicpenalty}\absmineig(x) && \text{and} && \absmineig(x) = \max\!\Big(0, \hessapproxconst \|\grad \mfc(x)\| - \lambda_{\min}(\hess \mfc(x))\Big).
  \end{align*}
  In particular, ARC has
  the~\eqref{eq:vanishing-steps} property around second-order critical points with
  $\boundvanishingsteps(x) = \retrdistboundconst \tilde \boundvanishingsteps(x,
  \cubicpenalty_{\min})$, where $\retrdistboundconst$
  controls possible retraction distortion as
  in~\eqref{eq:retr-distance-condition}.
\end{lemma}
\begin{proof}
  The model decrease in~\eqref{eq:arc-dynamics} and the model
  accuracy~\aref{assu:hess-approx} ensure
  \begin{align*}
    \cubicpenalty_k\|s_k\|^3 &\leq -3 \Big\langle s_k, \grad \mfc(x_k) + \frac{1}{2}\hess \mfc(x_k)[s_k] + \frac{1}{2}(\linearmap_k - \hess \mfc(x_k))[s_k]\Big\rangle\\
                                  &\leq 3 \|s_k\|\Big(\|\grad \mfc(x_k)\| + \frac{1}{2}\absmineig(x_k)\|s_k\|\Big).
  \end{align*}
  Divide by $\|s_k\|$ and solve the quadratic inequality for $\|s_k\|$ to get the result,
  recalling $\cubicpenalty_{k} \geq \cubicpenalty_{\min}$. %
\end{proof}

The function $\boundvanishingsteps$ is indeed continuous with value zero on
$\optimalset$, as required in~\eqref{eq:vanishing-steps}.
To obtain the vanishing steps property we only relied on the decrease
requirement $\arcmodel_k(s_k) \leq \arcmodel_k(\zeros)$.
Assuming a \pl{} condition and a locally Lipschitz continuous Hessian, we now
derive sharper bounds for the steps.
We rely on the fact that the gradient of the model
$\arcmodel_k$~\eqref{eq:modelmk}
\begin{align}\label{eq:arc-model-gradient}
  \grad \arcmodel_k(s_k) = \grad \mfc(x_k) + \linearmap_k[s_k] + \cubicpenalty_k \|s_k\| s_k
\end{align}
is small.
We will exploit the particular alignment of $\grad \mfc$ given in
Lemma~\ref{lemma:grad-image-hess}.
In the following statements, given a point $\optpoint \in \optimalset$ and
$\rankop = \rank(\hess \mfc(\optpoint))$, we let $P(x) \colon
\tangent_x\manifold \to \tangent_x\manifold$ denote the orthogonal projector
onto the top $d$ eigenspace of $\hess \mfc(x)$.
This is always well defined in a neighborhood of $\optpoint$ by continuity of
eigenvalues.
Additionally, we let $Q(x) = I - P(x)$ be the projector onto the orthogonal
complement.

\begin{lemma}\label{lemma:arc-s-dist-bound}
  Suppose that~\aref{assu:hess-lip}, \aref{assu:hess-approx}
  and~$\plconstant$-\eqref{eq:local-pl} hold around $\optpoint \in \optimalset$.
  Given $\varepsilon > 0$, $\lamflat < \plconstant$ and $\lamsharp >
  \lambda_{\max}(\hess \mfc(\optpoint))$, there exists a neighborhood
  $\mathcal{U}$ of $\optpoint$ and a constant $\qgradconst \geq 0$ such that if
  $x_k$ is an iterate in $\mathcal{U}$ and $s_k$ is a step
  satisfying~\eqref{eq:arc-dynamics} then
  \begin{align}\label{eq:arc-s-bound-p}
    (1 - \varepsilon)\frac{\|\grad \mfc(x_k)\|}{\lamsharp + \cubicpenalty_k\|s_k\|} \leq \|P(x_k)s_k\| \leq (1 + \varepsilon)\frac{\|\grad \mfc(x_k)\|}{\lamflat + \cubicpenalty_k\|s_k\|} \qquad\text{and}
  \end{align}
  \begin{align}\label{eq:arc-s-bound-q}
    \|Q(x_k)s_k\| \leq \frac{1}{\cubicpenalty_k}\Big((\arctheta + \hessapproxconst)\|\grad \mfc(x_k)\| + (\hesslipconstant + \qgradconst\sqrt{\cubicpenalty_k})\dist(x_k, \optimalset)\Big).
  \end{align}
\end{lemma}
\begin{proof}
  Let $\mathcal{U}$ be a neighborhood of $\optpoint$ where the orthogonal
  projector $P(x)$ is well defined for all $x \in \mathcal{U}$.
  Let $x_k \in \mathcal{U}$ and $s_k$ be a step that
  satisfies~\eqref{eq:arc-dynamics}.
  We first bound the term $P(x_k)s_k$.
  Multiply~\eqref{eq:arc-model-gradient} by $P(x_k)$ and use commutativity of $P(x_k)$ and $\hess \mfc(x_k)$ to get
  \begin{align*}
    \Big(\hess \mfc(x_k) + \cubicpenalty_k \|s_k\| I\Big)P(x_k) s_k = P(x_k) \Big(- \grad \mfc(x_k) + \grad \arcmodel_k(s_k) - \big(\linearmap_k - \hess \mfc(x_k)\big)[s_k]\Big).
  \end{align*}
  If we apply~\aref{assu:hess-approx} and~\eqref{eq:arc-dynamics}, we find that
  $\|\grad \arcmodel_k(s_k) - (\linearmap_k - \hess \mfc(x_k))[s_k]\| \leq
  (\arctheta + \hessapproxconst)\|\grad \mfc(x_k)\|\|s_k\|$ when $x_k$ is close
  enough to $\optpoint$ (shrink $\mathcal{U}$ as needed).
  Consequently, the previous equality yields
  \begin{align*}
    \frac{\|P(x_k)\grad \mfc(x_k)\| - (\arctheta + \hessapproxconst)\|\grad \mfc(x_k)\|\|s_k\|}{\lambda_1(\hess \mfc(x_k)) + \cubicpenalty_k\|s_k\|} \leq \|P(x_k)s_k\| \leq \frac{\|P(x_k)\grad \mfc(x_k)\| + (\arctheta + \hessapproxconst)\|\grad \mfc(x_k)\|\|s_k\|}{\lambda_d(\hess \mfc(x_k)) + \cubicpenalty_k\|s_k\|}.
  \end{align*}
  Lemma~\ref{lemma:grad-image-hess} gives that $\|P(x)\grad \mfc(x)\| = \|\grad
  \mfc(x)\| + o(\|\grad \mfc(x)\|)$ as $x \to \optpoint$.
  Moreover, the steps $s_k$ vanish (as shown in
  Lemma~\ref{lemma:arc-vanishing-condition}), so we obtain the
  bound~\eqref{eq:arc-s-bound-p} when $x_k$ is sufficiently close to
  $\optpoint$.
  We now let $Q(x) = I - P(x)$ and consider the term $Q(x_k)s_k$.
  Multiply~\eqref{eq:arc-model-gradient} by $Q(x_k)$ to obtain
  \begin{align*}
    Q(x_k) \grad \arcmodel_k(s_k) = Q(x_k) \grad \mfc(x_k) + \hess \mfc(x_k) Q(x_k) s_k + Q(x_k)(\linearmap_k - \hess \mfc(x_k))[s_k] + \cubicpenalty_k \|s_k\| Q(x_k) s_k.
  \end{align*}
  Taking the inner product of this expression with $Q(x_k)s_k$, applying the
  \causchwarz{} inequality, dividing by $\|s_k\|$, and
  using~\aref{assu:hess-approx} yields
  \begin{align*}
    \cubicpenalty_k \|Q(x_k)s_k\|^2 \leq \|Q(x_k)\grad \mfc(x_k)\| + \Big(\arctheta \|\grad \mfc(x_k)\| + \absmineig(x_k)\Big) \|Q(x_k)s_k\|,
  \end{align*}
  where $\absmineig$ is as in Lemma~\ref{lemma:arc-vanishing-condition}.
  Solving the quadratic inequality gives
  \begin{align*}
    \|Q(x_k)s_k\| \leq \frac{1}{\cubicpenalty_k}\Big(\sqrt{\cubicpenalty_k\|Q(x_k)\grad \mfc(x_k)\|} + \arctheta \|\grad \mfc(x_k)\| + \absmineig(x_k)\Big).
  \end{align*}
  Local Lipschitz continuity of $\hess \mfc$ provides
  $\absmineig(x_k) \leq \hessapproxconst \|\grad \mfc(x_k)\| +
  \hesslipconstant \dist(x_k, \optimalset)$ when $x_k$ is close to $\optpoint$.
  Via Lemma~\ref{lemma:grad-image-hess}, it also provides a constant
  $\qgradconst \geq 0$ such that $\sqrt{\|Q(x_k)\grad \mfc(x_k)\|} \leq
  \qgradconst \dist(x_k, \optimalset)$ when $x_k$ is sufficiently close to
  $\optpoint$. This is enough to secure~\eqref{eq:arc-s-bound-q}.
\end{proof}

Using these bounds, we now show that $\|P(x_k)s_k\|$ cannot be too small
compared to $\|s_k\|$ when $x_k$ is close to a minimum where \pl{} holds.

\begin{lemma}\label{lemma:ratio-steps}
  Suppose that~\aref{assu:hess-lip}, \aref{assu:hess-approx}
  and~$\plconstant$-\eqref{eq:local-pl} hold around $\optpoint \in \optimalset$.
  Given $\varepsilon > 0$ and $\lamsharp > \lambda_{\max}(\hess
  \mfc(\optpoint))$, there exists a neighborhood $\mathcal{U}$ of $\optpoint$
  and a constant $\qgradconst \geq 0$ such that if $x_k$ is an iterate in
  $\mathcal{U}$ and $s_k$ is a step satisfying~\eqref{eq:arc-dynamics} then
  $\|P(x_k)s_k\| \geq \ratiosteps \|s_k\|$ where $\ratiosteps > 0$ is the constant
  \begin{align*}
    \ratiosteps = \sqrt{1 - \frac{1}{1 + \tilde \ratiosteps^2}} && \text{with} && \tilde \ratiosteps = \frac{(1 - \varepsilon)\cubicpenalty_{\min}}{\big(\lamsharp + \varepsilon(1 + \sqrt{\cubicpenalty_{\min}})\big)\big(\arctheta + \hessapproxconst + \frac{\hesslipconstant + \qgradconst\sqrt{\cubicpenalty_{\min}}}{\plconstant}\big)}.
  \end{align*}
\end{lemma}
\begin{proof}
  Assume $x_k$ is sufficiently close to $\optpoint$ for the projectors $P(x_k)$
  and $Q(x_k)$ to be well defined.
  Define $\nu_p = \|P(x_k)s_k\|$, $\nu_q = \|Q(x_k)s_k\|$ and $\xi =
  \frac{\nu_p}{\nu_q}$ (consider $\nu_q \neq 0$ as otherwise the claim is clear).
  We compute that $\nu_p^2 = \big(1 - \frac{1}{1 + \xi^2}\big)\|s_k\|^2$ and
  find a lower-bound for $\xi$.
  Remark~\ref{remark:pl-implies-eb-c1} gives that $\dist(x_k, \optimalset) \leq
  \frac{1}{\plconstant}\|\grad \mfc(x_k)\|$ when $x_k$ is sufficiently close to
  $\optpoint$.
  Together with the bound on $\nu_q$ in Lemma~\ref{lemma:arc-s-dist-bound}, this
  gives
  \begin{align*}
    \nu_q \leq \frac{\|\grad \mfc(x_k)\|}{\cubicpenalty_k}\Big(\arctheta + \hessapproxconst + \frac{\hesslipconstant + \qgradconst\sqrt{\cubicpenalty_k}}{\plconstant}\Big)
  \end{align*}
  Combining this with the lower-bound on $\nu_p$ in
  Lemma~\ref{lemma:arc-s-dist-bound}, we find
  \begin{align*}
    \xi \geq \frac{(1 - \varepsilon)\cubicpenalty_k}{(\lamsharp + \cubicpenalty_k\|s_k\|)\big(\arctheta + \hessapproxconst + \frac{\hesslipconstant + \qgradconst\sqrt{\cubicpenalty_k}}{\plconstant}\big)}.
  \end{align*}
  With Lemma~\ref{lemma:arc-vanishing-condition} we can upper-bound
  $\cubicpenalty_k\|s_k\| \leq \sqrt{3\cubicpenalty_k\|\grad \mfc(x_k)\|} +
  \frac{3}{2}\absmineig(x_k)$, and consequently, $\cubicpenalty_k\|s_k\| \leq
  \varepsilon (1 + \sqrt{\cubicpenalty_k})$ whenever $x_k$ is sufficiently close
  to $\optpoint$.
  We finally notice that the resulting lower-bound on $\xi$ is an increasing
  function of $\cubicpenalty_k$.
  Therefore, we get the desired inequality by using $\cubicpenalty_k \geq
  \cubicpenalty_{\min}$.
\end{proof}

From this we deduce a lower-bound on the quadratic term of the model.

\begin{lemma}\label{lemma:inner-product-lower-bound}
  Suppose that~\aref{assu:hess-lip}, \aref{assu:hess-approx} and
  $\plconstant$-\eqref{eq:local-pl} hold around $\optpoint \in \optimalset$.
  Given $\lamflat < \plconstant$, there exists a constant $\ratiosteps > 0$
  (provided by Lemma~\ref{lemma:ratio-steps}) and a neighborhood $\mathcal{U}$
  of $\optpoint$ such that if $x_k \in \mathcal{U}$ then the step satisfies
  \begin{align*}
    \inner{s_k}{\hess \mfc(x_k)[s_k]} \geq \lamflat \ratiosteps^2 \|s_k\|^2.
  \end{align*}
\end{lemma}
\begin{proof}
  Let $\ratiosteps > 0$ be a constant and $\mathcal{U}$ a neighborhood of
  $\optpoint$ as in Lemma~\ref{lemma:ratio-steps}.
  Shrink $\mathcal{U}$ for the projectors $P$ and $Q$ to be well defined in
  $\mathcal{U}$.
  If $x_k$ is in $\mathcal{U}$ we compute
  \begin{align*}
    \inner{s_k}{\hess \mfc(x_k)[s_k]} &= \inner{P(x_k)s_k}{\hess \mfc(x_k)[P(x_k)s_k]} + \inner{Q(x_k)s_k}{\hess \mfc(x_k)[Q(x_k)s_k]}\\
                                      &\geq \lambda_d(\hess \mfc(x_k))\ratiosteps^2\|s_k\|^2 + \lambda_{\min}(\hess \mfc(x_k))\|Q(x_k)s_k\|^2,
  \end{align*}
  where we used $\|P(x_k)s_k\| \geq r\|s_k\|$.
  We obtain the result by noticing that the second term is lower-bounded by
  $\max\!\big(0, -\lambda_{\min}(\hess \mfc(x_k))\big)\|s_k\|^2$.
\end{proof}

We now show that the ratio $\arcrho_k$ is large when $x_k$ is close to a local
minimum where \pl{} holds.
The upshot is that iterations near $\optimalset$ are successful.
\begin{lemma}\label{lemma:arc-eventually-successful}
  Suppose that~\aref{assu:hess-lip}, \aref{assu:hess-lipschitz-like},
  \aref{assu:hess-approx} and $\plconstant$-\eqref{eq:local-pl} hold around
  $\optpoint \in \optimalset$.
  For all $\varepsilon > 0$ there exists a neighborhood $\mathcal{U}$ of
  $\optpoint$ such that if $x_k \in \mathcal{U}$ then $\arcrho_k \geq 1 -
  \varepsilon$.
\end{lemma}
\begin{proof}
  Let $\taylor_k(s) = \mfc(x_k) + \inner{s}{\grad \mfc(x)} +
  \frac{1}{2}\inner{s}{\linearmap_k[s]}$ be the second-order component of the
  model $\arcmodel_k$.
  Then,
  \begin{align*}
    1 - \arcrho_k = \frac{\mfc(\retr_{x_k}(s_k)) - \taylor_k(s_k)}{\taylor_k(\zeros) - \taylor_k(s_k)} \leq \frac{1}{6} \frac{\fliplikeconst\|s_k\|^3 + 3\hessapproxconst\|s_k\|^2\|\grad \mfc(x_k)\|}{\taylor_k(\zeros) - \taylor_k(s_k)},
  \end{align*}
  where the bound on the numerator comes from \aref{assu:hess-lipschitz-like}
  and~\aref{assu:hess-approx}.
  The denominator is given by
  \begin{align*}
    \taylor_k(\zeros) - \taylor_k(s_k) &= -\inner{s_k}{\grad \mfc(x_k)} - \frac{1}{2}\inner{s_k}{\linearmap_k[s_k]}\\
                                       &= -\inner{s_k}{\grad \arcmodel_k(s_k)} + \frac{1}{2}\inner{s_k}{(\linearmap_k - \hess \mfc(x_k))[s_k]} + \frac{1}{2}\inner{s_k}{\hess \mfc(x_k)[s_k]} + \cubicpenalty_k\|s_k\|^3,
  \end{align*}
  where we used identity~\eqref{eq:arc-model-gradient} for the second equality.
  With the \causchwarz{} inequality, \aref{assu:hess-approx}
  and~\eqref{eq:arc-dynamics}, we bound the two first terms as
  \begin{align*}
    |\inner{s_k}{\grad \arcmodel_k(s_k)}| \leq \arctheta \|\grad \mfc(x_k)\| \|s_k\|^2 && \text{and} && |\inner{s_k}{(\linearmap_k - \hess \mfc(x_k))[s_k]}| \leq \hessapproxconst \|\grad \mfc(x_k)\| \|s_k\|^2.
  \end{align*}
  If we combine these bounds with Lemma~\ref{lemma:inner-product-lower-bound},
  we find that given $\lamflat < \plconstant$, there exists $\ratiosteps > 0$
  and a neighborhood $\mathcal{U}$ of the local minimum $\optpoint$ such that
  \begin{align*}
    x_k \in \mathcal{U} \;\;\Rightarrow\;\; \taylor_k(\zeros) - \taylor_k(s_k) \geq \lamflat \ratiosteps^2 \|s_k\|^2,
  \end{align*}
  and therefore $1 - \arcrho_k \leq \frac{\fliplikeconst \|s_k\| +
    3\hessapproxconst\|\grad \mfc(x_k)\|}{6 \lamflat \ratiosteps^2}$.
  Owing to Lemma~\ref{lemma:arc-vanishing-condition} the steps $s_k$ vanish
  around second-order critical points so we can guarantee $1 - \arcrho_k$
  becomes as small as desired.
\end{proof}

We can now show that the iterates produced by ARC satisfy the strong decrease
property~\eqref{eq:sufficient-decrease-lyapunov} around minima where \pl{}
holds.

\begin{proposition}
  Suppose that~\aref{assu:hess-lip}, \aref{assu:hess-lipschitz-like},
  \aref{assu:hess-approx} and~$\plconstant$-\eqref{eq:local-pl} hold around
  $\optpoint \in \optimalset$.
  Given $\lamflat < \plconstant$ and $\lamsharp > \lambda_{\max}(\hess
  \mfc(\optpoint))$, there exists a neighborhood of $\optpoint$ where ARC
  satisfies the strong decrease property~\eqref{eq:sufficient-decrease-lyapunov}
  with constant $\lyapsufficientdecreaseconst =
  \frac{\ratiosteps\lamflat}{2\retrdistboundconst\lamsharp}$, where
  $\retrdistboundconst$ is defined in~\eqref{eq:retr-distance-condition} and
  $\ratiosteps > 0$ is provided by Lemma~\ref{lemma:ratio-steps}.
\end{proposition}
\begin{proof}
  From Lemma~\ref{lemma:arc-eventually-successful}, there exists a neighborhood
  $\mathcal{U}$ of $\optpoint$ in which all the steps are successful.
  Given an iterate $x_k$ in $\mathcal{U}$, success implies $x_{k+1} =
  \retr_{x_k}(s_k)$ and therefore, by definition of
  $\arcrho_k$~\eqref{eq:arcrhok},
  \begin{align*}
    \mfc(x_k) - \mfc(x_{k + 1}) = \arcrho_k\Big(\arcmodel_k(\zeros) - \arcmodel_k(s_k) + \frac{\cubicpenalty_k}{3}\|s_k\|^3\Big) = -\arcrho_k\Big(\inner{s_k}{\grad \mfc(x_k)} + \frac{1}{2}\big\langle s_k, \linearmap_k[s_k]\big\rangle\Big).
  \end{align*}
  Also, taking the inner product of~\eqref{eq:arc-model-gradient} with $s_k$ and
  using~\eqref{eq:arc-dynamics} yields
  \begin{align*}
    \inner{s_k}{\linearmap_k[s_k]} \leq -\inner{s_k}{\grad \mfc(x_k)} + \arctheta \|\grad \mfc(x_k)\|\|s_k\|^2.
  \end{align*}
  Multiply the latter by $-\frac{\arcrho_k}{2}$ and plug into the former to
  deduce
  \begin{align}
    \mfc(x_k) - \mfc(x_{k + 1}) \geq -\frac{\arcrho_k}{2}\Big(\inner{s_k}{\grad \mfc(x_k)} + \arctheta \|\grad \mfc(x_k)\|\|s_k\|^2\Big).
    \label{eq:fdecreasearc}
  \end{align}
  We now bound the inner product $\inner{s_k}{\grad \mfc(x_k)}$.
  Let $\rankop = \rank \hess \mfc(\optpoint)$ and restrict the neighborhood
  $\mathcal{U}$ if need be to ensure $\lambda_{\rankop}(\hess\mfc(x_k)) > 0$ and
  $\lambda_{\rankop}(\hess\mfc(x_k)) > \lambda_{\rankop+1}(\hess\mfc(x_k))$.
  In particular, the orthogonal projector $P(x_k)$ onto the top $\rankop$
  eigenspace of $\hess \mfc(x_k)$ is well defined.
  Let $Q(x_k) = I - P(x_k)$.
  Decompose $s_k = P(x_k)s_k + Q(x_k)s_k$ and apply the \causchwarz{} inequality
  to obtain
  \begin{align}
    \inner{s_k}{\grad \mfc(x_k)} \leq \inner{P(x_k)s_k}{\grad \mfc(x_k)} + \|Q(x_k)\grad \mfc(x_k)\|\|s_k\|.
    \label{eq:innerskgradfxkarc}
  \end{align}
  The second term is small owing to Lemma~\ref{lemma:grad-image-hess}.
  Let us focus on the first term.
  To this end, multiply~\eqref{eq:arc-model-gradient} by $P(x_k)$ to verify the
  following (recall that $\hess\mfc(x_k)$ and $P(x_k)$ commute):
  \begin{align*}
      P(x_k) \grad\mfc(x_k) & = -\hess\mfc(x_k)P(x_k)s_k + P(x_k)\grad\arcmodel_k(s_k) - P(x_k)\big(\linearmap_k - \hess\mfc(x_k)\big)s_k - \cubicpenalty_k \|s_k\| P(x_k) s_k.
  \end{align*}
  On the one hand, we can use it with~\aref{assu:hess-approx}
  and~\eqref{eq:arc-dynamics} to lower-bound the norm of $P(x_k)s_k$, through:
  \begin{align*}
      \|P(x_k) \grad\mfc(x_k)\| & \leq \left(\lambda_1(\hess\mfc(x_k)) + \cubicpenalty_k \|s_k\|\right) \|P(x_k)s_k\| + (\arctheta + \hessapproxconst) \|\grad\mfc(x_k)\|\|s_k\|.
  \end{align*}
  On the other hand, we can use it to upper-bound
  $\inner{s_k}{P(x_k)\grad\mfc(x_k)}$, also with~\aref{assu:hess-approx}
  and~\eqref{eq:arc-dynamics}, and using the fact that $P(x_k)s_k$ lives in the
  top-$\rankop$ eigenspace of $\hess\mfc(x_k)$, like so:
  \begin{align*}
      \inner{s_k}{P(x_k)\grad\mfc(x_k)} & \leq -\left(\lambda_{\rankop}(\hess\mfc(x_k)) + \cubicpenalty_k \|s_k\|\right) \|P(x_k)s_k\|^2
      + (\arctheta + \hessapproxconst) \|\grad\mfc(x_k)\| \|s_k\| \|P(x_k)s_k\|.
  \end{align*}
  Combine the two inequalities above as follows: use the former to upper-bound
  one of the first factors $\|P(x_k)s_k\|$ in the latter.
  Also using
  $\frac{\lambda_{\rankop}(\hess\mfc(x_k))}{\lambda_{1}(\hess\mfc(x_k))} \leq
  \frac{\lambda_{\rankop}(\hess\mfc(x_k)) + \cubicpenalty_k
    \|s_k\|}{\lambda_{1}(\hess\mfc(x_k)) + \cubicpenalty_k \|s_k\|} \leq 1$,
  this yields:
  \begin{multline}
      \inner{s_k}{P(x_k)\grad\mfc(x_k)} \leq - \frac{\lambda_{\rankop}(\hess\mfc(x_k))}{\lambda_{1}(\hess\mfc(x_k))} \|P(x_k) \grad\mfc(x_k)\| \|P(x_k)s_k\| \\
      + 2 (\arctheta + \hessapproxconst) \|\grad\mfc(x_k)\| \|s_k\| \|P(x_k)s_k\|.
  \end{multline}
  We now plug this back into~\eqref{eq:innerskgradfxkarc}.
  Using Lemma~\ref{lemma:grad-image-hess} for its second term and also
  Lemma~\ref{lemma:arc-vanishing-condition} which asserts $\|s_k\|$ is
  arbitrarily small for $x_k$ near $\optpoint$, we find that for all
  $\varepsilon > 0$ we can restrict the neighborhood $\mathcal{U}$ in order to
  secure
  \begin{align}
      \inner{s_k}{\grad \mfc(x_k)} \leq - \frac{\lambda_{\rankop}(\hess\mfc(x_k))}{\lambda_{1}(\hess\mfc(x_k))} \|P(x_k) \grad\mfc(x_k)\| \|P(x_k)s_k\|
        + \varepsilon \|\grad\mfc(x_k)\| \|s_k\|.
  \end{align}
  Lemma~\ref{lemma:ratio-steps} provides a (possibly smaller) neighborhood and a
  positive $r$ such that $\|P(x_k)s_k\| \geq r \|s_k\|$.
  Also, for any $\delta > 0$, Lemma~\ref{lemma:grad-image-hess} ensures
  $\|P(x_k) \grad\mfc(x_k)\| \geq (1-\delta) \|\grad\mfc(x_k)\|$ upon
  appropriate neighborhood restriction.
  Thus,
  \begin{align}
    \inner{s_k}{\grad \mfc(x_k)} \leq - \left( \frac{\lambda_{\rankop}(\hess\mfc(x_k))}{\lambda_{1}(\hess\mfc(x_k))} (1-\delta) r - \varepsilon \right) \|\grad\mfc(x_k)\| \|s_k\|.
  \end{align}
  Now plugging back into~\eqref{eq:fdecreasearc} and possibly restricting the
  neighborhood again,
  \begin{align}
      \mfc(x_k) - \mfc(x_{k + 1}) & \geq \frac{\arcrho_k}{2} \left( \frac{\lambda_{\rankop}(\hess\mfc(x_k))}{\lambda_{1}(\hess\mfc(x_k))} (1-\delta) r - \varepsilon - \arctheta \|s_k\| \right) \|\grad\mfc(x_k)\| \|s_k\|.
  \end{align}
  The result now follows since we can arrange to make $\varepsilon, \delta$ and
  $\|s_k\|$ arbitrarily small; to make $\lambda_{\rankop}(\hess\mfc(x_k))$ and
  $\lambda_{1}(\hess\mfc(x_k))$ arbitrarily close to $\plconstant$ and
  $\lambda_{\max}(\hess\mfc(\optpoint))$ (respectively); and to make $\arcrho_k$
  larger than a number arbitrarily close to 1
  (Lemma~\ref{lemma:arc-eventually-successful}).
  The final step is to account for potential distortion due to a
  retraction~\eqref{eq:retr-distance-condition}: this adds the factor
  $\retrdistboundconst$.
\end{proof}

If we combine this result with Lemma~\ref{lemma:bound-discrete-gd-path-length}
we obtain that ARC satisfies the bounded path length
property~\eqref{eq:bounded-path-length-prop}.
In Lemma~\ref{lemma:arc-vanishing-condition} we found that it also
satisfies the vanishing steps property~\eqref{eq:vanishing-steps}.
Moreover, if the iterates of ARC stay in a compact region then they accumulate
only at critical points~\cite[Cor.~2.6]{cartis2011adaptive}.
As a result, Corollary~\ref{cor:capture-to-single-point} applies to ARC: if an
iterate gets close enough to a point where \pl{} holds then the sequence has a
limit.
We conclude this section with the quadratic convergence rate of ARC.

\begin{proposition}\label{prop:arc-converges-quadratically}
  Suppose that $\sequence{x_k}$ converges to some $\optpoint \in \optimalset$
  and that $\mfc$ is~\eqref{eq:local-pl} around $\optpoint$.
  Also assume that~\aref{assu:hess-lip}, \aref{assu:hess-lipschitz-like}
  and~\aref{assu:hess-approx} hold around $\optpoint$.
  Then $\sequence{\dist(x_k, \optimalset)}$ converges quadratically to zero.
\end{proposition}
\begin{proof}
  From Lemma~\ref{lemma:arc-eventually-successful} all the steps are eventually
  successful.
  In particular, the penalty weights eventually stop increasing: there exists
  $\cubicpenalty_{\max} > 0$ such that $\cubicpenalty_k \leq \cubicpenalty_{\max}$
  for all $k$.
  Let $\lamflat < \plconstant$ (where $\plconstant$ is the \pl{} constant)
  and $\lamsharp > \lambda_{\max}(\hess
  \mfc(\optpoint))$.
  We first apply the Pythagorean theorem with the upper-bounds from
  Lemma~\ref{lemma:arc-s-dist-bound}.
  Together with the upper-bound in Proposition~\ref{prop:grad-dist-bounds}, it
  implies for all large enough $k$ that
  \begin{align}\label{eq:arc-s2-bound}
    \|s_k\|^2 \leq \arcstepdistconst^2 \dist(x_k, \optimalset)^2
    && \textrm{ where } &&
    \arcstepdistconst^2 = \left(\frac{\lamsharp}{\lamflat}\right)^2 + \frac{1}{\cubicpenalty_{\min}^2}\!\left((\arctheta + \hessapproxconst)\lamsharp + \hesslipconstant + \qgradconst \sqrt{\cubicpenalty_{\min}}\right)^2.
  \end{align}
  We now let $v_k = s_k - \Log_{x_k}(x_{k + 1})$, which is always well defined
  for large enough $k$.
  Using \eb{} (given by Remark~\ref{remark:pl-implies-eb-c1}), we obtain that
  for all large enough $k$ we have
  \begin{align*}
    \dist(x_{k + 1}, \optimalset) &\leq \frac{1}{\plconstant} \|\grad \mfc(x_{k + 1})\|\\
                                  &= \frac{1}{\plconstant} \big\|\ptransport{x_{k + 1}}{x_k}\grad \mfc(x_{k + 1}) - \grad \mfc(x_k) - \hess \mfc(x_k)[\Log_{x_k}(x_{k + 1})]\\
                                  &\qquad\qquad\qquad- \hess \mfc(x_k)[v_k] - \big(\linearmap_k - \hess \mfc(x_k)\big)[s_k] - \cubicpenalty_k\|s_k\|s_k + \grad \arcmodel_k(s_k)\big\|,
  \end{align*}
  where we used identity~\eqref{eq:arc-model-gradient} for $\grad \arcmodel_k$.
  Now the triangle inequality, \aref{assu:hess-approx}
  and~\eqref{eq:arc-dynamics} give
  \begin{align}\label{eq:arc-next-dist-bound}
    \dist(x_{k + 1}, \optimalset) \leq \frac{1}{\plconstant}\Big(\frac{\hesslipconstant}{2}\dist(x_k, x_{k + 1})^2 + \lamsharp\|v_k\| + \cubicpenalty_k\|s_k\|^2 + (\arctheta + \hessapproxconst)\|\grad \mfc(x_k)\| \|s_k\|\Big).
  \end{align}
  Notice that $\dist(x_k, x_{k + 1}) \leq \retrdistboundconst \|s_k\|$
  using~\eqref{eq:retr-distance-condition}.
  We now bound the quantity $\|v_k\|$.
  For all $x \in \manifold$, since $\D \retr_x(\zeros) = I$, the inverse
  function theorem implies that $\retr_x$ is locally invertible and $\D
  \retr_x^{-1}(x) = I$.
  It follows that there exists a neighborhood $\mathcal{U} \subseteq
  \ball(\optpoint, \inj(\optpoint))$ of $\optpoint$ such that for all $x, y \in
  \mathcal{U}$ the quantity $\retr_x^{-1}(y)$ is well defined and satisfies
  \begin{align*}
    \retr_x^{-1}(y) = \retr_x^{-1}(x) + \D\retr_x^{-1}(x)[\Log_x(y)] + O(\dist(x, y)^2) = \Log_x(y) + O(\dist(x, y)^2).
  \end{align*}
  In particular, using the identity $s_k = \retr_{x_k}^{-1}(x_{k + 1})$, we find
  that there exists a constant $c_2$ such that $\|v_k\| \leq c_2\dist(x_k, x_{k
    + 1})^2$ holds for large enough $k$.
  Combining this with~\eqref{eq:arc-next-dist-bound}, we obtain
  \begin{align*}
    \dist(x_{k + 1}, \optimalset) \leq \frac{1}{\plconstant}\bigg(\Big(\frac{\hesslipconstant}{2}\retrdistboundconst^2 + \lamsharp c_2 \retrdistboundconst^2 + \cubicpenalty_k\Big)\|s_k\|^2 + (\arctheta + \hessapproxconst) \|\grad \mfc(x_k)\| \|s_k\|\bigg).
  \end{align*}
  Finally, using~\eqref{eq:arc-s2-bound} and the upper-bound from
  Proposition~\ref{prop:grad-dist-bounds}, we conclude that
  \begin{align*}
    \dist(x_{k + 1}, \optimalset) \leq \arcquadconvconst \dist(x_k, \optimalset)^2
      && \textrm{ where } &&
    \arcquadconvconst = \frac{\arcstepdistconst^2}{\plconstant} \Big(\frac{\hesslipconstant}{2}\retrdistboundconst^2 + \lamsharp c_2 \retrdistboundconst^2 + \cubicpenalty_{\max}\Big) + \frac{\lamsharp}{\plconstant}(\arctheta + \hessapproxconst)\arcstepdistconst. & \qedhere{}
  \end{align*}
\end{proof}

The quadratic convergence rates of the sequences $\sequence{\mfc(x_k)}$ and
$\sequence{\|\grad \mfc(x_k)\|}$ follow immediately by \qg{} and \eb{}.
Theorem~\ref{th:arc-main-theorem} is a direct consequence of
Corollary~\ref{cor:capture-to-single-point} and
Proposition~\ref{prop:arc-converges-quadratically}.

\subsection{Trust-region algorithms}\label{subsec:rtr}

In this section we analyze Riemannian trust-region algorithms (TR), which embed
Newton iterations in safeguards to ensure global convergence
guarantees~\citep{absil2007trust}.
They produce sequences $\sequence{(x_k, \Delta_k)}$, where $x_k$ is the
current iterate and $\Delta_k$ is the trust-region radius.
At iteration $k$, we define the trust-region model as
\begin{align}\label{eq:rtr-model}
  \rtrmodel_k(s) = \mfc(x_k) + \inner{s}{\grad \mfc(x_k)} + \frac{1}{2}\inner{s}{\linearmap_k[s]},
\end{align}
where $\linearmap_k \colon \tangent_{x_k}\manifold \to \tangent_{x_k}\manifold$
is a linear map close to $\hess \mfc(x_k)$,
satisfying~\aref{assu:hess-approx}.
The step $s_k$ is chosen by (usually approximately) solving the trust-region
subproblem
\begin{align}\label{eq:rtr-subproblem}\tag{TRS}
  \min_{s_k \in \tangent_{x_k}\manifold} \rtrmodel_k(s_k) \quad\text{subject to}\quad \|s_k\| \leq \Delta_k.
\end{align}
The point $x_k$ and radius $\Delta_k$ are then updated depending on how good the
model is, as measured by the ratio
\begin{align*}
  \rtrrho_k = \frac{\mfc(x_k) - \mfc(\retr_{x_k}(s_k))}{\rtrmodel_k(\zeros) - \rtrmodel_k(s_k)}.
\end{align*}
(If the denominator is zero, we let $\rtrrho_k = 1$.)
Specifically, given parameters $\rtrrho' \in \interval[open]{0}{\frac{1}{4}}$ and $\bar\Delta > 0$,
the update rules for the state are
\begin{align*}
  x_{k + 1} = \begin{cases}
    \retr_{x_k}(s_k) &\text{if $\rtrrho_k > \rtrrho'$},\\
    x_k &\text{otherwise},
  \end{cases}
  \qquad\quad
  \Delta_{k + 1} = \begin{cases}
    \frac{1}{4}\Delta_k &\text{if $\rtrrho_k < \frac{1}{4}$},\\
    \min(2\Delta_k, \bar\Delta) &\text{if $\rtrrho_k > \frac{3}{4}$ and $\|s_k\| = \Delta_k$},\\
    \Delta_k &\text{otherwise}.
  \end{cases}
\end{align*}

\paragraph{Shortcomings of trust-region with exact subproblem solver.}\label{par:shortcomings-exact-solver}

There exist algorithms to efficiently solve the subproblem exactly (up to some
accuracy).
Can we provide strong guarantees in the presence of non-isolated minima using an
exact subproblem solver?
Assume for simplicity that $\linearmap_k = \hess \mfc(x_k)$.
We recall~\cite[Thm.~4.1]{nocedal2006numerical} that a vector $s \in
\tangent_{x_k}\manifold$ is a global solution of the
subproblem~\eqref{eq:rtr-subproblem} if and only if $\|s\| \leq \Delta_k$ and there
exists a scalar $\lambda \geq 0$ such that
\begin{align}\label{eq:tr-subproblem-optimality}
  \big(\hess \mfc(x_k) + \lambda I\big)s = -\grad \mfc(x_k), && \lambda(\Delta_k - \|s\|) = 0 && \text{and} && \hess \mfc(x_k) + \lambda I \succeq \zeros.
\end{align}
As mentioned
in~\citep{wojtowytsch2023stochastic,wojtowytsch2024stochastic,liu2022loss}, if
$\mfc$ is convex in a neighborhood of $\optpoint$ then $\optimalset$ is locally
convex, hence affine.
Assuming that $\optimalset$ is not flat around $\optpoint$, it follows that
$\hess \mfc$ must have a negative eigenvalue in any neighborhood of $\optpoint$.
Consider an iterate $x_k$ close to $\optpoint$, and for which $\hess \mfc(x_k)$
has a negative eigenvalue.
Conditions~\eqref{eq:tr-subproblem-optimality} imply $\lambda > 0$ and hence
$\|s\| = \Delta_k$, meaning that $s$ is at the border of the trust region.
Consequently, even if $x_0$ is arbitrarily close to $\optpoint$, we can arrange
for the next iterate to be far away (if the radius $\Delta_0$ is large).
This shows that capture results such as Corollary~\ref{cor:capture-to-single-point}
fail for this algorithm.

As an example, define $\mfc \colon \reals^2 \to \reals$ as $\mfc(x, y) = (x^2 +
y^2 - 1)^2$.
The set of minima is the unit circle, on which $\mfc$ satisfies \mb{}.
Given $t > 0$, define $x(t) = 0$ and $y(t) = 1 - t$.
Consider the subproblem~\eqref{eq:rtr-subproblem} at $( x(t), y(t) )^\top$ with
exact Hessian $\hess \mfc$ and with parameter $\Delta$.
For small enough $t$ there are two solutions:
\begin{align*}
  s = \begin{bmatrix}
    \pm\sqrt{\Delta^2 - \frac{t^2(t - 2)^2}{4(t - 1)^2}}\\
    \frac{t(t - 2)}{2(t - 1)}
  \end{bmatrix}.
\end{align*}
The criterion~\eqref{eq:tr-subproblem-optimality} with $\lambda = 4t(2 - t)$
confirms optimality.
We find that $s \to ( \pm\Delta, 0 )^\top$ as $t \to 0$, so the tentative step
is far even when $t$ is small.
We could arrange for that step to be accepted by adjusting the function value at
the tentative iterate.
That rules out even basic capture-type theorems.
This type of behavior does not happen when the Hessian is positive definite at
the minimum.

\paragraph{Trust-region with Cauchy steps.}\label{par:cauchy-steps}

As just discussed, TR with an exact subproblem solver can fail in the face of
non-isolated minima.
However, practical implementations of TR typically solve the subproblem only
approximately.
We set out to investigate the robustness of such mechanisms to non-isolated
minima.

Our investigation is prompted by the empirical observation that TR with a
popular approximate subproblem solver known as truncated conjugate gradient
(tCG, see~\citep{conn2000trust,absil2007trust}) seems to enjoy superlinear
convergence under \pl{}, even with non-isolated minima.
We confirmed this subsequently~\citep{rebjock2023tcg}
using significant additional machinery.

As a more direct illustration, we show the following theorem, regarding TR with a crude
subproblem solver that computes \emph{Cauchy steps} (see~\eqref{eq:cauchy-steps}
below).
It is relevant in particular because tCG generates a sequence of increasingly
good tentative steps, the first of which is the Cauchy step.

\begin{theorem}\label{th:main-rtr}
  Suppose~\aref{assu:hess-lipschitz-like}, \aref{assu:hess-approx} and
  $\plconstant$-\eqref{eq:local-pl} hold around $\optpoint \in \optimalset$.
  Let $\mathcal{U}$ be a neighborhood of $\optpoint$.
  There exists a neighborhood $\mathcal{V}$ of $\optpoint$ such that if a
  sequence of iterates generated by TR with Cauchy steps enters $\mathcal{V}$
  then the sequence converges linearly to some $x_\infty \in \mathcal{U} \cap
  \optimalset$ with rate $\sqrt{1 - \frac{\plconstant}{\lambda_{\max}}}$, where
  $\lambda_{\max} = \lambda_{\max}(\hess \mfc(x_\infty))$.
\end{theorem}

A local convergence analysis of TR with Cauchy steps is given
in~\citep{muoi2017local} for \emph{non-singular} local minima.
Here, we prove that the favorable convergence properties also hold if we only
assume \pl{}.

To prove Theorem~\ref{th:main-rtr}, we first establish a number of intermediate
results only assuming the subproblem solver satisfies the
properties~\eqref{eq:rtr-better-than-cauchy} and~\eqref{eq:strong-vs} defined
below.
We then secure these properties for Cauchy steps.
First, given a local minimum $\optpoint \in \optimalset$, we assume that the
step $s_k$ satisfies the \emph{sufficient decrease} condition
\begin{align}\label{eq:rtr-better-than-cauchy}
  \rtrmodel_k(0) - \rtrmodel_k(s_k) \geq \rtrsufficientdecrease \|\grad \mfc(x_k)\| \min\!\bigg(\Delta_k, \frac{\|\grad \mfc(x_k)\|^3}{\big|\inner{\grad \mfc(x_k)}{\linearmap_k[\grad \mfc(x_k)]}\big|}\bigg)
\end{align}
whenever the iterate $x_k$ is sufficiently close to $\optpoint$.
(If the denominator is zero, consider the rightmost expression to be infinite.)
This condition holds for many practical subproblem solvers and ensures global
convergence guarantees in particular (see~\citep[\S7.4]{absil2009optimization}
and \cite[\S6.4]{boumal2020introduction}).
Second, given a local minimum $\optpoint \in \optimalset$, we assume that there
exists a constant $\rtrstepconstant \geq 0$ such that
\begin{align}\label{eq:strong-vs}
  \|s_k\| \leq \rtrstepconstant \|\grad \mfc(x_k)\|
\end{align}
when $x_k$ is sufficiently close to $\optpoint$.

We find that the ratios $\sequence{\rtrrho_k}$ are large around minima where
these two conditions hold.
This is because they imply that the trust-region model is an accurate
approximation of the local behavior of $\mfc$.
It follows that the steps $\sequence{s_k}$ decrease $\mfc$ nearly as much as
predicted by the model.

\begin{proposition}\label{prop:ratios-converge}
  Suppose that~\aref{assu:hess-lipschitz-like}
  and~\aref{assu:hess-approx} hold around $\optpoint \in \optimalset$.
  Also assume that the steps $s_k$ satisfy~\eqref{eq:rtr-better-than-cauchy}
  and~\eqref{eq:strong-vs} around $\optpoint$.
  For all $\varepsilon > 0$ there exists a neighborhood $\mathcal{U}$ of
  $\optpoint$ such that if an iterate $x_k$ is in $\mathcal{U}$ then $\rtrrho_k
  \geq 1 - \varepsilon$.
\end{proposition}
\begin{proof}
  We follow and adapt some arguments
  from~\cite[Thm.~7.4.11]{absil2009optimization}, which is stated there assuming
  $\hess \mfc(\optpoint) \succ \zeros$.
  We can dismiss the case where $\grad \mfc(x_k) = \zeros$ because it implies
  $\rtrrho_k = 1$.
  Using the definitions of $\rtrmodel_k$ and $\rtrrho_k$ we have
  \begin{align*}
    1 - \rtrrho_k &= \frac{\mfc(\retr_{x_k}(s_k)) - \rtrmodel_k(s_k)}{\rtrmodel_k(\zeros) - \rtrmodel_k(s_k)}.
  \end{align*}
  Assuming $x_k$ is sufficiently close to $\optpoint$, we bound the numerator as
  $\mfc(\retr_{x_k}(s_k)) - \rtrmodel_k(s_k) \leq
  \frac{\fliplikeconst}{6}\|s_k\|^3 \leq
  \frac{\fliplikeconst\rtrstepconstant^2}{6}\|s_k\|\|\grad \mfc(x_k)\|^2$
  using~\aref{assu:hess-lipschitz-like} and inequality~\eqref{eq:strong-vs}.
  Combining this with the sufficient decrease~\eqref{eq:rtr-better-than-cauchy}
  and~\aref{assu:hess-approx} gives
  \begin{align*}
    1 - \rtrrho_k \leq \frac{\fliplikeconst \rtrstepconstant^2 \|s_k\| \|\grad \mfc(x_k)\|}{6 \rtrsufficientdecrease \min\!\Big(\Delta_k, \frac{\|\grad \mfc(x_k)\|}{\|\hess \mfc(x_k)\| + \hessapproxconst \|\grad \mfc(x_k)\|}\Big)}.
  \end{align*}
  If $\Delta_k$ is active in the denominator then we obtain $1 - \rtrrho_k \leq
  \frac{\fliplikeconst \rtrstepconstant^2}{6\rtrsufficientdecrease} \|\grad
  \mfc(x_k)\|$ because $\|s_k\| \leq \Delta_k$.
  Otherwise, using~\eqref{eq:strong-vs} we obtain $1 - \rtrrho_k \leq
  \frac{\fliplikeconst \rtrstepconstant^3}{6\rtrsufficientdecrease} \|\grad
  \mfc(x_k)\|\big(\|\hess \mfc(x_k)\| + \hessapproxconst \|\grad
  \mfc(x_k)\|\big)$.
  In both cases this yields the result.
\end{proof}

This result notably implies that the trust-region radius does not decrease in
the vicinity of the minimum $\optpoint$.
It means that the trust region eventually becomes inactive when the iterates
converge to $\optpoint$.
We now employ the particular alignment between the gradient and the top
eigenspace of the Hessian induced by \pl{} (see
Lemma~\ref{lemma:grad-image-hess}) to derive bounds on the inner products
$\inner{\grad \mfc(x)}{\hess \mfc(x)[\grad \mfc(x)]}$.

\begin{proposition}\label{prop:inner-product-bounds}
  Suppose that $\mfc$ is~$\plconstant$-\eqref{eq:local-pl} around $\optpoint \in
  \optimalset$.
  Let $\lamflat < \plconstant$ and $\lamsharp > \lammax(\hess \mfc(\optpoint))$.
  Then there exists a neighborhood $\mathcal{U}$ of $\optpoint$ such that for all
  $x \in \mathcal{U}$ we have
  \begin{align*}
    \lamflat \|\grad \mfc(x)\|^2 \leq \inner{\grad \mfc(x)}{\hess \mfc(x)[\grad \mfc(x)]} \leq \lamsharp \|\grad \mfc(x)\|^2.
  \end{align*}
\end{proposition}
\begin{proof}
  By continuity of eigenvalues, there exists a neighborhood $\mathcal{U}$ of
  $\optpoint$ such that for all $x \in \mathcal{U}$ we have $\lammax(\hess
  \mfc(x)) \leq \lamsharp$.
  The upper-bound $\inner{\grad \mfc(x)}{\hess \mfc(x)[\grad \mfc(x)]} \leq
  \lamsharp \|\grad \mfc(x)\|^2$ follows immediately.
  We now prove the lower-bound.
  Let $\rankop$ be the rank of $\hess \mfc(\optpoint)$.
  Given $x$ sufficiently close to $\optpoint$, we let $P(x) \colon
  \tangent_x\manifold \to \tangent_x\manifold$ denote the orthogonal projector
  onto the top $\rankop$ eigenspace of $\hess \mfc(x)$.
  From Lemma~\ref{lemma:grad-image-hess} we have $\|(I - P(x)) \grad \mfc(x)\|^2
  = o(\|\grad \mfc(x)\|^2)$ as $x \to \optpoint$.
  If we write $Q(x) = I - P(x)$, we obtain
  \begin{align*}
    \inner{\grad \mfc(x)}{\hess \mfc(x)[\grad \mfc(x)]} &= \inner{P(x) \grad \mfc(x)}{\hess \mfc(x) P(x) \grad \mfc(x)} + \inner{Q(x) \grad \mfc(x)}{\hess \mfc(x) Q(x) \grad \mfc(x)}\\
                                                        &\geq \lambda_\rankop(\hess \mfc(x)) \|\grad \mfc(x)\|^2 + o(\|\grad \mfc(x)\|^2)
  \end{align*}
  as $x \to \optpoint$.
\end{proof}

Combining these results guarantees a linear rate of convergence.

\begin{proposition}\label{prop:rtr-linear-convergence}
  Suppose that~\aref{assu:hess-lipschitz-like}, \aref{assu:hess-approx} and
  $\plconstant$-\eqref{eq:local-pl} hold around $\optpoint \in \optimalset$.
  Let $\sequence{x_k}$ be a sequence of iterates produced by TR converging to
  $\optpoint$.
  Assume that the steps $s_k$ satisfy~\eqref{eq:rtr-better-than-cauchy}
  and~\eqref{eq:strong-vs} around $\optpoint$.
  Then the iterates converge at least linearly with rate $\sqrt{1 -
    \frac{2\rtrsufficientdecrease\plconstant}{\lammax}}$, where $\lammax =
  \lambda_{\max}(\hess \mfc(\optpoint))$.
\end{proposition}
\begin{proof}
  We can assume that $\grad \mfc(x_k)$ is non-zero for all $k$ (otherwise the
  sequence converges in a finite number of steps).
  We show that the sequence satisfies the sufficient decrease
  property~\eqref{eq:sufficient-decrease}.
  Given $\lamflat < \plconstant$ and $\lamsharp > \lambda_{\max}(\hess
  \mfc(\optpoint))$, Proposition~\ref{prop:inner-product-bounds}
  and~\aref{assu:hess-approx} ensure that
  \begin{align*}
    \frac{1}{\lamsharp} \leq \frac{\|\grad \mfc(x_k)\|^2}{\inner{\grad \mfc(x_k)}{\linearmap_k[\grad \mfc(x_k]}} \leq \frac{1}{\lamflat}
  \end{align*}
  for all large enough $k$.
  We let $0 < \varepsilon < \frac{3}{4}$ and
  Proposition~\ref{prop:ratios-converge} implies that $\rtrrho_k \geq 1 -
  \varepsilon$ for all large enough $k$.
  In particular, the radii $\sequence{\Delta_k}$ are bounded away from zero
  (because the update mechanism does not decrease the radius when $\rtrrho_k
  \geq \frac{1}{4}$).
  Combining the definition of $\rtrrho_k$ and the sufficient
  decrease~\eqref{eq:rtr-better-than-cauchy} gives
  \begin{align*}
    \mfc(x_k) - \mfc(x_{k + 1}) = \rtrrho_k\big(\rtrmodel_k(\zeros) - \rtrmodel_k(s_k)\big) \geq \frac{(1 - \varepsilon)\rtrsufficientdecrease}{\lamsharp}\|\grad \mfc(x_k)\|^2
  \end{align*}
  for all large enough $k$.
  We can now conclude with
  Proposition~\ref{prop:linear-convergence-sufficient-decrease}.
\end{proof}

We are now in a position to prove Theorem~\ref{th:main-rtr}.
The Cauchy step at iteration $x_k$ is defined as the minimum
of~\eqref{eq:rtr-subproblem} with the additional constraint that $s_k \in
\vecspan(\grad \mfc(x_k))$.
We can find an explicit expression for it: when $\grad \mfc(x_k) \neq \zeros$,
the Cauchy step is $s_k^c = -\cauchystep_k \grad \mfc(x_k)$, where
\begin{align}\label{eq:cauchy-steps}
  \cauchystep_k =
  \begin{cases}
    \min\! \Big( \frac{\|\grad \mfc(x_k)\|^2}{\inner{\grad \mfc(x_k)}{\linearmap_k[\grad \mfc(x_k)]}}, \frac{\Delta_k}{\|\grad \mfc(x_k)\|} \Big) \quad&\text{if}\quad \inner{\grad \mfc(x_k)}{\linearmap_k[\grad \mfc(x_k)]} > 0,\\
    \frac{\Delta_k}{\|\grad \mfc(x_k)\|} &\text{otherwise}.
  \end{cases}
\end{align}
Cauchy steps notably satisfy the sufficient decrease
property~\eqref{eq:rtr-better-than-cauchy} globally with $\rtrsufficientdecrease
= \frac{1}{2}$ (see~\citep[Thm.~6.3.1]{conn2000trust}).
We now prove that they also satisfy~\eqref{eq:strong-vs} around minima where
\pl{} holds.

\begin{proposition}\label{prop:rtr-cs-vs}
  Suppose that~\aref{assu:hess-approx} and $\plconstant$-\eqref{eq:local-pl}
  hold around $\optpoint \in \optimalset$.
  Given $\lamflat < \plconstant$, there exists a neighborhood $\mathcal{U}$
  of $\optpoint$ such that if an iterate $x_k$ is in $\mathcal{U}$ then the
  Cauchy step satisfies $\|s_k^c\| \leq \frac{1}{\lamflat}\|\grad \mfc(x_k)\|$.
\end{proposition}
\begin{proof}
  Given $\lamflat < \plconstant$, Proposition~\ref{prop:inner-product-bounds}
  and~\aref{assu:hess-approx} yield that $\inner{\grad
    \mfc(x_k)}{\linearmap_k[\grad \mfc(x_k)]} \geq \lamflat \|\grad
  \mfc(x_k)\|^2$ if $x_k$ is sufficiently close to $\optpoint$.
  It implies that the step-sizes defined in~\eqref{eq:cauchy-steps} are bounded
  as $\cauchystep_k \leq \frac{1}{\lamflat}$, which gives the result.
\end{proof}

In particular, this proposition shows that TR with Cauchy steps
satisfies the~\eqref{eq:vanishing-steps} property at $\optpoint$ with
$\boundvanishingsteps(x) = \frac{\retrdistboundconst}{\lamflat}\|\grad
\mfc(x)\|$, where $\retrdistboundconst$ is as
in~\eqref{eq:retr-distance-condition}.
Furthermore, Cauchy steps satisfy the model decrease
\begin{align*}
  \rtrmodel_k(\zeros) - \rtrmodel_k(s_k^c) \geq \frac{1}{2}\|\grad \mfc(x_k)\|\|s_k^c\|,
\end{align*}
as shown in~\cite[Lem.~4.3]{absil2005convergence}.
It implies that TR with Cauchy steps generates sequences that satisfy the
strong decrease property~\eqref{eq:sufficient-decrease-lyapunov} with
$\lyapsufficientdecreaseconst = \frac{\rtrrho'}{2\retrdistboundconst}$, where
$\retrdistboundconst$ is as in~\eqref{eq:retr-distance-condition}, and
$\rtrrho'$ is defined in the algorithm description in Section~\ref{subsec:rtr}.
See~\cite[Thm.~4.4]{absil2005convergence} for details on this.
As a consequence, TR with Cauchy steps satisfies the bounded path length
property~\eqref{eq:bounded-path-length-prop} at points where a \loja{}
inequality holds (Lemma~\ref{lemma:bound-discrete-gd-path-length}).
Moreover, if the iterates of this algorithm stay in a compact region then they
accumulate only at critical points~\cite[Thm.~7.4.4]{absil2009optimization}.
We can finally combine the statements from
Corollary~\ref{cor:capture-to-single-point} and
Proposition~\ref{prop:rtr-linear-convergence} to obtain
Theorem~\ref{th:main-rtr}.

\begin{remark}\label{remark:more-than-cauchy-point}
  The model decrease~\eqref{eq:rtr-better-than-cauchy} is not a
  sufficient condition for the strong decrease
  property~\eqref{eq:sufficient-decrease-lyapunov} to hold.
  As a result, it is not straightforward to determine whether the bounded path
  length property~\eqref{eq:bounded-path-length-prop} holds for a given
  subproblem solver: see~\cite[\S4.2]{absil2005convergence} for a
  discussion of this.
\end{remark}

\section{Conclusions and perspectives}\label{sec:conclusion}

We showed the (local) equivalence (up to arbitrarily small losses in
constants) of \mb{}, \pl{}, \eb{} and \qg{}
(Section~\ref{sec:equiv-properties}).
We then revisited classical capture results compatible with non-isolated minima
to factor out the roles of vanishing step-sizes and bounded path lengths
(Section~\ref{sec:reminders}).
The \mb{} property and the alignment of the gradient with respect to the Hessian
eigenspaces (Lemma~\ref{lemma:grad-image-hess}) are particularly adapted to
analyze second-order algorithms.
Accordingly, assuming the above conditions we establish quadratic convergence
for ARC with inexact subproblem solvers and linear convergence for TR with Cauchy steps
(Section~\ref{sec:super-linear-conv}).

We conclude with a few research directions:
\begin{itemize}
\item In Section~\ref{par:cauchy-steps} we analyze a simple subproblem solver
  (Cauchy steps) for TR.
  It is natural to explore more advanced subproblem solvers.
  For example in~\citep{rebjock2023tcg} we show superlinear convergence for a
  truncated CG method assuming \mb{}.
\item In Section~\ref{par:shortcomings-exact-solver} we argue that TR with an
  \emph{exact} subproblem solver cannot satisfy a standard capture property
  in the presence of non-isolated minima.
  However, establishing capture is only a means to an end.
  It may still be possible to obtain other satisfactory guarantees.
\item More generally, using the tools from Sections~\ref{sec:equiv-properties}
  and~\ref{sec:reminders}, there is an opportunity to revisit analyses of other
  algorithms that currently require strong local convexity.
  For example, \cite{goyens2024riemannian} control the global complexity of
  hybrid TR algorithms for strict saddle functions, and currently require
  non-singular minima.
\item Likewise, Remark~\ref{rem:RSI} illustrates equivalence of \mb{} with the restricted secant inequality (RSI). Showing \mb{} implies RSI is direct. The converse is facilitated via equivalence of \mb{} with \eb{}. There may be other local properties in the literature that turn out to be equivalent to these.
\end{itemize}

\section*{Acknowledgments}

We thank Dmitriy Drusvyatskiy, J\'er\^ome Bolte and Christopher Criscitiello for
insightful discussions and pointers to literature,
and Brighton Ancelin for helpful feedback on an earlier version.

\bibliographystyle{plainnat}
\bibliography{references}

\clearpage
\appendix
\section{\loja{} and function growth}\label{sec:loja-proofs}

We review here the classical arguments at the basis of
Proposition~\ref{prop:pl-implies-qg}.
Given an initial point $x_0 \in \manifold$, we let $x \colon I \to \manifold$
denote a solution of the negative gradient flow
\begin{align}\label{eq:gradient-flow}\tag{GF}
    x'(t) = -\grad \mfc(x(t)) \qquad\text{with}\qquad x(0) = x_0
\end{align}
on the maximum interval $I$.
The following is classical~\citep{lojasiewicz1982trajectoires}, restated
succinctly to highlight neighborhood assumptions.

\begin{lemma}\label{lemma:bound-gd-flow-path-length}
  Suppose $\mfc$ satisfies~\eqref{eq:local-loja} with constants $\plexp \in
  \interval[open right]{0}{1}$ and $\plconstant > 0$ in a neighborhood
  $\mathcal{U}$ of $\optpoint \in \optimalset$.
  Also suppose that $\mfc(x) > \mfcopt$ for all $x \in \mathcal{U} \setminus
  \optimalset$.
  Let $x$ be a solution
  to~\eqref{eq:gradient-flow} for some $x_0 \in \mathcal{U} \setminus
  \optimalset$.
  Suppose that for all $t \in \interval[open]{0}{T}$ we have $x(t) \in
  \mathcal{U}$ and $\grad \mfc(x(t)) \neq \zeros$.
  Then the path length is bounded as
  \begin{align*}
    \int_0^T \|x'(t)\|\deriv t &\leq \frac{1}{(1 - \plexp)\sqrt{2\plconstant}}\vert \mfc(x_0) - \mfcopt \vert^{1 - \plexp}.
  \end{align*}
\end{lemma}
\begin{proof}
  Define $h(t) = \frac{1}{(1 - \plexp)\sqrt{2\plconstant}}\vert \mfc(x(t)) -
  \mfcopt \vert^{1 - \plexp}$.
  For all $t \in \interval[open right]{0}{T}$, the \loja{} inequality provides:
  \begin{align*}
    \|x'(t)\| = \|\grad \mfc(x(t))\| = \frac{\|\grad \mfc(x(t))\|^2}{\|\grad \mfc(x(t))\|} \leq \frac{\|\grad \mfc(x(t))\|^2}{\sqrt{2\plconstant}(\mfc(x(t)) - \mfcopt)^\plexp} = -h'(t).
  \end{align*}
  It follows that $\int_0^T \|x'(t)\| \deriv t \leq \int_0^T -h'(t) \deriv t =
  h(0) - h(T) \leq h(0)$.
\end{proof}

As shown below, this bound implies that the trajectories are trapped and have a
limit point if $x_0$ is close enough to $\optimalset$.

\begin{proposition}[Lyapunov stability]\label{prop:lyapunov-gd-flow}
  Suppose that $\mfc$ satisfies~\eqref{eq:local-loja} around $\optpoint \in
  \optimalset$ and let $\mathcal{U}$ be a neighborhood of $\optpoint$.
  There exists a neighborhood $\mathcal{V}$ of $\optpoint$ such that if $x_0 \in
  \mathcal{V}$ then the solution $x$ to~\eqref{eq:gradient-flow} is defined on
  $\interval[open right]{0}{+\infty}$, and $x(t) \in \mathcal{U}$ for all $t \geq
  0$.
\end{proposition}
\begin{proof}
  The set $\mathcal{U}$ contains a ball $\mathcal{B}$ centered on $\optpoint$ of
  radius $\delta_1$ such that \emph{(i)} $\mfc$ satisfies~\eqref{eq:local-loja}
  with constants $\plexp$ and $\plconstant$ in $\mathcal{B}$, and \emph{(ii)}
  $\mfc(y) > \mfcopt$ for all $y \in \mathcal{B} \setminus \optimalset$.
  By continuity of $\mfc$ there exists an open ball $\mathcal{V} \subseteq
  \mathcal{B}$ of radius $\delta_2$ around $\optpoint$ such that for all $y \in
  \mathcal{V}$ we have
  \begin{align*}
    \frac{1}{(1 - \plexp)\sqrt{2\plconstant}}\vert \mfc(y) - \mfcopt \vert^{1 - \plexp} + \delta_2 < \delta_1.
  \end{align*}
  Given $x_0 \in \mathcal{V}$, let $x \colon I \to \manifold$ be the maximal
  solution to~\eqref{eq:gradient-flow}.
  Suppose that $\{t \in I : t \geq 0 \text{ and } x(t) \notin \mathcal{B}\}$ is
  non-empty and let $T$ be the infimum of this set.
  Then $x(t) \in \mathcal{B}$ for all $0 \leq t < T$.
  Suppose first that there exists $t \in \interval[open right]{0}{T}$ such that
  $\grad \mfc(x(t)) = \zeros$.
  Then $x(t') = x(t)$ for all $t' > t$, which is impossible.
  So $\grad \mfc(x(t)) \neq \zeros$ for all $t \in \interval[open right]{0}{T}$.
  It follows that the assumptions of Lemma~\ref{lemma:bound-gd-flow-path-length}
  are satisfied, and the path length is bounded as
  \begin{align*}
    \int_0^T \|x'(t)\|\deriv t \leq \frac{1}{(1 - \plexp)\sqrt{2\plconstant}}\vert \mfc(x_0) - \mfcopt \vert^{1 - \plexp}.
  \end{align*}
  This implies that
  \begin{align*}
    \dist(x(T), \optpoint) \leq \int_0^T \|x'(t)\|\deriv t + \dist(x_0, \optpoint) < \delta_1,
  \end{align*}
  and hence that $x(T) \in \mathcal{B}$.
  This is a contradiction and we deduce that $x(t) \in \mathcal{B}$ for all $t
  \in I$.
  Therefore, the total path length of $x$ is bounded, and the escape
  lemma~\cite[Lem.~A.43]{lee2018introduction} implies that $x$ is defined for
  all $t \geq 0$.
\end{proof}

\begin{corollary}\label{cor:gf-converges}
  Suppose that $\mfc$ satisfies~\eqref{eq:local-loja} around $\optpoint \in
  \optimalset$.
  There exists a neighborhood $\mathcal{V}$ of $\optpoint$ such that for all
  $x_0 \in \mathcal{V}$ the solution to~\eqref{eq:gradient-flow} is defined on
  $\interval[open right]{0}{+\infty}$ and has a limit in $\optimalset$.
\end{corollary}
\begin{proof}
  \TODOF{Here we use the fact that manifolds are locally compact.}
  Let $\mathcal{U}$ be a compact neighborhood of $\optpoint$ such that $\mfc$
  satisfies~\eqref{eq:local-loja} around all points in $\mathcal{U} \cap
  \optimalset$, and such that all critical points of $\mfc$ in $\mathcal{U}$ are
  in $\optimalset$.
  Let $\mathcal{V}$ be a neighborhood of $\optpoint$ associated to $\mathcal{U}$ as in
  Proposition~\ref{prop:lyapunov-gd-flow}.
  Let $x \colon \interval[open right]{0}{+\infty} \to \manifold$ denote a
  solution of~\eqref{eq:gradient-flow} starting from $x_0 \in \mathcal{V}$.
  Then $x(t) \in \mathcal{U}$ for all $t \geq 0$.
  The set $\mathcal{U}$ is compact so $x$ has an accumulation point $x_\infty
  \in \mathcal{U}$.
  This is a critical point for $\mfc$ so $x_\infty \in \optimalset$.
  It is also Lyapunov stable because $\mfc$ satisfies~\eqref{eq:local-loja}
  around $x_\infty$ (Proposition~\ref{prop:lyapunov-gd-flow}).
  We deduce that $\lim x(t) = x_\infty$.
  \TODOF{Check that accumulation points of gradient flow are critical.
    Section 9.3 of Hirsch Smale Devaney - Differential Equations Dynamical
    Systems and Chaos}
\end{proof}

From this we deduce that the \loja{} inequality implies the local growth of
$\mfc$ announced in Proposition~\ref{prop:pl-implies-qg}.

\begin{proof}[Proof of Proposition~\ref{prop:pl-implies-qg}]
  Let $\mathcal{U}$ be a neighborhood of $\optpoint$ where~\eqref{eq:local-loja}
  holds with constants $\plconstant$ and $\plexp$.
  Proposition~\ref{prop:lyapunov-gd-flow} and Corollary~\ref{cor:gf-converges} give a neighborhood $\mathcal{V}$
  of $\optpoint$ such that for all $x_0 \in \mathcal{V}$ the solution
  to~\eqref{eq:gradient-flow} is defined on $\interval[open right]{0}{\infty}$,
  stays in $\mathcal{U}$ at all times, and has a limit
  $x_\infty \in \optimalset$.
  Then, Lemma~\ref{lemma:bound-gd-flow-path-length} provides the first
  inequality in
  \begin{align*}
    \frac{1}{(1 - \plexp)\sqrt{2\plconstant}}\vert \mfc(x_0) - \mfcopt \vert^{1 - \plexp} \geq \int_0^\infty \|x'(t)\|\deriv t \geq \dist(x_0, x_\infty) \geq \dist(x_0, \optimalset),
  \end{align*}
  which concludes the proof.
\end{proof}

\section{Other \loja{} exponents}\label{sec:other-loja-exponents}

We prove here the statements of Remark~\ref{remark:other-loja-exponents}.
From Proposition~\ref{prop:pl-implies-qg}, there exists $c > 0$ such that
\begin{align}\label{eq:function-growth}
  \mfc(x) - \mfcopt \geq c \dist(x, \optimalset)^{\frac{1}{1 - \plexp}}
\end{align}
for all $x$ sufficiently close to $\optpoint$.

Assume first that $\mfc$ is $\smooth{1}$ and $\grad \mfc$ is $L$-Lipschitz
continuous around $\optpoint$.
Then $\mfc(x) - \mfcopt \leq \frac{L}{2}\dist(x, \optimalset)^2$ for all $x$
sufficiently close to $\optpoint$ (see for
example~\cite[Cor.~10.54]{boumal2020introduction}).
When $\plexp < \frac{1}{2}$, this inequality is incompatible
with~\eqref{eq:function-growth} if $\mfc$ is non-constant around $\optpoint$.

Now assume that $\mfc$ is $\smooth{2}$ and $\hess \mfc$ is $L$-Lipschitz
continuous around $\optpoint$.
Define $h \colon (y, v) \mapsto \inner{v}{\hess \mfc(y)[v]}$.
Lipschitz continuity of $\hess \mfc$ gives
that~\citep[Cor.~10.56]{boumal2020introduction}
\begin{align*}
  \mfc(\Exp_y(tv)) - \mfcopt - \frac{t^2}{2}h(y, v) \leq \frac{L}{6}t^3
\end{align*}
for all $y \in \optimalset$ close enough to $\optpoint$, all unitary $v \in
\normal_y\optimalset$ and $t > 0$ small enough.
Take $t \to 0$ and invoke Lemma~\ref{lemma:dist-equiv} to see that $h(y, v)$ must be positive for
this inequality to be compatible with the function
growth~\eqref{eq:function-growth} when $\plexp \in \interval[open
right]{\frac{1}{2}}{\frac{2}{3}}$.
We conclude that there is a compact neighborhood $\mathcal{V}$ of $\optpoint$
such that $h(y, v) > 0$ for all $y \in \mathcal{V} \cap \optimalset$ and $v \in
\normal_y\optimalset$ unitary.
The function $h$ is continuous and the set $\mathcal{D} = \{(y, v) : y \in
\mathcal{V} \cap \optimalset, v \in \normal_y\optimalset \text{ unitary}\}$ is
compact so $\plconstant = \inf_{(y, v) \in \mathcal{D}} h(y, v)$ is positive.
Now let $\mathcal{W}$ be a neighborhood of $\optpoint$ as in
Lemma~\ref{lemma:log-well-defined} such that for all $x \in \mathcal{W}$ the
projection $\proj_\optimalset(x)$ is included in $\mathcal{V} \cap \optimalset$.
Then for all $x \in \mathcal{W}$ we have
\begin{align*}
  \mfc(x) - \mfcopt = h(y, v) + o(\|v\|^2) \geq \plconstant \dist(x, \optimalset)^2 + o(\dist(x, \optimalset)^2),
\end{align*}
where $y \in \proj_\optimalset(x)$ and $v = \Log_y(x)$.
We conclude that \qg{} holds around $\optpoint$, which implies \pl{}, as shown
in Section~\ref{subsec:qg-implies-eb}.

Now assume that $\mfc$ is $\smooth{3}$.
The argument is similar in this case.
Let $y \in \optimalset$ and $v \in \normal_y\optimalset$ unitary and assume that
$h(y, v) = 0$.
Then a Taylor expansion gives
\begin{align*}
  \mfc(\Exp_x(tv)) - \mfcopt = o(t^3)
\end{align*}
because the third-order term vanishes.
This is incompatible with the function growth~\eqref{eq:function-growth} when
$y$ is close to $\optpoint$ and $\plexp \in
\interval{\frac{1}{2}}{\frac{2}{3}}$.
So $h(y, v)$ must be positive and we conclude with the same arguments as above.

\section{\loja{} and bounded path length}\label{sec:loja-discrete-proofs}

\begin{proof}[Proof of Lemma~\ref{lemma:bound-discrete-gd-path-length}]
  Let $\mathcal{U}$ be an open neighborhood of $\optpoint$ such that \emph{(i)}
  $\mfc$ satisfies~\eqref{eq:local-loja} with constants $\plexp \in
  \interval[open right]{0}{1}$ and $\plconstant > 0$ in $\mathcal{U}$,
  \emph{(ii)} $\mfc(x) \geq \mfcopt$ for all $x \in \mathcal{U}$ and
  \emph{(iii)} condition~\eqref{eq:sufficient-decrease-lyapunov} holds in
  $\mathcal{U}$.
  For all $x \in \mathcal{U}$ we have
  \begin{align*}
    0 \leq |\mfc(x) - \mfcopt|^{2\plexp} \leq \frac{1}{2\plconstant}\|\grad \mfc(x)\|^2, \qquad\text{and so}\qquad \|\grad \mfc(x)\| \geq \sqrt{2\plconstant}|\mfc(x) - \mfcopt|^{\plexp}.
  \end{align*}
  Let $x_L, \dots, x_K$ be consecutive iterates in $\mathcal{U}$.
  For such an iterate $x_k$, either $\grad \mfc(x_k) = \zeros$ and $\dist(x_k,
  x_{k + 1}) = 0$, or $\grad \mfc(x_k) \neq \zeros$ and $\mfc(x_k) > \mfcopt$.
  In this second case, combining the lower-bound on $\|\grad \mfc(x)\|$ above
  with the strong decrease condition in~\eqref{eq:sufficient-decrease-lyapunov},
  we find that
  \begin{align*}
    \dist(x_k, x_{k + 1}) \leq \frac{1}{\lyapsufficientdecreaseconst\sqrt{2\plconstant}}\frac{\mfc(x_k) - \mfc(x_{k + 1})}{|\mfc(x_k) - \mfcopt|^{\plexp}} \leq \frac{1}{\lyapsufficientdecreaseconst (1 - \plexp) \sqrt{2\plconstant}} \Big(\big(\mfc(x_k) - \mfcopt\big)^{1 - \plexp} - \big(\mfc(x_{k + 1}) - \mfcopt\big)^{1 - \plexp}\Big),
  \end{align*}
  where the second inequality comes from
  \begin{align*}
    \frac{\mfc(x_k) - \mfc(x_{k + 1})}{|\mfc(x_k) - \mfcopt|^{\plexp}} &= \int_{\mfc(x_{k + 1})}^{\mfc(x_k)} \frac{1}{|\mfc(x_k) - \mfcopt|^{\plexp}}\deriv t\\
                                                                       &\leq \int_{\mfc(x_{k + 1})}^{\mfc(x_k)} \frac{1}{|t - \mfcopt|^{\plexp}} \deriv t\\
                                                                       &= \frac{1}{1 - \plexp}\Big(\big(\mfc(x_k) - \mfcopt\big)^{1 - \plexp} - \big(\mfc(x_{k + 1}) - \mfcopt\big)^{1 - \plexp}\Big).
  \end{align*}
  Summing the bound on $\dist(x_k, x_{k + 1})$ over $k = L, \dots, K - 1$ gives
  the~\eqref{eq:bounded-path-length-prop} property.
\end{proof}

\section{\loja{} and linear convergence rate}\label{sec:loja-conv-rate}

\begin{proof}[Proof of Proposition~\ref{prop:linear-convergence-sufficient-decrease}]
  Let $\mathcal{U}$ be a neighborhood of $\bar x$ where \pl{} holds.
  The sufficient decrease~\eqref{eq:sufficient-decrease} and~\eqref{eq:local-pl}
  give
  \begin{align*}
    \mfc(x_{k + 1}) - \mfc(x_k) \leq -\sufficientdecreaseconst\|\grad \mfc(x_k)\|^2 \leq -2\plconstant\sufficientdecreaseconst\big(\mfc(x_k) - \mfcopt\big)
  \end{align*}
  for large enough $k$.
  Adding $\mfc(x_k) - \mfcopt$ on both sides yields
  \begin{align*}
    \mfc(x_{k + 1}) - \mfcopt &\leq (1 - 2\sufficientdecreaseconst \plconstant)\big(\mfc(x_k) - \mfcopt\big),
  \end{align*}
  showing the linear rate for function values.
  We now prove the rate for $\sequence{\|\grad \mfc(x_k)\|}$ and
  $\sequence{\dist(x_k, \optimalset)}$.
  The sufficient decrease~\eqref{eq:sufficient-decrease} gives that for large
  enough $k$ we have
  \begin{align*}
    \sufficientdecreaseconst \|\grad \mfc(x_k)\|^2 \leq \mfc(x_k) - \mfc(x_{k + 1}) \leq \mfc(x_k) - \mfcopt, \qquad\text{and so}\qquad \|\grad \mfc(x_k)\| \leq \sqrt{\frac{1}{\sufficientdecreaseconst}\big(\mfc(x_k) - \mfcopt\big)}.
  \end{align*}
  This shows that $\sequence{\|\grad \mfc(x_k)\|}$ converges linearly to zero
  with rate $\sqrt{1 - 2 \sufficientdecreaseconst \plconstant}$.
  For $\sequence{\dist(x_k, \optimalset)}$ we use the local quadratic growth
  (Proposition~\ref{prop:pl-implies-qg}).
  For all $x$ sufficiently close to $\bar x$ we have
  \begin{align*}
    \mfc(x) - \mfcopt \geq \frac{\plconstant}{2} \dist(x, \optimalset)^2, \qquad\text{and so}\qquad \dist(x_k, \optimalset) \leq \sqrt{\frac{2}{\plconstant}\big(\mfc(x_k) - \mfcopt\big)}.
  \end{align*}
  We conclude that $\dist(x_k, \optimalset)$ converges linearly to zero with
  rate $\sqrt{1 - 2 \sufficientdecreaseconst \plconstant}$.
\end{proof}

\begin{proof}[Proof of Proposition~\ref{prop:sufficient-decrease-gd}]
  When $x_k$ is sufficiently close to $\optpoint$, the
  Lipschitz-type property~\eqref{eq:lipschitz-like} implies
  \begin{align}\label{eq:constant-gd-decrease}
    \mfc(x_k) - \mfc(x_{k + 1}) = \mfc(x_k) - \mfc(\retr_x(s_k)) \geq -\Big(\inner{\grad \mfc(x_k)}{s_k} + \frac{\liplikegd}{2}\|s_k\|^2\Big).%
  \end{align}
  Plugging $s_k = -\gamma \grad \mfc(x_k)$ yields $\mfc(x_k) - \mfc(x_{k + 1})
  \geq \big(1 - \frac{\liplikegd}{2}\gamma\big)\|\grad \mfc(x_k)\|\|s_k\|$.
  The number $1 - \frac{\liplikegd}{2}\gamma$ is positive because $\gamma \in
  \interval[open]{0}{\frac{2}{\liplikegd}}$.
  So this algorithm satisfies the strong decrease
  property~\eqref{eq:sufficient-decrease-lyapunov} around $\optpoint$ with
  constant $\lyapsufficientdecreaseconst = \frac{1}{\retrdistboundconst}(1 -
  \frac{\liplikegd}{2}\gamma)$, where $\retrdistboundconst$ is as
  in~\eqref{eq:retr-distance-condition}.
  It further satisfies~\eqref{eq:bounded-path-length-prop} by
  Lemma~\ref{lemma:bound-discrete-gd-path-length}, so
  Corollary~\ref{cor:capture-to-single-point} ensures capture of the iterates.
  Equation~\eqref{eq:constant-gd-decrease} additionally gives $\mfc(x_k) -
  \mfc(x_{k + 1}) \geq \big(\gamma - \frac{\liplikegd}{2}\gamma^2\big)\|\grad
  \mfc(x_k)\|^2$.
  It follows that the iterates satisfy the sufficient
  decrease~\eqref{eq:sufficient-decrease} with $\sufficientdecreaseconst =
  \gamma - \frac{\liplikegd}{2}\gamma^2$.
  We obtain the linear convergence rate with
  Proposition~\ref{prop:linear-convergence-sufficient-decrease}.
\end{proof}

Using performance estimation, it is possible to derive sharper convergence rates
for gradient descent with constant step-sizes under the \pl{}
assumption~\cite[Thm.~3]{abbaszadehpeivasti2023conditions}.
The Lipschitz-type property~\eqref{eq:lipschitz-like} also implies the
sufficient decrease~\eqref{eq:sufficient-decrease} for gradient descent with
backtracking line-search.

\begin{proposition}\label{prop:sufficient-decrease-gd-linesearch}
  Suppose that $\mfc$ and $\retr$ satisfy~\eqref{eq:lipschitz-like} around
  $\optpoint \in \optimalset$.
  Let $\sequence{x_k}$ be a sequence of iterates generated by gradient descent
  with Armijo backtracking line-search converging to $\optpoint$.
  Then $\sequence{x_k}$ satisfies~\eqref{eq:sufficient-decrease} with
  \begin{align*}
    \sufficientdecreaseconst = \linesearchsigma \min\Big( \bar \alpha, \frac{2\beta(1 - \linesearchsigma)}{\liplikegd} \Big).
  \end{align*}
\end{proposition}
\begin{proof}%
  See~\cite[Lem.~4.12]{boumal2020introduction} for the Riemannian case.
\end{proof}

See also~\citep{khanh2022inexact} for more results on line-search with \loja{}
inequalities.

\end{document}